\documentclass[a4paper, 10pt]{amsart}

\usepackage{amsmath,amssymb,amscd}

\usepackage{amsfonts}
\usepackage[english]{babel}
\usepackage[leqno]{amsmath}
\usepackage{amssymb,amsthm}
\usepackage{amscd}
\usepackage{enumerate}
\usepackage{epsfig}

\usepackage{layout}
\usepackage{fullpage}

\usepackage[usenames,dvipsnames]{color}

\usepackage[backref=page]{hyperref}

\hypersetup{
 colorlinks,
 citecolor=green,
 linkcolor=blue,
 urlcolor=Blue}

\usepackage[matrix,arrow,tips,curve]{xy}

\input{xy}
\xyoption{all}

\newenvironment{sis}{\left\{\begin{aligned}}{\end{aligned}\right.}

\newtheorem{thm}{Theorem}[section]
\newtheorem{lemma}[thm]{Lemma}
\newtheorem{lemdef}[thm]{Lemma-Definition}
\newtheorem{prop}[thm]{Proposition}
\newtheorem{cor}[thm]{Corollary}

\newtheorem{fact}[thm]{Fact}

\newtheorem*{thmA}{Theorem A}
\newtheorem*{thmB}{Theorem B}
\newtheorem*{thmC}{Theorem C}
\newtheorem*{thmD}{Theorem D}

\numberwithin{equation}{section}
\theoremstyle{definition}
\newtheorem{defi}[thm]{Definition}
\newtheorem{convention}[thm]{}

\theoremstyle{remark}
\newtheorem{remark}[thm]{Remark}

\newtheorem{example}[thm]{Example}

\newcommand{\Z}{\mathbb{Z}}
\newcommand{\Q}{\mathbb{Q}}
\newcommand{\R}{\mathbb{R}}

\newcommand{\Pic}{\operatorname{Pic}}
\newcommand{\Hilb}{\operatorname{Hilb}}
\newcommand{\un}{\underline}
\newcommand{\ov}{\overline}
\newcommand{\wt}{\widetilde}
\newcommand{\wh}{\widehat}

\DeclareMathOperator{\Hom}{{Hom}}
\DeclareMathOperator{\Ext}{{Ext}}
\DeclareMathOperator{\Spec}{Spec \:}
\DeclareMathOperator{\Spf}{Spf \:}
\DeclareMathOperator{\Def}{Def}

\def \Im{{\rm Im}}
\DeclareMathOperator{\Supp}{Supp}

\DeclareMathOperator{\id}{id}

\def \PP{\mathbb{P}}

\def \Gm{{\mathbb G}_m}
\def \Ga{{\mathbb G}_a}
\def \GL{{\rm GL}}

\def\J{\overline J}

\def \P{\mathcal P}
\def \F{\mathcal F}

\def \X{\mathcal X}
\def \Y{\mathcal Y}
\def\I{\mathcal I}

\def \E{\mathcal E}
\def\O{\mathcal O}
\def \A{\mathcal A}

\def \SS{\mathcal S}

\def\M0{\mathcal M^0}

\def\bJ{{\mathbb J}}
\def\bJbar{\ov{{\mathbb J}}}

\def \M{\mathcal M}

\def\m {\mathfrak{m}}

\DeclareMathOperator{\Arr}{\mathcal{A}_X}
\DeclareMathOperator{\Pol}{\mathcal{P}_X}
\DeclareMathOperator{\V}{\mathcal{V}}

\newcommand{\bbN}{{\mathbb N}}
\newcommand{\bbC}{{\mathbb C}}
\newcommand{\bbQ}{{\mathbb Q}}
\newcommand{\bbZ}{{\mathbb Z}}

\newcommand{\cZ}{{\mathcal Z}}

\newcommand{\rk}{\operatorname{rk}}

\newcommand{\Hilbs}{\operatorname{{}^s Hilb}}

\title{Fine compactified Jacobians of reduced curves
}

\author{Margarida Melo, Antonio Rapagnetta, Filippo Viviani}

\begin{document}



\begin{abstract}

To every singular reduced projective curve $X$ one can associate many fine compactified Jacobians, depending on the choice of a polarization on $X$, each of which yields a modular compactification of a disjoint union of a finite
number of copies of the generalized Jacobian of $X$. We investigate the geometric properties of fine compactified Jacobians focusing on curves having locally planar singularities. We give examples of nodal curves admitting non isomorphic (and even non homeomorphic  over the field of complex numbers)  fine compactified Jacobians. We study universal fine compactified Jacobians, which are relative fine compactified Jacobians over the semiuniversal deformation space of the curve $X$. Finally, we investigate the existence of twisted Abel maps with values in suitable fine compactified Jacobians.

\end{abstract}

\maketitle

\tableofcontents

\section{Introduction}

\subsection*{Aim and motivation}

The aim of this paper is to study fine compactified Jacobians of a \emph{reduced} projective connected curve $X$ over an algebraically closed field $k$ (of arbitrary characteristic),
with special emphasis in the case where $X$ has locally planar singularities.

Recall that given such a curve $X$,  the {\em generalized Jacobian} $J(X)$ of $X$, defined to be the connected component of the Picard variety of $X$ containing the identity, parametrizes line bundles on $X$ that have multidegree zero, i.e. degree zero on each irreducible component of $X$. It turns out that $J(X)$ is a smooth irreducible algebraic group of dimension equal to the arithmetic genus $p_a(X)$ of $X$.
 However, if $X$ is a singular curve, the generalized Jacobian $J(X)$ is rarely complete. The problem of compactifying it, i.e. of constructing a projective variety (called a compactified Jacobian) containing
 $J(X)$ as an open subset,  is very natural and it has attracted the attention of many mathematicians, starting from the pioneering work of Mayer-Mumford and of  Igusa in the 50's, until the more recent works of D'Souza, Oda-Seshadri, Altmann-Kleiman, Caporaso, Pandharipande, Simpson, Jarvis, Esteves, etc... (we refer to the introduction of \cite{est1} for an  account of the different approaches).

In each of the above constructions,  compactified Jacobians parametrize (equivalence classes of) certain rank-$1$, torsion free sheaves on $X$ that are assumed to be semistable with respect to  a certain polarization. If the polarization is general (see below for the precise meaning of general), then all the semistable sheaves will also be stable. In this case, the associated compactified Jacobians will carry a universal sheaf and therefore we will speak of {\em fine compactified Jacobians} (see \cite{est1}).

\vspace{0.2cm}

The main motivation of this work, and of its sequels \cite{MRV1} and \cite{MRV2}, comes from the
Hitchin fibration for the moduli space of Higgs vector bundles on a fixed smooth and projective curve $C$ (see \cite{Hit}, \cite{Nit}), whose fibers are compactified Jacobians of certain singular covers of $C$, called spectral curves (see \cite{BNR}, \cite{Sch} and the Appendix of \cite{MRV1}).
The spectral curves  have always locally planar singularities (since they are contained in a smooth surface by construction), although they are not necessarily reduced nor irreducible. It is worth noticing that, in the case of reduced but not irreducible spectral curves, the compactified Jacobians appearing as fibers of the Hitchin fibration turn out to be fine compactified Jacobians under the assumption that the degree $d$ and the rank $r$ of the Higgs bundles are coprime. However, in the general case, Chaudouard-Laumon in their work \cite{CL1} and \cite{CL2} on the weighted fundamental Lemma (where they generalize the work of Ngo  on the fundamental Lemma, see \cite{Ngo1} and \cite{Ngo2}) have introduced a modified Hitchin fibration for which all fibers are fine compactified Jacobians.

According to Donagi-Pantev \cite{DP}, the conjectural geometric Langlands correspondence should induce, by passing to the semiclassical limit and taking into account that the general linear group $\GL_r$ is equal to its Langlands dual group, an autoequivalence of the derived category of the moduli space of Higgs bundles, which should intertwine the action of the classical limit tensorization functors with the action of the classical limit Hecke functors (see \cite[Conj. 2.5]{DP} for a precise formulation). In particular, such an autoequivalence should preserve the Hitchin fibration, thus inducing  fiberwise an autoequivalence of the compactified Jacobians of the spectral curves.  This conjecture is verified in loc. cit. over the open locus of smooth spectral curves, where the desired fiberwise autoequivalence reduces to the classical Fourier-Mukai autoequivalence for Jacobians of smooth curves, established by Mukai in \cite{mukai}. This autoequivalence was extended by D. Arinkin to compactified Jacobians of integral spectral curves in \cite{arin1} and \cite{arin2}. In the two sequels \cite{MRV1} and \cite{MRV2} to this work, which are strongly based on the present manuscript, we will extend the Fourier-Mukai autoequivalence
to any fine compactified Jacobian of a reduced curve with locally planar singularities.

\vspace{0,1cm}

\subsection*{Our results}

In order to state our main results, we need to review the definition of fine compactified Jacobians of a reduced curve $X$, following the approach of Esteves \cite{est1}
(referring the reader to \S\ref{S:comp-Jac}  for more details).

The starting point is a result of Altman-Kleiman \cite{AK} who showed that there is a scheme $\bJbar_X$, locally of finite type over $k$, parametrizing simple, rank-$1$, torsion-free sheaves on $X$, which, moreover, satisfies the existence part of the valuative criterion of properness (see Fact \ref{F:huge-Jac}). Clearly, $\bJbar_X$ admits a decomposition into a disjoint union $\bJbar_X=\coprod_{\chi\in \Z} \bJbar_X^{\chi}$, where $\bJbar_X^{\chi}$ is the open and closed subset of $\bJbar_X$ parametrizing sheaves $I$ of  Euler-Poincar\'e characteristic $\chi(I):=h^0(X,I)-h^1(X,I)$ equal to $\chi$.
As soon as $X$ is not irreducible, $\bJbar_X^{\chi}$ is not separated nor of finite type over $k$.
Esteves \cite{est1} showed that each $\bJbar_X^{\chi}$ can be covered by open and projective subschemes, the fine compactified Jacobians of $X$, depending on the choice of a generic polarization in the following way.

A \emph{polarization} on $X$ is a collection of rational numbers $\un q=\{\un q_{C_i}\}$, one for each irreducible component $C_i$ of $X$, such that $|\un q|:=\sum_i \un q_{C_i}\in \Z$.
A torsion-free rank-1  sheaf $I$ on $X$ of Euler characteristic  $\chi(I)$  equal to $|\un q|$ is called
\emph{$\un q$-semistable} (resp. \emph{$\un q$-stable}) if
for every proper subcurve $Y\subset X$, we have that
\begin{equation*}
\chi(I_Y)\geq  \un q_Y:= \sum_{C_i\subseteq Y} \un q_{C_i} \: \: (\text{resp. } > ),
\end{equation*}
where $I_Y$ is the biggest torsion-free quotient of the restriction $I_{|Y}$ of $I$ to the subcurve $Y$. A polarization $\un q$ is called {\em general} if $\un q_Y\not\in \Z$ for any proper subcurve $Y\subset X$ such that $Y$ and $Y^c$ are connected. If $\un q$ is general, there are no strictly $\un q$-semistable sheaves, i.e. if every $\un q$-semistable sheaf is also $\un q$-stable (see Lemma \ref{L:nondeg}); the converse being true for curves with locally planar singularities (see
Lemma \ref{L:ndg-conv}).  For every general polarization $\un q$, the subset $\ov{J}_X(\un q)\subseteq \bJbar_X$ parametrizing $\un q$-stable (or, equivalently, $\un q$-semistable) sheaves  is an open and projective subscheme (see Fact \ref{F:Este-Jac}), that we call the \emph{fine compactified Jacobian}  with respect to the polarization $\un q$.
The name ``fine'' comes from the fact that there exists an universal sheaf $\I$ on $X\times \bJbar_X$, unique up to tensor product with the pull-back of a line bundle from $\bJbar_X$, which restricts to a universal sheaf on $X\times \ov{J}_X(\un q)$
(see Fact \ref{F:huge-Jac}).

Our first main result concerns the properties of fine compactified Jacobians under the assumption that $X$ has locally planar singularities.

\begin{thmA}\label{T:MainA}
Let $X$ be a reduced projective connected curve of arithmetic genus $g$ and assume that $X$ has locally planar singularities. Then every fine compactified Jacobian $\ov{J}_X(\un q)$ satisfies the following properties:
\begin{enumerate}[(i)]
\item \label{T:MainA1} $\ov{J}_X(\un q)$ is a reduced scheme with locally complete intersection singularities and embedded dimension at most $2g$ at every point;
\item \label{T:MainA2} The smooth locus of $\ov{J}_X(\un q)$ coincides with the open subset $J_X(\un q)\subseteq \ov{J}_X(\un q)$ parametrizing line bundles; in particular $J_X(\un q)$ is dense in $\ov{J}_X(\un q)$ and
$\ov{J}_X(\un q)$ is of pure dimension equal to $p_a(X)$;
\item \label{T:MainA3} $\ov{J}_X(\un q)$ is connected;
\item \label{T:MainA4} $\ov{J}_X(\un q)$ has trivial dualizing sheaf;
\item \label{T:MainA5} $J_X(\un q)$ is the disjoint union of a number of copies of $J(X)$ equal to the complexity $c(X)$ of the curve $X$ (in the sense of Definition \ref{D:equiv-mult}); in particular, $\ov{J}_X(\un q)$ has $c(X)$ irreducible components, independently of the chosen polarization $\un q$ (see Corollary \ref{C:irre-comp}).
\end{enumerate}
\end{thmA}

Part \eqref{T:MainA1} and part \eqref{T:MainA2} of the above Theorem are deduced in Corollary \ref{C:prop-fineJac} from the analogous statements about the scheme $\bJbar_X$, which are in turn deduced, via the Abel map, from similar statements on the Hilbert scheme $\Hilb^n(X)$ of zero-dimensional subschemes of $X$ of length $n$ (see Theorem \ref{T:prop-J-planar}).
Part \eqref{T:MainA3} and part \eqref{T:MainA4} are proved in Section \ref{S:univ-Jac} (see Corollaries \ref{C:connect} and \ref{C:triv-can}), where we used in a crucial way the properties of the universal fine compactified Jacobians (see the discussion below). Finally, part \eqref{T:MainA5} is deduced in Corollary \ref{C:irre-comp} from a result of J. L. Kass \cite{Kas2} (generalizing previous results of S. Busonero (unpublished) and Melo-Viviani \cite{MV} for nodal curves) that says that any relative fine compactified Jacobian associated to a $1$-parameter regular smoothing  of $X$ (in the sense of Definition \ref{D:1par-sm}) is a compactification of the N\'eron model of the Jacobian of the generic fiber (see Fact \ref{F:Jac-Ner}), together with a result of Raynaud \cite{Ray} that describes the connected component of the central fiber of the above N\'eron model (see Fact \ref{F:complNer}).
In the proof of all the statements of the above Theorem, we use in an essential way the fact that the curve has locally planar singularities and indeed we expect that many of the above properties are false without this assumption (see also Remark \ref{R:Jac-integral}).

\vspace{0.1cm}

Notice that the above Theorem A implies that any two fine compactified Jacobians of a curve $X$ with locally planar singularities  are birational (singular) Calabi-Yau varieties. However, for a reducible curve, fine compactified Jacobians are not necessarily isomorphic (and not even homeomorphic if $k=\bbC$).

\begin{thmB}\label{T:MainB}
Let $X$ be a reduced projective connected curve.
\begin{enumerate}[(i)]
\item \label{T:MainB1} There is a finite number of isomorphism classes of fine compactified Jacobians of $X$.
\item \label{T:MainB2} The number of isomorphism classes of fine compactified Jacobians of a given curve $X$ can be arbitrarily large as $X$ varies, even among the class of nodal curves of genus $2$.
\item \label{T:MainB3}  If $k=\bbC$ then the number of homeomorphism classes of fine compactified Jacobians of a given curve $X$ can be arbitrarily large as $X$ varies, even among the class of nodal curves of genus $2$.
\end{enumerate}
\end{thmB}

Part \eqref{T:MainB1} of the above Theorem follows by Proposition \ref{P:finite-eq}, which says that there is a finite number of fine compactified Jacobians of a given curve $X$ from which all the others can be obtained via tensorization with some line bundle.
Parts \eqref{T:MainB2} and \eqref{T:MainB3} are proved by analyzing the poset of orbits for the natural action of the generalized Jacobian on  a given fine compactified Jacobian of a nodal curve. Proposition \ref{P:inv-orb} says that the poset of orbits is an invariant of the fine compactified Jacobian (i.e. it does not depend on the action of the generalized Jacobian) while Proposition \ref{P:inv-top} says that over $k=\bbC$ the poset of orbits is a topological invariant. Moreover, from the work of Oda-Seshadri \cite{OS}, it follows that the poset of orbits of a fine compactified Jacobian of a nodal curve $X$ is isomorphic to the poset of regions of a suitable simple toric arrangement of hyperplanes (see Fact \ref{F:OS}). In Example
 \ref{E:non-iso}, we construct a family of nodal curves of genus $2$ for which the number of simple toric arrangements with pairwise non isomorphic poset of regions grows to infinity, which concludes the proof of parts \eqref{T:MainB2} and \eqref{T:MainB3}.

We mention that, even though if fine compactified Jacobians of a given curve $X$  can be non isomorphic, they nevertheless share many geometric properties.
For example, the authors proved in \cite{MRV2} that any two fine compactified Jacobians of a reduced $X$ with locally planar singularities are derived equivalent under the Fourier-Mukai transform with kernel given by a natural Poincar\'e sheaf on the product.  This result seems to suggest an extension to (mildly) singular varieties of the conjecture of Kawamata \cite{Kaw}, which predicts that birational Calabi-Yau smooth projective varieties should be derived equivalent.
Moreover,  the third author, together with L. Migliorini and V. Schende, proved in \cite{MSV} that any two fine compactified Jacobians of $X$ (under the same assumptions on $X$) have the same Betti numbers if $k=\bbC$, which again seems to suggest an extension to (mildly) singular varieties of the result of Batyrev \cite{Bat} which  says that birational Calabi-Yau smooth projective varieties have the same Hodge numbers.

\vspace{0.2cm}

As briefly mentioned above, in the proof of parts \eqref{T:MainA3} and \eqref{T:MainA4} of Theorem A, an essential role is played by the properties of the \textbf{universal fine compactified Jacobians}, which are defined as follows. Consider the effective semiuniversal deformation $\pi:\X\to \Spec R_X$ of $X$ (see \S\ref{S:def-space} for details). For any (schematic) point $s\in \Spec R_X$, we denote by $\X_s:=\pi^{-1}(s)$ the fiber of $\pi$ over $s$ and by $\X_{\ov s}:=\X_s\times_{k(s)} \ov{k(s)}$ the geometric fiber over $s$. By definition, $X=\X_o=\X_{\ov{o}}$ where $o=[\m_X]\in \Spec R_X$ is the unique closed point $o\in \Spec R_X$ corresponding to the maximal ideal $\m_X$ of the complete local  $k$-algebra $R_X$.
A polarization $\un q$ on $X$ induces in a natural way a polarization $\un q^s$ on $\X_ {\ov s}$ for every $s\in \Spec R_X$ which, moreover, will be general if we start from a general polarization $\un q$ (see Lemma-Definition \ref{D:def-pola}).

\begin{thmC}\label{T:MainC}
Let $\un q$ be a general polarization on a reduced projective connected curve $X$.
There exists a scheme $u: \J_{\X}(\un q)\to \Spec R_X$ parametrizing coherent sheaves $\I$ on $\X$, flat over $\Spec R_X$, whose geometric fiber $\I_{\ov s}$ over any $s\in \Spec R_X$ is  a $\un q^s$-semistable (or, equivalently, $\un q^s$-stable) sheaf on $\X_{\ov s}$.  The morphism $u$  is projective
and its geometric fiber over any point $s\in \Spec R_X$ is isomorphic to $\J_{\X_{\ov s}}(\un q^s)$. In particular, the fiber of $\J_{\X}(\un q)\to \Spec R_X$
over the closed point $o=[\m_X]\in \Spec R_X$ is isomorphic to $\J_X(\un q)$.

Moreover, if $X$ has locally planar singularities then we have:
\begin{enumerate}[(i)]
\item \label{T:MainC1} The scheme $\J_{\X}(\un q)$ is regular and irreducible.
\item \label{T:MainC2} The map $u:\J_{\X}(\un q)\to \Spec R_X$ is flat of relative dimension $p_a(X)$ and it has trivial relative dualizing sheaf.
\item \label{T:MainC3} The smooth locus of $u$ is the open subset $J_{\X}(\un q)\subseteq \J_{\X}(\un q)$ parametrizing  line bundles on $\X$.
\end{enumerate}

\end{thmC}

The first statement of the above Theorem is obtained in Theorem \ref{T:univ-fine} by applying to the family $\pi:\X\to \Spec R_X$ a result of Esteves \cite{est1} on the existence of relative fine compactified
Jacobians.

In order to prove the second part of the above Theorem, the crucial step is to identify the completed local ring of $\J_{\X}(\un q)$ at a point $I$ of the central fiber $u^{-1}(o)=\J_X(\un q)$ with the semiuniversal deformation ring
for the deformation functor $\Def_{(X,I)}$ of the pair $(X,I)$ (see Theorem \ref{T:univ-rings}).  Then, we deduce the regularity of  $\J_{\X}(\un q)$ from a result of Fantechi-G\"ottsche-van Straten \cite{FGvS} which says that, if $X$ has locally planar singularities, then the deformation functor $\Def_{(X,I)}$ is smooth.
The other properties stated in the second part of Theorem C, which are proved in Theorem  \ref{T:univ-Jac} and Corollary \ref{C:triv-can}, follow from the regularity of $\J_{\X}(\un q)$ together with the properties of the geometric fibers of the morphism $u$.

\vspace{0.2cm}

Our final result concerns the existence of (twisted) \textbf{Abel maps}  of degree one into fine compactified Jacobians, a topic which has been extensively studied (see e.g. \cite{AK}, \cite{egk0}, \cite{egk}, \cite{EK}, \cite{CE}, \cite{CCE}, \cite{CP}).
To this aim, we restrict ourselves to connected and projective reduced curves $X$ satisfying the following
\begin{equation*}
\un{\text{Condition } (\dagger)}: \text{Every separating point is a node,}
\end{equation*}
where a separating point of $X$ is a singular point $p$ of $X$ for which there exists a subcurve $Z$ of $X$ such that $p$ is the scheme-theoretic intersection of $Z$ and its complementary subcurve $Z^c:=\ov{X\setminus Z}$.
For example, every Gorenstein curve satisfies condition $(\dagger)$ by \cite[Prop. 1.10]{Cat}.
Fix now a curve $X$ satisfying condition $(\dagger)$ and let $\{n_1,\ldots$ $ ,n_{r-1}\}$ be its separating points, which are nodes.   Denote by $\wt{X}$ the partial normalization of $X$ at the set $\{n_1,\ldots,n_{r-1}\}$.
Since each $n_i$ is a node, the curve $\wt{X}$ is a disjoint union of $r$ connected reduced curves $\{Y_1,\ldots,Y_r\}$ such that each $Y_i$ does not have separating points.
We have a natural morphism
\begin{equation*}
\tau:\wt{X}=\coprod_i Y_i\to X.
\end{equation*}
We can naturally identify each $Y_i$ with a subcurve of $X$ in such a way that their union  is $X$ and that they do not have common
 irreducible components. We call the components $Y_i$ (or their image in $X$) the \emph{separating blocks}  of $X$.

\begin{thmD}\label{T:MainD}
Let $X$ be a reduced projective connected curve satisfying condition $(\dagger)$.
\begin{enumerate}[(i)]
\item \label{T:MainD1} The pull-back map
$$\begin{aligned}
\tau^*: \bJbar_X & \longrightarrow \prod_{i=1}^r \bJbar_{Y_i}\\
I & \mapsto (I_{|Y_1},\ldots, I_{|Y_r}),
\end{aligned}
$$
is an isomorphism. Moreover, given any fine compactified Jacobians $\ov{J}_{Y_i}(\un q^i)$ on $Y_i$, $i=1,\dots,r$, there exists a (uniquely determined) fine compactified Jacobian $\J_X(\un q)$ on $X$ such that
\begin{equation*}
\tau^*: \ov{J}_X(\un q)\xrightarrow{\cong} \prod_i \ov{J}_{Y_i}(\un q^i),
\end{equation*}
and every fine compactified Jacobian on $X$ is obtained in this way.

\item \label{T:MainD2}
For every $L\in \Pic(X)$, there exists a unique morphism $A_L:X\to \bJbar_X^{\chi(L)-1}$  such that for every $1\leq i\leq r$ and every $p\in Y_i$ it holds that
\begin{equation*}
\tau^*(A_L(p))=(M_1^i,\ldots,M_{i-1}^i,\m_p \otimes L_{|Y_i},M_{i+1}^i,\ldots,M_r^i)
\end{equation*}
for some (uniquely determined) line bundle $M_j^i$ on $Y_j$ for any $j\neq i$, where $\m_p$ is the ideal of the point $p$ in $Y_i$.

\item \label{T:MainD3} If, moreover, $X$ is Gorenstein, then there exists a general polarization $\un q$ with $|\un q|=\chi(L)-1$ such that
$\Im A_L\subseteq \J_X(\un q)$.

\item \label{T:MainD4} For every $L\in \Pic(X)$, the morphism $A_L$ is an embedding away from the rational separating blocks (which are isomorphic to $\PP^1$) while it contracts each rational separating block $Y_i\cong \PP^1$
 into a seminormal point of $A_L(X)$, i.e. an ordinary singularity with linearly independent tangent directions.


\end{enumerate}
\end{thmD}

Some comments on the above Theorem are in order.

Part \eqref{T:MainD1}, which follows from Proposition \ref{P:norm-sheaves}, says that all fine compactified Jacobians of a curve satisfying
assumption $(\dagger)$ decompose uniquely as a product of fine compactified Jacobians of its separating blocks. This allows one to reduce many properties of fine compactified Jacobians of $X$ to properties of the fine compactified
Jacobians of its separating blocks $Y_i$,  which have the advantage of not having separating points.
Indeed, the first statement of part \eqref{T:MainD1} is due to Esteves \cite[Prop. 3.2]{est2}.

The map $A_L$ of part \eqref{T:MainD2}, which is
constructed  in Proposition \ref{P:Abel-sep}, is called the $L$-twisted Abel map. For a curve $X$ without separating points, e.g. the separating blocks $Y_i$, the map $A_L:X\to \bJbar_X$ is the natural map sending $p$ to
$\m_p\otimes L$. However, if $X$ has a separating point $p$, the ideal sheaf $\m_p$ is not simple and therefore the above definition is ill-behaved.
Part \eqref{T:MainD2} is saying that we can put together the natural Abel maps $A_{L_{|Y_i}}:Y_i\to \bJbar_{Y_i}$ on each separating block $Y_i$
in order to have a map $A_L$ whose restriction to $Y_i$ has $i$-th component equal to $A_{L_{|Y_i}}$  and it is constant  on the $j$-th components with $j\neq i$. Note that special cases of the Abel map $A_L$ (with $L=\O_X$ or $L=\O_X(p)$ for some smooth point $p\in X$) in the presence of separating
points have been considered before
by Caporaso-Esteves in \cite[Sec. 4 and Sec. 5]{CE} for nodal curves, by Caporaso-Coelho-Esteves in \cite[Sec. 4 and 5]{CCE} for Gorenstein curves and by Coelho-Pacini in \cite[Sec. 2]{CP} for curves of compact type.

Part \eqref{T:MainD3} says that if $X$ is Gorenstein then the image of each twisted Abel map $A_L$ is contained in a (non unique) fine compactified Jacobian. Any fine compactified Jacobian which contains the image
of a twisted Abel map is said to \emph{admit an Abel map}. Therefore, part \eqref{T:MainD3} says that any Gorenstein curve has some fine compactified Jacobian admitting an Abel map. However, we show that, in general,
 not every fine compactified Jacobian admits an Abel map: see Propositions \ref{P:JacIn}  and \ref{P:JacIV}  for some examples.

Part \eqref{T:MainD4} is proved by  Caporaso-Coelho-Esteves \cite[Thm. 6.3]{CCE} for Gorenstein curves, but their proof extends verbatim to our (more general) case.

\subsection*{Outline of the paper}

The paper is organized as follows.

Section \ref{S:comp-Jac} is devoted to collecting several facts on fine compactified Jacobians of reduced curves: in \S\ref{S:sheaves}, we consider the scheme $\bJbar_X$ parametrizing all simple torsion-free rank-1 sheaves on a curve $X$ (see Fact \ref{F:huge-Jac}) and we investigate its properties under the assumption that $X$ has locally planar singularities (see Theorem \ref{T:prop-J-planar}); in \S\ref{S:fine-Jac}, we introduce fine compactified Jacobians of $X$ (see Fact \ref{F:Este-Jac}) and study them under the assumption that $X$ has locally planar singularities (see Corollary \ref{C:prop-fineJac}).

In Section \ref{S:vary} we prove that there is a finite number of isomorphism classes of fine compactified Jacobians of a given curve  (see Proposition \ref{P:finite-eq}) although this number can be arbitrarily large even for nodal curves (see Corollary
 \ref{C:crit-iso} and Example \ref{E:non-iso}). In order to establish this second result, we study in detail in \S \ref{S:nodal} the poset of orbits for fine compactified Jacobians of nodal curves.

Section \ref{S:defo} is devoted to recalling and proving some basic facts on deformation theory: we study the deformation functor $\Def_X$ of a curve $X$ (see \S\ref{S:DefX}) and the deformation functor $\Def_{(X,I)}$ of a pair $(X,I)$
 consisting of a curve $X$ together with a torsion-free, rank-1 sheaf $I$ on $X$ (see \S\ref{S:Def-pair}). Finally, in \S\ref{S:def-space}, we study the semiuniversal deformation spaces for a curve $X$ and for a pair $(X,I)$ as above.

In Section \ref{S:univ-Jac}, we introduce the universal fine compactified Jacobians relative to the semiuniversal deformation of a curve $X$ (see Theorem \ref{T:univ-fine}) and we study its properties  under the assumption that $X$ has locally planar singularities (see Theorem \ref{T:univ-Jac}). We then deduce some interesting consequences of our results for fine compactified Jacobians (see Corollaries \ref{C:connect} and \ref{C:triv-can}).
In \S\ref{S:1par-sm}, we use a result of J.  L. Kass in order to prove that the pull-back of any universal fine compactified Jacobian under a 1-parameter regular smoothing of the curve (see Definition \ref{D:1par-sm}) is a compactification of the N\'eron model of the Jacobian of the general fiber (see Fact \ref{F:Jac-Ner}). From this result we get a formula for the number of irreducible components of a fine compactified Jacobian (see Corollary \ref{C:irre-comp}).

In Section \ref{S:Abel}, we introduce Abel maps: first for curves that do not have separating points (see \S\ref{S:Abel-nonsep}) and then for curves all of whose separating points are nodes (see \S\ref{S:Abel-sep}).

In Section \ref{S:genus1}, we illustrate the general theory developed so far with the study of fine compactified Jacobians of Kodaira curves, i.e. curves of arithmetic genus one with locally planar singularities and without separating points.

\subsection*{Notations}

The following notations will be used throughout the paper.

\begin{convention}
 	$k$ will denote an algebraically closed field (of arbitrary characteristic), unless otherwise stated. All \textbf{schemes} are $k$-schemes, and all morphisms are implicitly assumed to respect
	the $k$-structure.
\end{convention}

\begin{convention}\label{N:curves}
	A \textbf{curve}  is a \emph{reduced} projective scheme over $k$ of pure dimension $1$.
	
Given a curve $X$, we denote by $X_{\rm sm}$ the smooth locus of $X$, by $X_{\rm sing}$ its singular locus and by $\nu:X^{\nu}\to X$ the normalization morphism.
 We denote by $\gamma(X)$, or simply by $\gamma$ where there is no danger of confusion, the number of irreducible components of $X$.

We denote by $p_a(X)$  the \emph{arithmetic genus} of $X$, i.e.  $p_a(X):=1-\chi(\O_X)=1-h^0(X,\O_X)+h^1(X, \O_X) $.
We denote by $g^{\nu}(X)$ the \emph{geometric genus} of $X$, i.e. the sum of the genera of the connected components of the normalization $X^{\nu}$.

\end{convention}

\begin{convention}
	A \textbf{subcurve} $Z$ of a curve $X$ is a closed $k$-subscheme $Z \subseteq X$ that is reduced  and of pure dimension $1$.  We say that a subcurve $Z\subseteq X$ is proper if
	$Z\neq \emptyset, X$.
	
	Given two subcurves $Z$ and $W$ of $X$ without common irreducible components, we denote by $Z\cap W$ the $0$-dimensional subscheme of $X$ that is obtained as the
	scheme-theoretic intersection of $Z$ and $W$ and we denote by $|Z\cap W|$ its length.
	
	Given a subcurve $Z\subseteq X$, we denote by $Z^c:=\ov{X\setminus Z}$ the \textbf{complementary subcurve} of $Z$ and we set $\delta_Z=\delta_{Z^c}:=|Z\cap Z^c|$.
 \end{convention}

\begin{convention}
A curve $X$ is called \textbf{Gorenstein} if its dualizing sheaf $\omega_X$ is a line bundle.
\end{convention}

\begin{convention}
A curve $X$ has \textbf{locally complete intersection (l.c.i.) singularities at $p\in X$} if the completion $\wh{\O}_{X,p}$ of the local ring of $X$ at $p$ can be written as
$$\wh{\O}_{X,p}=k[[x_1,\ldots,x_r]]/(f_1,\ldots,f_{r-1}),$$
for some $r\geq 2$ and some $f_i\in k[[x_1,\ldots,x_r]]$. A curve $X$ has locally complete intersection (l.c.i.)
singularities if $X$ is l.c.i. at every $p\in X$.
It is well know that a curve with l.c.i. singularities is Gorenstein.
\end{convention}

\begin{convention}\label{N:lps}
A curve $X$ has \textbf{locally planar singularities at $p\in X$} if  the completion
$\wh{\O}_{X,p}$ of the local ring of $X$ at $p$ has embedded dimension two, or equivalently if it can be written
as
$$\wh{\O}_{X,p}=k[[x,y]]/(f),$$
for a reduced series $f=f(x,y)\in k[[x,y]]$.
A curve $X$ has locally planar singularities if $X$ has locally planar singularities at every $p\in X$.
Clearly, a curve with locally planar singularities has l.c.i. singularities, hence it is Gorenstein. A (reduced) curve has locally planar singularities if and only if it can be embedded in a smooth surface (see \cite{AK0}).
\end{convention}

\begin{convention}
A curve $X$ has a \textbf{node at $p\in X$} if  the completion
$\wh{\O}_{X,p}$ of the local ring of $X$ at $p$ is isomorphic to
$$\wh{\O}_{X,p}=k[[x,y]]/(xy).$$
\end{convention}

\begin{convention}\label{N:sep-node}
A \textbf{separating point} of a curve $X$ is a geometric point $n\in X$ for which there exists a subcurve
$Z\subset X$ such that $\delta_Z=1$ and $Z\cap Z^c=\{n\}$.
If $X$ is Gorenstein, then a separating point $n$ of $X$
is a node of $X$, i.e. $\wh{\O}_{X,n}=k[[x,y]]/(xy)$ (see Fact \ref{F:Gor-dagger}). However this is false in general
without the Gorenstein assumption (see Example \ref{Ex:non-dagger}).
\end{convention}



\begin{convention}\label{N:Jac-gen}
Given a curve $X$, the \textbf{generalized Jacobian} of $X$, denoted by $J(X)$ or by $\Pic^{\un 0}(X)$,
is the algebraic group whose group of $k$-valued points is the group of line bundles on $X$ of multidegree $\un 0$ (i.e. having
degree $0$ on each irreducible component of $X$) together with the multiplication given by the tensor product.
The generalized Jacobian of $X$ is a connected commutative smooth algebraic group of dimension equal to
$h^1(X,\O_X)$.
\end{convention}

\section{Fine Compactified Jacobians}\label{S:comp-Jac}

The aim of this section is to collect several facts about compactified Jacobians of connected reduced curves, with
special emphasis on connected reduced curves with locally planar singularities. Many of these facts are well-known to the experts but for many of them we could not find satisfactory references in the existing literature, at least at the
level of generality we need, e.g. for reducible curves. Throughout this section, we fix a connected reduced curve $X$.


\subsection{Simple rank-1 torsion-free sheaves}\label{S:sheaves}

We start by defining the sheaves on the connected curve $X$ we will be working with.

\begin{defi}
A coherent sheaf $I$ on a connected curve $X$ is said to be:
\begin{enumerate}[(i)]
\item \emph{rank-1} if $I$ has generic rank $1$ on every irreducible component of $X$;
\item \emph{torsion-free} (or pure, or $S_1$) if $\Supp(I)=X$ and every non-zero subsheaf $J\subseteq I$ is such that $\dim \Supp(J)=1$;
\item \emph{simple} if ${\rm End}_k(I)=k$.
\end{enumerate}
\end{defi}
\noindent Note that any line bundle on $X$ is a simple rank-1 torsion-free sheaf.

Consider the functor
\begin{equation}\label{E:func-Jbar}
\bJbar_X^* : \{{\rm Schemes}/k\}  \to \{{\rm Sets}\}
\end{equation}
which associates to a $k$-scheme $T$ the set of isomorphism classes of $T$-flat, coherent sheaves on $X\times _k T$
whose fibers over $T$ are simple rank-1 torsion-free sheaves (this definition agrees with the one in \cite[Def. 5.1]{AK} by virtue of \cite[Cor. 5.3]{AK}).
  The functor $\bJbar_X^*$ contains the open subfunctor
\begin{equation}\label{E:func-J}
\bJ_X^* : \{{\rm Schemes}/k\}  \to \{{\rm Sets}\}
\end{equation}
which associates to a $k$-scheme $T$ the set of isomorphism classes of line bundles on $X\times _k T$.

\begin{fact}[Murre-Oort, Altman-Kleiman, Esteves]\label{F:huge-Jac}
Let $X$ be a connected reduced curve. Then
\noindent
\begin{enumerate}[(i)]
\item \label{F:huge1} The Zariski (equiv. \'etale, equiv. fppf) sheafification of $\bJ_X^*$ is represented by a $k$-scheme $\Pic(X)=\bJ_X$, locally of finite type over $k$. Moreover, $\bJ_X$ is formally smooth over $k$.
\item \label{F:huge2} The Zariski (equiv. \'etale, equiv. fppf)  sheafification of $\bJbar_X^*$ is represented by a $k$-scheme $\bJbar_X$, locally of finite type over $k$. Moreover, $\bJ_X$ is an open subset of $\bJbar_X$ and $\bJbar_X$ satisfies the valuative criterion for universally closedness or, equivalently, the existence part of the valuative criterion for properness\footnote{Notice however that the scheme $\bJbar_X$ fails to be universally closed because it is not quasi-compact.
}.
\item \label{F:huge3} There exists a sheaf $\I$ on $X\times \bJbar_X$  such that for every $\F\in \bJbar_X^*(T)$ there exists a unique map $\alpha_{\F}:T\to \bJbar_X$ with the property that
$\F=(\id_X\times \alpha_{\F})^*(\I)\otimes \pi_2^*(N)$ for some $N\in \Pic(T)$, where $\pi_2:X\times T\to T$ is the projection onto the second factor.
The sheaf $\I$ is uniquely determined up  to tensor product with the pullback of an invertible sheaf on $\bJbar_X$ and it is called a \emph{universal sheaf}.
\end{enumerate}
\end{fact}
\begin{proof}
Part (i): the representability of the fppf sheafification of $\bJ_X^*$ follows from a result of Murre-Oort
(see \cite[Sec. 8.2, Thm. 3]{BLR} and the references therein). However, since $X$ admits a $k$-rational point (because $k$ is assumed to be algebraically closed), the fppf sheafification of $\bJ_X^*$ coincides with its \'etale (resp. Zariski) sheafification (see \cite[Thm. 9.2.5(2)]{FGA}).
The formal smoothness of $\bJ_X$ follows from \cite[Sec. 8.4, Prop. 2]{BLR}.

Part (ii): the representability of the \'etale sheafification (and hence of the fppf sheafification) of $\bJbar_X^*$ by an algebraic space
$\bJbar_X$ locally of finite type over $k$ follows from a general result of Altmann-Kleiman (\cite[Thm. 7.4]{AK}). Indeed, in \cite[Thm. 7.4]{AK} the authors state the result for the moduli functor of simple sheaves; however, since the condition of being torsion-free and rank-1 is an open condition (see e.g. the proof of \cite[Prop 5.12(ii)(a)]{AK}), we also get the representability of $\bJbar_X^*$. The fact that $\bJbar_X$ is a scheme follows from a general result of Esteves  (\cite[Thm. B]{est1}), using the fact that each irreducible component of $X$ has a  $k$-point (recall that $k$ is assumed to be algebraically closed).
Moreover, since  $X$ admits a smooth $k$-rational point, the \'etale sheafification of $\bJbar_X^*$ coincides with the Zariski sheafification by \cite[Thm. 3.4(iii)]{AK2}.
 Since $\bJ_X^*$ is an open subfunctor of $\bJbar_X^*$ then  $\bJ_X$ is an open subscheme of $\bJbar_X$.
Finally, the fact that $\bJbar_X$ satisfies the existence condition of the valuative criterion for properness follows from \cite[Thm. 32]{est1}.

Part (iii) is an immediate consequence of the fact that $\bJbar_X$ represents the Zariski sheafification of $\bJbar_X^*$ (see  also \cite[Thm. 3.4]{AK2}).

\end{proof}

Since the Euler-Poincar\'e characteristic $\chi(I):=h^0(X,I)-h^1(X,I)$ of a sheaf $I$ on $X$ is constant under deformations, we get a decomposition
\begin{equation}\label{E:dec}
\begin{sis}
& \bJbar_X=\coprod_{\chi \in \Z} \bJbar_X^{\chi},\\
& \bJ_X=\coprod_{\chi\in \Z} \bJ_X^{\chi},\\
\end{sis}
\end{equation}
where $\bJbar_X^{\chi}$ (resp. $\bJ_X^{\chi}$) denotes the open and closed subscheme of $\bJbar_X$ (resp. $\bJ_X$) parametrizing
simple rank-1 torsion-free sheaves $I$ (resp. line bundles $L$) such that $\chi(I)=\chi$ (resp. $\chi(L)=\chi$).

If $X$ has locally planar singularities, then $\bJbar_X$ has the following properties.

\begin{thm}\label{T:prop-J-planar}
Let $X$ be a connected reduced curve with locally planar singularities. Then
\begin{enumerate}[(i)]
\item \label{T:prop-J-planar1} $\bJbar_X$ is a reduced scheme with locally complete intersection singularities and embedded dimension at most $2p_a(X)$ at every point.
\item \label{T:prop-J-planar2} $\bJ_X$ is dense in $\bJbar_X$.
\item \label{T:prop-J-planar3} $\bJ_X$ is the smooth locus of $\bJbar_X$.
\end{enumerate}
\end{thm}

The required properties of $\bJbar_X$ will be deduced from the analogous properties of the punctual Hilbert scheme (i.e. the Hilbert scheme of $0$-dimensional subschemes) of $X$ via the Abel map.

Let us first review the needed properties of the punctual Hilbert scheme. Denote by $\Hilb^d(X)$ the Hilbert scheme parametrizing subschemes $D$ of $X$ of finite length $d\geq 0$, or equivalently ideal sheaves $I\subset \O_X$ such that $\O_X/I$ is a finite scheme of length $d$. Given $D\in \Hilb^d X$, we will denote by $I_D$ its ideal sheaf. 
We introduce the following subschemes of $\Hilb^d(X)$:
$$\begin{sis}
& \Hilb^d(X)_s:=\{D \in \Hilb^d(X)\: : \: I_D \text{ is simple}\},\\
& \Hilb^d(X)_l:=\{D \in \Hilb^d(X)\: : \: I_D \text{ is a line bundle}\}.\\
\end{sis}$$
By combining the results of \cite[Prop. 5.2 and Prop 5.13(i)]{AK}, we get that the natural inclusions
$$ \Hilb^d(X)_l\subseteq \Hilb^d(X)_s\subseteq \Hilb^d(X)$$
are open inclusions.

\begin{fact}\label{F:propHilb}
If $X$ is a reduced curve with locally planar singularities, then the Hilbert scheme $\Hilb^d(X)$ has the following properties:
\begin{enumerate}[(a)]
\item \label{F:propHilba} $\Hilb^d(X)$  is reduced with locally complete intersection singularities and embedded dimension at most $2d$ at every point.
\item \label{F:propHilbb} $\Hilb^d(X)_l$ is dense in $\Hilb^d(X)$.
\item \label{F:propHilbc} $\Hilb^d(X)_l$ is the smooth locus of $\Hilb^d(X)$.
\end{enumerate}
The above properties do hold true if $\Hilb^d(X)$ is replaced by $\Hilb^d(X)_s$. 
\end{fact}
\begin{proof}
Part \eqref{F:propHilba} follows from \cite[Cor. 7]{AIK} (see also \cite[Prop. 1.4]{BGS}), part \eqref{F:propHilbb} follows from  \cite[Thm. 8]{AIK} (see also \cite[Prop. 1.4]{BGS}) and  part \eqref{F:propHilbc} follows from 
\cite[Prop. 2.3]{BGS}. 

The above properties do remain true for $\Hilb^d(X)_s$ since $\Hilb^d(X)_s$ is an open subset of $\Hilb^d(X)$ containing $\Hilb^d(X)_l$. 
\end{proof}

The punctual Hilbert scheme of $X$ and the moduli space $\bJbar_X$ are related via the Abel map, which is defined as follows. Given a line bundle $M$ on $X$, we define the $M$-twisted Abel map of degree $d$ by
\begin{equation}\label{E:Abel-map}
\begin{aligned}
A_M^d:\Hilbs^d_X & \longrightarrow \bJbar_X, \\
D & \mapsto I_D\otimes M. \\
\end{aligned}
\end{equation}
Note that, by definition, it follows that
\begin{equation}\label{E:lb-locus}
(A_M^d)^{-1}(\bJ_X)=\Hilb^d(X)_l.
\end{equation}

The following result (whose proof was kindly suggested to us by J.L. Kass) shows that, locally on the codomain, the $M$-twisted Abel map of degree $p_a(X)$ is smooth and surjective (for a suitable choice of $M\in \Pic(X)$),
at least if $X$ is Gorenstein.



\begin{prop}\label{P:Abel-sur}
Let $X$ be a (connected and reduced) Gorenstein curve of arithmetic genus $g:=p_a(X)$.
 There exists a cover of $\bJbar_X$ by $k$-finite type open subsets $\{U_{\beta}\}$ such that, for each such $U_{\beta}$, there exists $M_{\beta}\in  \Pic(X)$ with the property that 
 $\Hilbs^{g}_X\supseteq  V_{\beta}:=(A^g_{M_{\beta}})^{-1}(U_{\beta})\stackrel{A^g_{M_{\beta}}}{\longrightarrow} U_{\beta}$ is smooth and surjective.
\end{prop}
\begin{proof}
Observe that, given $I\in \bJbar_X^{\chi}$ and $M \in \Pic(X)$, we have:
\begin{enumerate}[(i)]
\item \label{E:cond1} $I$ belongs to the image of $A_M^{\chi(M)-\chi}$  if (and only if)
there exists an injective homomorphism $I\to M$;
\item \label{E:cond2} $A_M^{\chi(M)-\chi}$ is smooth along $(A_M^{\chi(M)-\chi})^{-1}(I)$ provided that $\Ext^1(I,M)=0$.
\end{enumerate}
Indeed, if there exists an injective homomorphism $I\to M$, then its cokernel  is the structure sheaf of
a $0$-dimensional subscheme $D\subset X$ of length equal to  $\chi(M)-\chi(I)=\chi(M)-\chi$ with the property that 
$I_D=I\otimes M^{-1}$. Therefore
$A_M^{\chi(M)-\chi}(D)=I_D\otimes M=(I\otimes M^{-1})\otimes M=I,$ which implies part \eqref{E:cond1}.  
Part \eqref{E:cond2} follows from \cite[Thm. 5.18(ii)]{AK}\footnote{Note that in loc.cit., this is stated under the assumption that $X$ is integral. However, a close inspection of the proof reveals that this continues to hold
true under the assumption that $X$ is only reduced. The irreducibility is only used in part (i) of \cite[Thm. 5.18]{AK}.}. 

Fixing $M\in \Pic(X)$, the conditions \eqref{E:cond1} and \eqref{E:cond2} are clearly open conditions on $\bJbar_X^{\chi}$; hence the proof of the Proposition follows from the case $n=g$ of the following

\un{Claim:} For any $I\in \bJbar_X^{\chi}$ and any $n\geq g$, there exists $M_n\in \Pic(X)$ with $\chi(M)=n+\chi$ such that 
\begin{enumerate}[(a)] 
\item there exists an injective homomorphism $I\to M_n$;
\item $\Ext^1(I,M_n)=0$.
\end{enumerate}

First of all, observe that, for any $I\in \bJbar_X^{\chi}$ and any line bundle $N$, the local-to-global spectral sequence $H^p(X,{\mathcal Ext}^q(I, N))\Rightarrow \Ext^{p+q}(I,N)$ 
gives that
$$H^0(X,{\mathcal Hom}(I,N))= \Ext^0(I,N), $$
$$0\to H^1(X,{\mathcal Hom}(I,N))\to \Ext^1(I,N)\to H^1(X,{\mathcal Ext}^1(I,N)).$$
Moreover, the sheaf ${\mathcal Ext}^1(I,N)={\mathcal Ext}^1(I,\O_X)\otimes N$ vanishes by \cite[Prop. 1.6]{Har3},
so that we get 
\begin{equation}\label{E:Ext-H}
\Ext^i(I, N)=H^i(X,{\mathcal Hom}(I,N))=H^i(X, I^*\otimes N)  \hspace{0.3cm} \text{Êfor }Êi=0,1
\end{equation}
where $I^*:={\mathcal Hom}(I,\O_X)\in \bJbar_X$.  
From \eqref{E:Ext-H} and Riemann-Roch, we get 
\begin{equation}\label{E:chiExt}
\dim \Ext^0(I, N)-\dim \Ext^1(I,N)=\chi(I^*\otimes N)= \deg N+\chi(I^*)=\deg N +2(1-g)-\chi(I)=\chi(N)-\chi+1-g.
\end{equation}

We will now prove the claim by decreasing induction on $n$.
The claim is true (using \eqref{E:Ext-H}) if $n\gg 0$ and $M_n$ is chosen to be a sufficiently high power of a very ample line bundle on $X$.
Suppose now that we have a line bundle $M_{n+1}\in \Pic^{n+1}(X)$ with $\chi(M_{n+1})=n+1+\chi$ (for a certain $n\geq g$) which satisfies the properties of the Claim. 
We are going to show that, for a generic smooth point $p\in X$,
 the line bundle $M_n:=M_{n+1}\otimes \O_X(-p)\in \Pic(X)$ also satisfies the properties of the Claim.

Using \eqref{E:chiExt} and the properties of $M_{n+1}$, it is enough to show that $M_n:=M_{n+1}\otimes \O_X(-p)$, for $p\in X$ generic, satisfies
\begin{equation*}
\dim \Hom(I,M_n)=\dim \Hom(I,M_{n+1})-1, \tag{*}
\end{equation*}
\begin{equation*}
\text{the generic element } [I\to M_n]\in \Hom(I,M_n) \text{ is injective.}\tag{**}
\end{equation*} 
Tensoring the exact sequence 
$$ 0\to \O_X(-p)\to \O_X \to \O_p\to 0$$
with $I^*\otimes M_{n+1}$ and taking cohomology, we get the exact sequence
$$0\to \Hom(I,M_n)=H^0(X, I^*\otimes M_n)\to \Hom(I,M_{n+1})=H^0(X, I^*\otimes M_{n+1})\stackrel{e}{\longrightarrow} {\bf k}_p, $$
where $e$ is the evaluation of sections at $p\in X$. By the assumptions on $M_{n+1}$ and \eqref{E:chiExt}, 
we have that 
$$\dim \Hom(I,M_{n+1})= \chi(M_{n+1})-\chi+1-g=n+1+1-g \geq 2,$$
and, moreover, that the generic element  $[I\to M_{n+1}]\in \Hom(I,M_{n+1})$ is injective.
By choosing a point $p\in X$ for which there exists a section $s\in H^0(X, I^*\otimes M_{n+1})$
which does not vanish in $p$, we get that (*) and (**) holds true for $M_n=M_{n+1}\otimes \O_X(-p)$, q.e.d. 

\end{proof}

\begin{remark}
From the proof of the second statement of Theorem \ref{T:prop-J-planar}\eqref{T:prop-J-planar1}  and Remark \ref{R:Jac-integral}\eqref{R:Jac-integral3} below, it will follow that the above Proposition is, in general, false if $g$ is replaced by any smaller integer. 
\end{remark}


With the above preliminaries results, we can now give a proof of Theorem \ref{T:prop-J-planar}.  

\begin{proof}[Proof of Theorem \ref{T:prop-J-planar}]

Observe that each of the three statements of the theorem is local in $\bJbar_X$, i.e. it is sufficient to check it on an open cover
of $\bJbar_X$. Consider the open cover $\{U_{\beta}\}$ given by Proposition \ref{P:Abel-sur}.

Part (i): from Fact \ref{F:propHilb}\eqref{F:propHilba}, it follows that $V_{\beta}\subset \Hilbs^{g}_X$ is reduced with locally complete intersection singularities and embedded dimension at most $2g=2p_a(X)$. Since $(A_M^g)_{|V_{\beta}}$ is smooth and surjective into $U_{\beta}$, also $U_{\beta}$ inherits the same properties.

Part (ii): from Fact \ref{F:propHilb}\eqref{F:propHilbb}, it follows that $\Hilb^g(X)_l\cap V_{\beta}$ is dense in $V_{\beta}$. From the surjectivity of $(A_M^g)_{|V_{\beta}}$ together with \eqref{E:lb-locus}, it follows that
$A_M^g(V_{\beta}\cap \Hilb^g(X)_l)=U_{\beta}\cap \bJ_X$ is dense in $A_M^g(V_{\beta})=U_{\beta}$.

Part (iii): from Fact \ref{F:propHilb}\eqref{F:propHilbc}, it follows that $\Hilb^g(X)_l\cap V_{\beta}$ is the smooth locus of $V_{\beta}$. Since $(A_M^g)_{|V_{\beta}}$ is smooth and surjective and \eqref{E:lb-locus} holds, 
we infer that $A_M^g(V_{\beta}\cap \Hilb^g(X)_l)=U_{\beta}\cap \bJ_X$ is the smooth locus of  $A_M^g(V_{\beta})=U_{\beta}$.

\end{proof}

\begin{remark}\label{R:Jac-integral}
\noindent
\begin{enumerate}[(i)]
\item Theorem \ref{T:prop-J-planar} is well-known (except perhaps the statement about the embedded dimension) in the case where $X$ is irreducible (and hence integral):  the first assertion in part (i) and part (ii) are due to Altman-Iarrobino-Kleiman \cite[Thm. 9]{AIK}; part (iii) is due to Kleppe \cite{Kle} (unpublished, for a proof see \cite[Prop. 6.4]{Kas}). Note that, for $X$ irreducible, part (ii) is equivalent to the irreducibility of $\bJbar_X^d$ for a certain 
$d\in \Z$ (hence for all $d\in \Z$).
\item The hypothesis that $X$ has locally planar singularities is crucial in the above Theorem \ref{T:prop-J-planar}:
\begin{itemize}
\item Altman-Iarrobino-Kleiman constructed in \cite[Exa. (13)]{AIK}
an  integral curve without locally planar singularities (indeed, a curve which is a complete intersection in $\PP^3$) for which $\bJbar_X^d$ (for any $d\in \Z$) is not irreducible (equivalently, $\bJ_X^d$ is not dense in $\bJbar_X^d$).
Later, Rego (\cite[Thm. A]{Reg}) and Kleppe-Kleiman (\cite[Thm. 1]{KK}) showed that, for $X$ irreducible, $\bJbar^d_X$ is irreducible if and only if $X$  has locally planar singularities.
\item Kass proved in \cite[Thm. 2.7]{Kas3} that if $X$ is an integral curve with a unique non-Gorenstein singularity, then its compactified Jacobian $\bJbar_X^d$ (for any $d\in \Z$) contains an irreducible component $D_d$  which does not meet the open subset $\bJ_X^d\subset \bJbar_X^d$ of line bundles  and it is generically smooth of dimension $p_a(X)$. In particular, the smooth locus of $\bJbar_X^d$ is bigger then the locus $\bJ_X^d$ of line bundles.
\item Kass constructed in \cite{Kas4} an integral rational space curve $X$ of arithmetic genus $4$ for which $\bJbar_X$  is non-reduced.
\end{itemize}
\item \label{R:Jac-integral3} The statement about the embedded dimension in Theorem \ref{T:prop-J-planar} is sharp: if $X$ is a rational nodal curve  with $g$ nodes and $I\in \bJbar_X^d$ is a sheaf that is not locally free at any of the  $g$ nodes (and any $\bJbar_X^d$ contains a sheaf with these properties), then it is proved in \cite[Prop. 2.7]{CMK} that $\bJbar_X^d$ is isomorphic formal locally at $I$ to the product of $g$ nodes, hence it has embedded dimension at $I$ equal to $2g$.

\end{enumerate}
\end{remark}

\subsection{Fine compactified Jacobians}\label{S:fine-Jac}

For any  $\chi \in\mathbb Z$, the scheme $\bJbar_X^{\chi}$ is neither of finite type nor separated over $k$ (and similarly for $\bJ_X^{\chi}$) if $X$ is reducible.
However, they can be covered by
open subsets that are proper (and even projective) over $k$: the fine compactified Jacobians of $X$. The fine compactified Jacobians  depend on the choice of a polarization, whose definition is as
follows.

\begin{defi}\label{pola-def}
A \emph{polarization} on a connected curve $X$ is a tuple of rational numbers $\un q=\{\un q_{C_i}\}$, one for each irreducible component $C_i$ of $X$, such that $|\un q|:=\sum_i \un q_{C_i}\in \Z$.
We call $|\un q|$ the total degree of $\un q$.

\end{defi}

Given any subcurve $Y \subseteq X$, we set $\un{q}_Y:=\sum_j \un{q}_{C_j}$ where the sum runs
over all the irreducible components $C_j$ of $Y$.
Note that giving a polarization $\un q$ is the same as giving an
assignment $(Y\subseteq X)\mapsto \un q_Y$ such that $\un q_X\in \Z$ and which is additive on $Y$, i.e. such that if $Y_1,Y_2\subseteq X$ are two subcurves of $X$ without common irreducible components, then $\un q_{Y_1\cup Y_2}=\un q_{Y_1}+\un q_{Y_2}$.

\begin{defi}\label{def-int}
A polarization $\un q$ is called \emph{integral} at a subcurve $Y\subseteq X$ if
$\un q_Z \in \Z$ for any connected component $Z$ of $Y$ and of $Y^c$.

A polarization is called
\emph{general} if it is not integral at any proper subcurve $Y\subset X$.
\end{defi}

\begin{remark}\label{R:conn-pola}
It is easily seen that  $\un q$ is general if and only if $\un q_Y\not\in \Z$ for any proper subcurve $ Y\subset X$ such that $Y$ and $Y^c$ are connected.
\end{remark}

For each subcurve $Y$ of $X$ and each torsion-free sheaf $I$ on $X$, the restriction $I_{|Y}$ of $I$ to $Y$ is not necessarily a
torsion-free sheaf on $Y$. However, $I_{|Y}$ contains a biggest subsheaf, call it temporarily $J$, whose support has dimension zero, or in other words  such that $J$ is a torsion sheaf.
We denote by $I_{Y}$ the quotient of $I_{|Y}$
by $J$. It is easily seen that $I_Y$ is torsion-free on $Y$ and it is the biggest torsion-free quotient of $I_{|Y}$: it is actually the unique torsion-free quotient of $I$ whose support is equal to $Y$.
Moreover, if $I$ is torsion-free rank-1 then $I_Y$ is torsion-free rank-1.
We let $\deg_Y (I)$ denote the degree of $I_Y$ on $Y$, that is, $\deg_Y(I) := \chi(I_Y )-\chi(\O_Y)$.

\begin{defi}\label{sheaf-ss-qs}
\noindent Let $\un q$ be a polarization on $X$.
Let $I$ be a torsion-free rank-1  sheaf on $X$ (not necessarily simple) such that $\chi(I)=|\un q|$.
\begin{enumerate}[(i)]
\item \label{sheaf-ss} We say that $I$ is \emph{semistable} with respect to $\un q$ (or $\un q$-semistable) if
for every proper subcurve $Y\subset X$, we have that
\begin{equation}\label{multdeg-sh1}
\chi(I_Y)\geq \un q_Y.
\end{equation}
\item \label{sheaf-s} We say that $I$ is \emph{stable} with respect to $\un q$ (or $\un q$-stable) if it is semistable with respect to $\un q$
and if the inequality (\ref{multdeg-sh1}) is always strict.
\end{enumerate}
\end{defi}

\begin{remark}\label{R:conn-subcurves}
It is easily seen that a torsion-free rank-1 sheaf $I$ is  $\un q$-semistable (resp. $\un q$-stable) if and only if \eqref{multdeg-sh1} is satisfied (resp. is satisfied with strict inequality)
for any subcurve $Y\subset X$ such that  $Y$ and $Y^c$ are connected.

\end{remark}

\begin{remark}\label{R:stable-general}
Let $\un q$ be a polarization on $X$ and let $\un q'$ be a general polarization on $X$ that is obtained by slightly perturbing $\un q$. Then, for a torsion-free rank-$1$ sheaf $I$ on $X$,  we have the following chain of implications:
$$I  \: \text{ is } \un q\text{-stable} \Rightarrow I  \: \text{ is } \un q'\text{-stable}  \Rightarrow I  \: \text{ is } \un q'\text{-semistable}  \Rightarrow I  \: \text{ is } \un q\text{-semistable}. $$
\end{remark}

\begin{remark}\label{R:ss-lb}
A line bundle $L$ on $X$ is $\un q$-semistable if and only if
\begin{equation}\label{E:ss-upper}
\chi(L_{|Y})\leq \un q_Y+|Y\cap Y^c|
\end{equation}
for any subcurve $Y\subseteq X$.  Indeed, tensoring with $L$ the  exact sequence
$$0\to \O_X\to \O_{Y}\oplus \O_{Y^c} \to \O_{Y\cap Y^c}\to 0,$$
and taking Euler-Poincar\'e characteristics, we find that
$$\chi(L_{|Y})+\chi(L_{|Y^c})=\chi(L)+|Y\cap Y^c|.$$
Using this equality, we get that
$$\chi(L_{|Y^c})\geq \un q_{Y^c} \Longleftrightarrow \chi(L_{|Y})=\chi(L)-\chi(L_{|Y^c})+|Z\cap Z^c|\leq |\un q|-\un q_{Y^c}+|Z\cap Z^c|=\un q_Y+|Z\cap Z^c|,$$
which gives that \eqref{multdeg-sh1} for $Y^c$ is equivalent to \eqref{E:ss-upper} for $Y$.

\end{remark}

\begin{remark}\label{R:stabGor}
If $X$ is Gorenstein, we can write the inequality \eqref{multdeg-sh1}  in terms of the degree of $I_Y$ as follows
\begin{equation}\label{E:ineqdeg}
\deg_Y(I)\geq \un q_Y-\chi(\O_Y)=\un q_Y+\frac{\deg_Y(\omega_X)}{2}-\frac{\delta_Y}{2},
\end{equation}
where we used the adjunction formula (see \cite[Lemma 1.12]{Cat})
$$\deg_Y(\omega_X)=2p_a(Y)-2+\delta_Y=-2\chi(\O_Y)+\delta_Y.$$
The inequality \eqref{E:ineqdeg} was used to define stable rank-$1$ torsion-free sheaves on nodal curves in \cite{MV}; in particular, there is a change of notation between this paper where $\un q$-(semi)stability is defined by means of the inequality \eqref{multdeg-sh1} and the notation of loc. cit. where $\un q$-(semi)stability is defined by means of the inequality \eqref{E:ineqdeg}.

\end{remark}

Polarizations on $X$ can be constructed from  vector bundles on $X$, as we now indicate.

\begin{remark}\label{R:compEst}
Given a vector bundle $E$ on  $X$, we define the polarization $\un q^{E}$ on $X$ by setting
\begin{equation}\label{E:pol-qE}
\un{q}^E_{Y}:=-\frac{\deg(E_{|Y})}{\rk(E)},
\end{equation}
 for each subcurve $Y$ (or equivalently for each irreducible component $C_i$) of $X$.
Then  a torsion-free rank-1 sheaf $I$ on $X$ is stable (resp. semistable) with respect to $\un q^E$ in the sense of Definition \ref{sheaf-ss-qs} if and only if
$$\chi(I_Y)> ( \geq)\un q^E_Y=- \frac{\deg(E_{|Y})}{\rk(E)},$$
i.e. if $I$ is stable (resp. semistable) with respect to $E$ in the sense of \cite[Sec. 1.2]{est1}.

Moreover,  every polarization $\un q$ on $X$ is of the form $\un q^E$ for some (non-unique) vector bundle $E$. Indeed, take $r>0$ a sufficiently divisible natural number such that
$r\un q_Y\in \Z$ for every subcurve  $Y\subseteq X$.
Consider a vector bundle $E$ on $X$ of rank $r$ such that, for every subcurve $Y\subseteq X$ (or equivalently for every irreducible component $C_i$ of $X$), the degree of $E$ restricted to $Y$ is equal to
\begin{equation}\label{E:not-este}
-\deg(E_{|Y})=r \un q_{Y}.
\end{equation}
Then, comparing \eqref{E:pol-qE} and \eqref{E:not-este}, we deduce that $\un q^E=\un q$.

\end{remark}

Finally, for completeness, we mention that the usual slope (semi)stability with respect to some ample line bundle on $X$ is a special case of the above (semi)stability.

\begin{remark}\label{R:slope-stab}
Given an ample line bundle $L$ on  $X$ and an integer $\chi\in \bbZ$, the  slope (semi)stability for rank-$1$ torsion-free sheaves on $X$ of Euler-Poincar\'e characteristic equal to $\chi$ is equal to the above (semi)stability with respect to the polarization 
${}^L\un q$ defined by setting 
\begin{equation}\label{E:L-q}
{}^L\un q_Y:=\frac{\deg(L_{|Y})}{\deg L}\chi,
\end{equation}
for any subcurve $Y$ of $X$. 
The proof of the above equivalence in the nodal case can be found in \cite[Sec. 1]{Ale2} (see also \cite[Fact 2.8]{CMKV}); the same proof extends verbatim to arbitrary reduced curves. Notice that, as observed already in 
\cite[Rmk. 2.12(iv)]{MV},  slope semistability with respect to some ample line bundle $L$ is much more restrictive than $\un q$-semistability: the extreme case being when $\chi=0$, in which case there is a unique slope semistability 
(independent on the chosen line bundle $L$) while there are plenty of $\un q$-semistability conditions!
\end{remark}

The geometric implications of having a general polarization are clarified by the following result.


\begin{lemma}\label{L:nondeg}
Let $I$ be a rank-$1$ torsion-free sheaf on $X$ which is semistable with respect to a polarization $\un q$ on $X$.
\noindent
\begin{enumerate}[(i)]
\item \label{L:nondeg1} If $\un q$ is general then $I$ is also $\un q$-stable.
\item \label{L:nondeg2} If $I$ is $\un q$-stable, then $I$ is simple.
\end{enumerate}
\end{lemma}
\begin{proof}
Let us first prove \eqref{L:nondeg1}. Since $\un q$ is general, from Remark \ref{R:conn-pola} it follows  that if $Y\subset X$ is a subcurve of $X$ such that $Y$ and $Y^c$ are connected then
 $\displaystyle \un q_Y \not \in \Z$. Therefore, the right hand side of \eqref{multdeg-sh1} is not an integer for such subcurves, hence the inequality is a fortiori always strict.
 This is enough to guarantee that a torsion-free rank-1 sheaf that is $\un q$-semistable is also $\un q$-stable, by Remark  \ref{R:conn-subcurves}.

Let us now prove part \eqref{L:nondeg2}. By contradiction, suppose that $I$ is $\un q$-stable and not simple. Since $I$ is not simple, we can find, according to \cite[Prop. 1]{est1}, a proper subcurve $Y\subset X$
such that  the natural map $I\to I_{Y}\oplus I_{Y^c}$ is an isomorphism, which implies that $\chi(I)=\chi(I_Y)+\chi(I_{Y^c})$.
Since $I$ is $\un q$-stable, we get from \eqref{multdeg-sh1} the two inequalities
\begin{equation*}\label{E:2-ineq}
\begin{sis}
& \chi(I_Y)> \un q_Y,\\
& \chi(I_{Y^c})> \un q_{Y^c}.\\
\end{sis}
\end{equation*}
Summing up the above inequalities, we get $\chi(I)=\chi(I_Y)+\chi(I_{Y^c})>\un q_Y+\un q_{Y^c}=|\un q|$, which is a contradiction since $\chi(I)=|\un q|$ by definition of $\un q$-stability.
\end{proof}

Later on (see Lemma \ref{L:ndg-conv}), we will see that the property stated in Lemma \ref{L:nondeg}\eqref{L:nondeg1} characterizes the polarizations  that are general, at least for curves with locally planar singularities.

For a polarization $\un q$  on $X$, we will denote by $\ov{J}^{ss}_X(\un q)$ (resp. $\ov{J}^{s}_X(\un q)$) the subscheme of
$\bJbar_X$ parametrizing simple rank-1 torsion-free sheaves $I$ on $X$ which are $\un q$-semistable (resp. $\un q$-stable). If $\un q=\un q^E$ for some vector bundle $E$ on $X$, then it follows from Remark \ref{R:compEst}
that the subscheme $J_X^s(\un q^E)$ (resp. $J_X^{ss}(\un q^E)$) coincides with the subscheme $J_E^s$ (resp. $J_E^{ss}$) in Esteves's notation (see \cite[Sec. 4]{est1}).
By \cite[Prop. 34]{est1}, the inclusions
$$\ov{J}^{s}_X(\un q)\subseteq \ov{J}^{ss}_X(\un q)\subset \bJbar_X$$
are open.


\begin{fact}[Esteves]\label{F:Este-Jac}
Let $X$ be a connected curve.
\noindent
\begin{enumerate}[(i)]
\item $\ov{J}^{s}_X(\un q)$ is a quasi-projective scheme over $k$ (not necessarily reduced). In particular,  $\ov{J}^{s}_X(\un q)$ is a scheme of
finite type and separated over $k$.
\item $\ov{J}^{ss}_X(\un q)$ is a $k$-scheme of finite type and universally closed over $k$.
\item If $\un q$ is general then $\ov{J}^{ss}_X(\un q)=\ov{J}^{s}_X(\un q)$ is a projective scheme over $k$ (not necessarily reduced).
\item $\displaystyle \bJbar_X=\bigcup_{{\un q} \text{ general}} \ov{J}_X^s(\un q).$
\end{enumerate}
\end{fact}
\begin{proof}
Part (i) follows from \cite[Thm. A(1) and Thm. C(4)]{est1}.

Part (ii) follows from \cite[Thm. A(1)]{est1}.

Part (iii): the fact that $\ov{J}^{ss}_X(\un q)=\ov{J}^{s}_X(\un q)$ follows from Lemma \ref{L:nondeg}. Its projectivity follows from (i) and (ii) since a quasi-projective scheme over $k$ which is
universally closed over $k$ must be projective over $k$.

Part (iv) follows from \cite[Cor. 15]{est1}, which asserts that a simple torsion-free rank-1 sheaf is stable with respect to a certain polarization, together with Remark \ref{R:stable-general}, which asserts that it is enough to consider general polarizations.

\end{proof}

If $\un q$ is general, we set $\ov{J}_X(\un q):=\ov{J}^{ss}_X(\un q)=\ov{J}^{s}_X(\un q)$ and we call it the
\emph{fine compactified Jacobian} with respect to the polarization $\un q$. We denote by $J_X(\un q)$ the open subset of $\ov{J}_X(\un q)$ parametrizing
line bundles on $X$. Note that $J_X(\un q)$ is isomorphic to the disjoint union of a certain number of copies of the generalized Jacobian $J(X)=\Pic^{\un 0}(X)$ of $X$.

Since, for $\un q$ general, $J_X(\un q)$ is an open subset of $\bJbar_X$, the above Theorem \ref{T:prop-J-planar} immediately yields the following
properties for fine compactified Jacobians of curves with locally planar singularities.

\begin{cor}\label{C:prop-fineJac}
Let $X$ be a connected curve with locally planar singularities and $\un q$ a general polarization on $X$. Then
\begin{enumerate}[(i)]
\item \label{C:prop-fineJac1} $\ov{J}_X(\un q)$ is a reduced scheme with locally complete intersection singularities and embedded dimension at most $2p_a(X)$ at every point.
\item \label{C:prop-fineJac2} $J_X(\un q)$ is dense in $\ov{J}_X(\un q)$. In particular, $\ov{J}_X(\un q)$ has pure dimension equal to the arithmetic genus $p_a(X)$ of $X$.
\item \label{C:prop-fineJac3} $J_X(\un q)$ is the smooth locus of $\ov{J}_X(\un q)$.
\end{enumerate}
\end{cor}

Later, we will prove that $\ov{J}_X(\un q)$ is connected (see Corollary \ref{C:connect}) and we will give a formula for the number of its irreducible components in terms solely of the combinatorics of the curve $X$  (see Corollary \ref{C:irre-comp}).

\begin{remark}\label{R:Ses}
If $\un q$ is not general, it may happen that $\J_X^{ss}(\un q)$ is not separated. However, it follows from \cite[Thm. 15, p. 155]{Ses} that $\J_X^{ss}(\un q)$ admits a morphism $\phi: \J_X^{ss}(\un q)\to U_X(\un q)$ onto a projective variety that is universal with respect to maps into separated varieties; in other words, $U_X(\un q)$ is the biggest separated quotient of $\J_X^{ss}(\un q)$.
We call the projective variety $U_X(\un q)$ a \emph{coarse compactified Jacobian}. The fibers of $\phi$ are S-equivalence classes of sheaves, and in particular $\phi$ is an isomorphism on the open subset $\J_X^s(\un q)$ (see loc. cit. for details). Coarse compactified Jacobians can also be constructed as a special case of moduli spaces of semistable pure sheaves, constructed by  Simpson in \cite{Sim}.

Coarse compactified Jacobians behave quite differently from fine compactified Jacobians, even for a nodal curve $X$; for example
\begin{enumerate}[(i)]
\item \label{bad1} they can have (and typically they do have) fewer irreducible components than the number $c(X)$ of irreducible components of fine compactified Jacobians, see \cite[Thm. 7.1]{MV};
\item \label{bad2} their smooth locus can be bigger than the locus of line bundles, see \cite[Thm. B(ii)]{CMKV}.
\item \label{bad3} their embedded dimension at some point can be bigger than $2p_a(X)$, see \cite[Ex. 7.2]{CMKV}. 
\end{enumerate}

\end{remark}

\section{Varying the polarization}\label{S:vary}

Fine compactified Jacobians of a connected curve $X$ depend on the choice of a general polarization $\un q$. The goal of this section is to study the dependence of fine compactified Jacobians upon the choice of the polarization.
In particular, we will prove Theorem B, which says that there is always a finite number of isomorphism classes (resp. homeomorphism classes if $k=\bbC$) of fine compactified Jacobians of a reduced curve $X$ even though this number can be arbitrarily large even for nodal curves.

To this aim, consider the space of polarizations on $X$
\begin{equation}\label{E:polspa}
\Pol:=\{\un q\in \Q^{\gamma(X)}: |\un q|\in \Z\}\subset \R^{\gamma(X)},
\end{equation}
where $\gamma(X)$ is the number of irreducible components of $X$. Define the arrangement of hyperplanes in $\R^{\gamma(X)}$
\begin{equation}\label{E:arrang}
\Arr:=\left\{\sum_{C_i\subseteq Y} x_i =n\right\}_{Y\subseteq X, n\in \Z}
\end{equation}
where $Y$ varies among all the subcurves of $X$ such that $Y$ and $Y^c$ are connected. By Remark \ref{R:conn-pola}, a polarization $\un q\in \Pol$ is general if and only if $\un q$ does not belong to $\Arr$.
Moreover, the arrangement of hyperplanes $\Arr$ subdivides $\Pol$ into chambers with the following property: if two general polarizations $\un q$ and $\un q'$ belong to the same chamber then $\lceil \un q_Y \rceil=\lceil \un q'_Y \rceil$ for any subcurve $Y\subseteq X$ such that $Y$ and $Y^c$ are connected, hence $\J_X(\un q)=\J_X(\un q')$ by Remark \ref{R:conn-subcurves}. Therefore, fine compactified Jacobians of $X$
correspond bijectively to the chambers of $\Pol$ cut out by the hyperplane arrangement $\Arr$.

Obviously, there are infinitely many chambers and therefore infinitely many different fine compactified Jacobians.
 However,  we are now going to show that there are finitely many  isomorphism classes of fine compactified Jacobians.  The simplest way to show that two fine compactified Jacobians are isomorphic is to show that there is a translation that sends one into the other.

\begin{defi}\label{D:transla}
Let $X$ be a connected curve. We say that two compactified Jacobians $\ov{J}_X(\un q)$ and $\ov{J}_X(\un q')$ are \emph{equivalent by translation } if there exists  a line bundle
$L$ on $X$ inducing an isomorphism
$$\begin{aligned}
\ov{J}_X(\un q) & \stackrel{\cong}{\longrightarrow} \ov{J}_X(\un q'), \\
I & \mapsto I\otimes L.
\end{aligned}$$
\end{defi}

Note however that, in general,  there could be  fine compactified Jacobians that are isomorphic without being equivalent by translation, see
Section \ref{S:genus1} for some explicit examples.

\begin{prop}\label{P:finite-eq}
Let $X$ be a connected curve. There is a finite number of fine compactified Jacobians up to equivalence by translation. In particular, there is a finite number of isomorphism classes of fine compactified Jacobians of $X$.
\end{prop}
\begin{proof}
If two generic polarizations $\un q$ and $\un q'$ are such that $\un q-\un q'\in \Z^{\gamma(X)}$, then the multiplication by a line bundle of multidegree $\un q-\un q'$ gives an isomorphism
between $\ov{J}_X(\un q')$ and $\ov{J}_X(\un q)$.  Therefore, any fine compactified Jacobian of $X$ is equivalent by translation to a fine compactified Jacobian $\ov{J}_X(\un q)$
such that $0\leq \un q_{C_i}<1$ for any irreducible component $C_i$ of $X$.
We conclude by noticing that the arrangement of hyperplanes $\Arr$ of \eqref{E:arrang} subdivides   the unitary cube $[0,1)^{\gamma(X)}\subset \R^{\gamma(X)}$ into finitely many chambers.
\end{proof}

\subsection{Nodal curves}\label{S:nodal}

In this subsection, we study how fine compactified Jacobians vary for a nodal curve.

Recall that the generalized Jacobian $J(X)$ of a reduced curve $X$ acts, via tensor product, on any fine compactified Jacobian $\J_X(\un q)$ and the orbits of this action form a stratification of $\J_X(\un q)$ into locally closed subsets.
This stratification was studied in the case of nodal curves by the first and third authors in \cite{MV}.
In order to recall these results,  let us introduce some notation. Let $X_{\rm sing}$ be the set of nodes of $X$ and for every subset $S\subseteq X_{\rm sing}$ denote by $\nu_S:X_S\to X$ the partial normalization of $X$ at the nodes belonging to $S$.  For any subcurve $Y$ of $X$, set $Y_S:=\nu_S^{-1}(Y)$. Note that $Y_S$ is the partial normalization of $Y$ at the nodes $S\cap Y_{\rm sing}$ and that every subcurve of $X_S$ is of the form $Y_S$ for some uniquely determined subcurve $Y\subseteq X$. Given a polarization $\un q$ on $X$,  define a polarization $\un q^S$ on $X_S$ by setting $\un q^S_{Y_S}:=\un q_Y$ for any subcurve $Y$ of $X$. Clearly, if $\un q$ is a general polarization on $X$ then $\un q^S$ is a general polarization on $X_S$. Moreover, consider the following subset of integral multidegrees on $X_S$:
$$B_S(\un q):=\{\un{\chi} \in \Z^{\gamma(X_S)}\: : \: |\un \chi|=|\un q^S|, \: \un \chi_{Y_S}\geq \un q_{Y_S} \text{ for any subcurve } Y\subseteq X\},$$
and for every $\un \chi\in B_S(\un q)$ denote by $J_{X_S}^{\un \chi}$ the $J(X_S)$-torsor consisting of all the line bundles $L$ on $X_S$ whose multi-Euler characteristic is equal to $\un \chi$, i.e. $\chi(L_{|Y_S})=\un \chi_{Y_S}$ for every subcurve $Y_S\subseteq X_S$.

\begin{fact}\label{F:strat}
Let $X$ be a connected nodal curve of arithmetic genus $p_a(X)=g$ and let $\J_X(\un q)$ be a fine compactified Jacobian of $X$.
\begin{enumerate}[(i)]
\item \label{F:strat1} For every $S\subseteq X_{\rm sing}$, denote by $J_{X, S}(\un q)$ the locally closed subset (with reduced scheme structure) of $\J_X(\un q)$ consisting of all the sheaves $I\in \J_X(\un q)$ such that $I$ is not locally free exactly at the nodes of $S$.  Then
\begin{enumerate}[(a)]
\item $J_{X, S}(\un q)\neq \emptyset$ if and only if $X_S$ is connected;
\item $\ov{J_{X, S}(\un q)}=\coprod_{S\subseteq S'} J_{X,S'}(\un q)$.
\end{enumerate}
\item \label{F:strat2} The pushforward ${\nu_S}_*$ along the normalization morphism $\nu_S:X_S\to X$ gives isomorphisms
$$
\begin{sis}
J_{X,S}(\un q)\cong& J_{X_S}(\un q^S)=\coprod_{{\un \chi} \in B_S(\un q)} J_{X_S}^{\un \chi}, \\
\ov{J_{X,S}(\un q)} \cong & \J_{X_S}(\un q^S).
\end{sis}
$$
\item \label{F:strat3} The decomposition of $\ov J_X(\un q)$ into orbits for the action of the generalized Jacobian $J(X)$ is equal to
$$ \J_X(\un q)=\coprod_{\stackrel{S\subseteq X_{\rm sing}}{\un \chi\in B_S(\un q)}} J_{X_S}^{\un \chi},$$
where the disjoint union runs over the subsets $S\subseteq X_{\rm sing}$ such that $X_S$ is connected.
\item \label{F:strat4} For every $I\in J_{X, S}(\un q)$, the completion of the local ring $\O_{\J_X(\un q), I}$ of the fine compactified Jacobian $\J_X(\un q)$ at $I$ is given by
$$\widehat{\O}_{\J_X(\un q), I}=k[[Z_1,\ldots,Z_{g-|S|}]] \widehat{\bigotimes_{1\leq i\leq |S|}} \frac{k[[ X_i, Y_i ]]}{( X_i Y_i )}.$$

\end{enumerate}
\end{fact}
\begin{proof}
Parts \eqref{F:strat1}, \eqref{F:strat2} and \eqref{F:strat3}  follow from \cite[Thm. 5.1]{MV} keeping in mind the change of notation of this paper (where we use the Euler characteristic) with respect to the notation of \cite{MV} (where the degree is used), see Remark \ref{R:stabGor}.

Part \eqref{F:strat4} is a special case of \cite[Thm. A]{CMKV} (see in particular \cite[Example 7.1]{CMKV}), where the local structure of (possibly non fine) compactified Jacobians of nodal curves is described. 
\end{proof}

The set of $J(X)$-orbits $\mathbb O(\J_X(\un q)):=\{J_{X_S}^{\un \chi}\}$ on $\J_X(\un q)$ forms naturally a poset (called the \emph{poset of orbits}  of $\J_X(\un q)$) by declaring that $J_{X_S}^{\un \chi}\geq J_{X_{S'}}^{\un \chi'}$ if and only if $\ov{J_{X_S}^{\un \chi}}\supseteq J_{X_{S'}}^{\un \chi'}$.
Observe that the generalized Jacobian $J(X)$ acts via tensor product on any coarse compactified Jacobian $U_X(\un q)$ (defined as in Remark \ref{R:Ses}) and hence we can define the poset of orbits of $U_X(\un q)$. However, the explicit description of Fact \ref{F:strat} fails for non fine compactified Jacobians.  

Clearly, the poset of orbits is an invariant of the fine compactified Jacobian endowed with the action of the generalized Jacobian. We will now give another description of the poset of orbits of $\J_X(\un q)$ in terms solely of the singularities of the variety
$\J_X(\un q)$ without any reference to the action of the generalized Jacobian.
With this in mind, consider a $k$-variety $V$, i.e. a reduced scheme of finite type over $k$. Define inductively a finite chain of closed subsets $\emptyset=V^{r+1}\subset V^r\subset \ldots \subset V^1\subset V^0=V$ by setting $V^i$
equal to the singular locus of $V^{i-1}$ endowed with the reduced scheme structure. The loci $V^i_{\rm reg}:=V^i\setminus V^{i+1}$, consisting of smooth points of $V^i$, form a partition of $V$ into locally closed subsets.
We define the \emph{singular poset}  of $V$, denoted by $\Sigma(V)$,  as the set of irreducible components of $V^i_{\rm reg}$ for $0\leq i\leq r$, endowed with the poset structure defined by setting $C_1\geq C_2$ if and only if $\ov{C_1}\supseteq C_2$.


\begin{prop}\label{P:inv-orb}
Let $X$ be a connected nodal curve and let $\un q$ be a general polarization on $X$. Then the poset of orbits $\mathbb O(\J_X(\un q))$  is isomorphic to the singular poset $\Sigma(\J_X(\un q))$.

In particular,  if $\J_X(\un q)\cong \J_X(\un q')$ for two general polarizations $\un q, \un q'$ on $X$ then $\mathbb O(\J_X(\un q))\cong \mathbb O(\J_X(\un q'))$.
\end{prop}

\begin{proof}
According to Corollary \ref{C:prop-fineJac}\eqref{C:prop-fineJac3}, the smooth locus of $\J_X(\un q)$ is the locus $J_X(\un q)$ of line bundles; therefore, using Fact \ref{F:strat},  the singular locus of $\J_X(\un q)$ is equal to
$$\J_X(\un q)^1=\coprod_{\emptyset \neq S \subseteq X_{\rm sing}} J_{X, S}(\un q)=\bigcup_{|S|=1} \ov{J_{X,S}(\un q)}\cong \bigcup_{|S|=1} \J_{X_S}(\un q^S).
$$
Applying again Corollary \ref{C:prop-fineJac}\eqref{C:prop-fineJac3} and Fact \ref{F:strat} and proceeding inductively, we get that
$$\J_X(\un q)^i=\coprod_{|S|\geq i } J_{X, S}(\un q)=\bigcup_{|S|=i} \ov{J_{X,S}(\un q)}\cong \bigcup_{|S|=i} \J_{X_S}(\un q^S).
$$
Therefore the smooth locus of $J_X(\un q)^i$ is equal to
$$\J_X(\un q)^i_{\rm reg}=\J_X(\un q)^i\setminus \J_X(\un q)^{i+1}=\coprod_{|S|=i } J_{X, S}(\un q)=\coprod_{\stackrel{|S|=i}{\un \chi\in B_S(\un q)}} J_{X_S}^{\un \chi}.$$
Since each subset $J_{X_S}^{\un \chi}$ is irreducible, being a $J(X_S)$-torsor, we deduce that the singular poset of $\J_X(\un q)$ is equal to its poset of orbits, q.e.d.
\end{proof}

Moreover, as we will show in the next proposition, if our base field $k$ is the field $\bbC$ of complex numbers, then the poset of orbits of a fine compactified Jacobian $\J_X(\un q)$ is a topological invariant of the  analytic space $\J_X(\un q)^{\rm an}$
associated to $\J_X(\un q)$, endowed  with the Euclidean topology.

\begin{prop}\label{P:inv-top}
Let $X$ be a connected nodal curve of arithmetic genus $g=p_a(X)$ and let $\un q, \un q'$  be two general polarizations on $X$. If $\J_X(\un q)^{\rm an}$ and $\J_X(\un q')^{\rm an}$ are homeomorphic then  $\mathbb O(\J_X(\un q))\cong \mathbb O(\J_X(\un q'))$.
\end{prop}
\begin{proof}
Let $\psi: \J_X(\un q)^{\rm an} \stackrel{\cong}{\longrightarrow} \J_X(\un q')^{\rm an}$ be a homeomorphism. Consider a sheaf $I\in J_{X,S}(\un q)^{\rm an}$ for some $S\subseteq X_{\rm sing}$  and denote by $S'$ the unique subset of $X_{\rm sing}$ such that $\psi(I)\in J_{X,S'}(\un q')^{\rm an}$.
Fact \ref{F:strat}\eqref{F:strat4} implies that  $\J_X(\un q)^{\rm an}$ (resp. $\J_X(\un q')^{\rm an}$) is locally (analytically) isomorphic at $I$ (resp. at $\psi(I)$) to  the complex analytic space given by the product of $|S|$ (resp. $|S'|$) nodes with a smooth variety of dimension $g-|S|$excision
(resp. $g-|S'|$). Using excision and Lemma \ref{L:lemtop}, we get that
$$\begin{sis}
& \dim_{\bbQ} H^{2g}(\J_X(\un q)^{\rm an}, \J_X(\un q)^{\rm an}\setminus \{I\}, \bbQ)=2^{|S|},\\
& \dim_{\bbQ} H^{2g}(\J_X(\un q')^{\rm an}, \J_X(\un q')^{\rm an}\setminus \{\psi(I)\}, \bbQ)=2^{|S'|}.\\
\end{sis}$$
Since $\psi$ is a homeomorphism, we conclude that $|S|=|S'|$, or in other words that
$$I\in \coprod_{|S|=i} J_{X,S}(\un q)^{\rm an} \: \text{ for some } i\geq 0 \Rightarrow  \psi(I)\in  \coprod_{|S|=i} J_{X,S}(\un q')^{\rm an}.$$
Therefore, the map $\psi$ induces a homeomorphism between $\displaystyle \coprod_{|S|=i} J_{X,S}(\un q)^{\rm an}\subseteq \J_X(\un q)$ and
$\displaystyle \coprod_{|S|=i} J_{X,S}(\un q')^{\rm an}\subseteq \J_X(\un q')$ for any $i\geq 0$.
Fact \ref{F:strat}\eqref{F:strat2} implies that we have the following decompositions into connected components
$$\coprod_{|S|=i } J_{X, S}(\un q)^{\rm an}=\coprod_{\stackrel{|S|=i}{\un \chi\in B_S(\un q)}} (J_{X_S}^{\un \chi})^{\rm an} \: \: \text{ and } \: \:
\coprod_{|S|=i } J_{X, S}(\un q')^{\rm an}=\coprod_{\stackrel{|S|=i}{\un \chi\in B_S(\un q')}} (J_{X_S}^{\un \chi})^{\rm an}.$$
Hence, $\psi$ induces a bijection $\psi_*: \mathbb O(\J_X(\un q))\stackrel{\cong}{\longrightarrow} \mathbb O(\J_X(\un q'))$ between the strata of $\J_X(\un q)$ and the strata of $\J_X(\un q')$ with the property that each stratum $(J_{X_S}^{\un \chi})^{\rm an}$
of $\J_X(\un q)^{\rm an}$ is sent homeomorphically by $\psi$ onto the stratum $\psi_*(J_{X_S}^{\un \chi})^{\rm an}$ of $\J_X(\un q')^{\rm an}$.
Therefore, the bijection $\psi_*$ is also an isomorphism of posets, q.e.d.
\end{proof}

\begin{lemma}\label{L:lemtop}
Let $V$ be the complex subvariety of $\bbC^{2k+n-k}$ of equations $x_1x_2=x_3x_4=\ldots=x_{2k-1}x_{2k}=0$, for some $0\leq k\leq n$.
Then $\dim_{\bbQ} H^{2n}(V,V\setminus \{0\}, \bbQ)=2^k$.
\end{lemma}
\begin{proof}
Since $V$ is contractible, by homotopical invariance of the cohomology of groups we have that
\begin{equation}\label{E:red-link}
H^{2n}(V,V\setminus \{0\}, \bbQ)=H^{2n-1}(L, \bbQ)=H_{2n-1}(L,\bbQ)^{\vee},
\end{equation}
where $L$ is the link of the origin $0$ in $V$, i.e. the intersection of $V$ with a small sphere of $\bbC^{2k+n-k}$ centered at $0$.
Observe that $V$ is the union of $2^k$ vector subspaces of $\bbC^{2k+n-k}$ of dimension $n$:
$$V_{\epsilon_{\bullet}}=\langle e_{1+\epsilon_1}, e_{3+\epsilon_2}, \ldots, e_{2k-1+\epsilon_k}, e_{2k+1}, \ldots, e_{2k+n-k}\rangle \: \text{ where } \: \epsilon_{\bullet}=(\epsilon_1,\ldots, \epsilon_k)\in \{0,1\}^k, $$
 which intersect along vector subspaces of dimension less than or equal to $n-1$. It follows that $L$ is the union of $2^k$ spheres $\{S_1,\ldots, S_{2^k}\}$ of dimension $2n-1$ which intersect along spheres of dimension less than or equal to $2n-3$.
 Fix a triangulation of $L$ that induces a triangulation of each sphere $S_i$ and of their pairwise intersections. Consider the natural map
$$\eta: \bbQ^{2^k}\cong \oplus_{i=1}^{2^k} H_{2n-1}(S_i,\bbQ)\longrightarrow H_{2n-1}(L,\bbQ).$$
Since there are no simplices in $L$ of dimension greater than $2n-1$, the map $\eta$ is injective. Moreover, using the fact that the spheres $S_i$ only intersect along spheres of dimension less that or equal to $2n-3$, we can prove that $\eta$ is surjective. Indeed, let $C\in Z_{2n-1}(L,\bbQ)$ be a cycle
in $L$ of dimension $2n-1$, i.e. a simplicial $(2n-1)$-chain whose boundary $\partial(C)$ vanishes. For every $1\leq i\leq 2^k$, let $C_i$ be the chain obtained from $C$ by erasing all the simplices that are not contained in the sphere $S_i$.  By construction, we have that $C=\sum_{i=1}^{2^k} C_i$.
Therefore,  using that $C$ is a cycle, we get that (for every $i$)
$$\partial(C_i)=-\sum_{j\neq i} \partial(C_j).$$
Observe now that $\partial(C_i)$ is a $(2n-2)$-chain contained in $S_i$ while $\sum_{j\neq i}\partial(C_j)$ is a $(2n-2)$-chain contained in $\bigcup_{j\neq i} S_j$. Since $S_i$ intersects each  $S_j$ with $j\neq i$ in spheres of dimension less than or equal to $(2n-3)$, we conclude that $\partial(C_i)=0$, or in other words that $C_i\in Z_{2n-1}(S_i,\bbQ)$. Therefore, we get that $\eta\left(\sum_{i=1}^{2^k}[C_i]\right)=[C]$, which shows that $\eta$ is surjective.

The assertion now follows from the equality \eqref{E:red-link} together with the fact that $\eta$ is an isomorphism.
\end{proof}

\begin{remark}\label{R:suff-crit}
We do not know of any example of two non isomorphic (or non homeomorphic if $k=\bbC$) fine compactified Jacobians having isomorphic posets of orbits; therefore, we wonder if the converse of the second assertion of Proposition  \ref{P:inv-orb} 
or the converse of Proposition \ref{P:inv-top}  might hold true.
\end{remark}

The poset of orbits of a fine compactified Jacobian $\J_X(\un q)$ (or more generally of any coarse compactified Jacobian $U_X(\un q)$) is isomorphic to the poset of regions of a certain toric arrangement of hyperplanes, as we now explain.
Let $\Gamma_X$ be the (connected) dual graph of the nodal curve $X$, i.e. the graph whose vertices $V(\Gamma_X)$ correspond to irreducible components of $X$ and whose edges $E(\Gamma_X)$ correspond to nodes of $X$: an edge being incident to a vertex if the node corresponding to the former belongs to the  irreducible component corresponding to the latter.
We fix an orientation of $\Gamma_X$, i.e. we specify the source and target $s,t:E(\Gamma_X)\to V(\Gamma_X)$ of each edge of $\Gamma_X$. The first homology group $H_1(\Gamma_X,A)$ of the graph $\Gamma_X$ with coefficients in a commutative ring with unit $A$ (e.g. $A=\Z, \Q, \R$) is the kernel of the boundary morphism
\begin{equation}\label{E:bounda}
\begin{aligned}
\partial: C_1(\Gamma_X,A)=\bigoplus_{e\in E(\Gamma_X)} A\cdot e  & \longrightarrow C_0(\Gamma_X,A)=\bigoplus_{v\in V(\Gamma_X)} A\cdot v,\\
e & \mapsto t(e)-s(e).
\end{aligned}
\end{equation}
The map $\partial$ depends upon the choice of the orientation of $\Gamma_X$; however, $H_1(\Gamma_X,A)$ does not depend upon the chosen orientation.
Since the graph $\Gamma_X$ is connected, the image of the boundary map $\partial$ is the subgroup
$$C_0(\Gamma_X,A)_0:=\left\{\sum_{v\in V(\Gamma_X)} a_v \cdot v\: : \: \sum_{v\in V(\Gamma_X)} a_v=0\right\}\subset C_0(\Gamma_X,A). $$
When $A=\Q$ or $\R$, we can endow the vector space $C_1(\Gamma_X,A)$ with a non-degenerate bilinear form $(,)$ defined by requiring that
$$(e,f)=
\begin{cases}
0 & \text{ if } e\neq f, \\
1 & \text{ if } e=f,
\end{cases}
$$
for any $e,f\in E(\Gamma_X)$. Denoting by $H_1(\Gamma_X,A)^{\perp}$ the subspace of $C_1(\Gamma_X,A)$ perpendicular to $H_1(\Gamma_X, A)$, we have that the boundary map induces an isomorphism of vector spaces
\begin{equation}\label{E:iso-vs}
\partial: H_1(\Gamma_X, A)^{\perp}\stackrel{\cong}{\longrightarrow} C_0(\Gamma_X,A)_0.
\end{equation}

Let now $\un q$ be a polarization of total degree $|\un q|=1-p_a(X)=1-g$ and consider the element
$$\phi:=\sum_{v\in V(\Gamma_X)} \phi_v \cdot v= \sum_{v\in V(\Gamma_X)} \left(\un q_{Y_v}+\frac{\deg_{Y_v}(\omega_X)}{2}\right)\cdot v\in C_0(\Gamma_X,\Q)_0,$$
where $Y_v$ is the irreducible component of $X$ corresponding to the vertex $v\in V(\Gamma_X)$. Using the isomorphism \eqref{E:iso-vs}, we can find a unique element
$\psi=\sum_{e\in E(\Gamma_X)} \psi_e \cdot e \in H_1(\Gamma_X,\Q)^{\perp}$ such that $\partial(\psi)=\phi$.  
Consider now the arrangement of affine hyperplanes in $H_1(\Gamma,\R)$ given by
\begin{equation}\label{E:arrang-q}
\V_{\un q}:=\left\{e^*=n+\frac{1}{2}-\psi_e\right\}_{n\in \Z, e\in E(\Gamma_X)}
\end{equation}
where $e^*$ is the functional on $C_1(\Gamma,\R)$ (hence on $H_1(\Gamma_X, \R)$ by restriction) given by $e^*=(e,-)$.
 The arrangement of hyperplanes $\V_{\un q}$ is periodic with respect to the action of  $H_1(\Gamma_X,\Z)$ on $H_1(\Gamma,\R)$; hence, it induces an arrangement of hyperplanes in the real torus
 $\displaystyle \frac{H_1(\Gamma_X,\R)}{H_1(\Gamma_X,\Z)}$, which we will still denote by $\V_{\un q}$ and we will call the toric arrangement of hyperplanes associated to $\un q$.
  The toric arrangement $\V_{\un q}$ of hyperplanes subdivides  the real torus $\displaystyle \frac{H_1(\Gamma_X,\R)}{H_1(\Gamma_X,\Z)}$ into finitely many regions, which form naturally a partially ordered set (poset for short) under the natural containment relation. This poset is related to the coarse compactified Jacobian $U_X(\un q)$ as follows.

\begin{fact}[Oda-Seshadri]\label{F:OS}
Let $\un q$ be a polarization of total degree $|\un q|=1-p_a(X)=1-g$ on a connected nodal curve $X$.
The poset of regions cut out by the toric arrangement of hyperplanes $\V_{\un q}$
is isomorphic to the poset $\mathbb O(\J_X(\un q))$ of $J(X)$-orbits on $U_X(\un q)$.
\end{fact}
\begin{proof}
See \cite{OS} or \cite[Thm. 2.9]{Ale2}.
\end{proof}

The arrangement of hyperplanes $\V_{\un q}$ determines wether the polarization $\un q$ on $X$ is generic or not, at least if $X$ does not have separating nodes, i.e. nodes whose removal disconnects the curve.
Recall that a toric arrangement of hyperplanes is said to be \emph{simple}  if the intersection of $r$ non-trivial hyperplanes in the given arrangement  has codimension at least $r$.
Moreover, following \cite[Def. 2.8]{MV}, we say that a polarization $\un q$ is \emph{non-degenerate}  if and only if $\un q$ is not integral at any proper subcurve $Y\subset X$ such that $Y$ intersects $Y^c$ in at least one non-separating node.
Note that a general polarization on a nodal curve $X$ is non-degenerate and that the converse is true if $X$ does not have separating nodes.

\begin{lemma}\label{L:simple-ar}
Let $\un q$ be a polarization of total degree $|\un q|=1-p_a(X)=1-g$  on a connected nodal curve $X$.
\begin{enumerate}[(i)]
\item \label{L:simple-ar1} $\V_{\un q}$ is simple if and only if $\un q$ is non-degenerate.
In particular, if $\un q$ is general then $\V_{\un q}$ is simple, and the converse is true if $X$ does not have separating nodes.
\item \label{L:simple-ar2} If $\V_{\un q}$ is simple then we can find a general polarization $\un q'$ such that $\V_{\un q}$ and $\V_{\un q'}$ have isomorphic poset of regions.
\end{enumerate}
\end{lemma}
\begin{proof}
By \cite[Thm. 7.1]{MV}, $\un q$ is non-degenerate if and only if the number of irreducible components of the compactified Jacobian $U_X(\un q)$ is the maximum possible which is indeed equal to the complexity $c(X)$ of the curve $X$, i.e. the number of spanning  trees of the dual graph $\Gamma_X$ of $X$. By Fact \ref{F:OS}, this happens if and only if the number of full-dimensional regions cut out by the toric arrangement $\V_{\un q}$  is as big as possible. This is equivalent, in turn, to the fact that
$\V_{\un q}$ is simple, which concludes the proof of \eqref{L:simple-ar1}.

Now suppose that $\V_{\un q}$ is simple. Then there exists a small Eucliden open neighborhood $U$ of $\un q$ in the space $\P_X$ of polarizations (see \eqref{E:polspa})  such that for every $\un q'\in U$ the toric arrangement $\V_{\un q'}$ of hyperplanes has its poset of regions isomorphic to the poset of regions of $\V_{\un q}$. Clearly, such an open subset $U$ will contain a point $\un q'$ not belonging to the arrangement of hyperplanes $\A_X$ defined in \eqref{E:arrang};  any such point
$\un q'$ will satisfy the conclusions of part \eqref{L:simple-ar2}.
\end{proof}

Although Fact \ref{F:OS} and Lemma \ref{L:simple-ar} are only stated for (fine) compactified Jacobians of total degree $1-p_a(X)$, they can be easily extended to any compactified Jacobian since any (fine) compactified Jacobian of a curve $X$ is
equivalent by translation to a (fine) compactified Jacobian of total degree $1-p_a(X)$ (although not to a unique one).
In particular, combining Propositions \ref{P:inv-orb} and \ref{P:inv-top}, Fact \ref{F:OS} and Lemma \ref{L:simple-ar}, we get the following lower bound for the number of non isomorphic (resp. non homeomorphic if $k=\bbC$) fine compactified Jacobians of a nodal curve $X$.

\begin{cor}\label{C:crit-iso}
Let $X$ be a connected nodal curve. Then the number of non isomorphic (resp. non homeomorphic if $k=\bbC$) fine compactified Jacobians of $X$ is bounded from below by the number of simple toric arrangement of hyperplanes
of the form $\V_{\un q}$ whose posets of regions are pairwise non isomorphic.
\end{cor}

We end this section by giving a sequence  of nodal curves $\{X_n\}_{n\in \bbN}$ of genus two such that the number of simple toric arrangements of hyperplanes $\{\V_{\un q}\}_{\un q\in \P_X}$ having pairwise non isomorphic posets of regions becomes arbitrarily large as $n$ goes to infinity; this implies, by Corollary \ref{C:crit-iso}, that the number of non isomorphic (resp. non homeomorphic if $k=\bbC$) fine compactified Jacobians can be arbitrarily large even for nodal curves, thus completing the proof of Theorem B from the introduction.

\begin{example}\label{E:non-iso}

Consider a genus-$2$ curve $X=X_1$ obtained from a dollar sign curve blowing up two of its 3 nodes. Then the dual graph $\Gamma_X$ of $X$ is as follows:

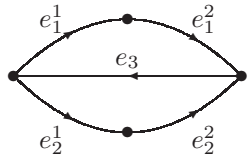
\unitlength 0.5mm 
\linethickness{0.4pt}
\ifx\plotpoint\undefined\newsavebox{\plotpoint}\fi 
\begin{picture}(130,45)(10,89)
\qbezier(130,110)(160,140)(190,110)
\put(146,122){\vector(3,2){.07}}
\put(179,119){\vector(3,-2){.07}}
\put(160,110){\vector(-1,0){.07}}
\put(146,98){\vector(3,-2){.07}}
\put(179,101){\vector(3,2){.07}}
\qbezier(130,110)(160,80)(190,110)
\qbezier(130,110)(160,110)(190,110)
\put(160,125){\circle*{3}}
\put(160,95){\circle*{3}}
\put(130,110){\circle*{3}}
\put(190,110){\circle*{3}}
\put(140,125){\makebox(0,0)[cc]{$e_1^1$}}
\put(180,125){\makebox(0,0)[cc]{$e_1^2$}}
\put(140,93){\makebox(0,0)[cc]{$e_2^1$}}
\put(180,93){\makebox(0,0)[cc]{$e_2^2$}}
\put(160,114){\makebox(0,0)[cc]{$e_3$}}
\end{picture}

Using the orientation depicted in  the above Figure, a basis for $H_1(\Gamma_X,\Z)\cong \Z^2$ is given by
$x:=e_1^1+e_1^2+e_3$ and $y:=e_2^1+e_2^2+e_3$. Therefore, the functionals on $H_1(\Gamma_X, \R)\cong \R^2$ associated to the edges of $\Gamma_X$ are given, in the above basis, by
$$(e_1^1)^*=(e_1^2)^*=x^* \hspace{0.5cm} (e_2^1)^*=(e_2^2)^*=y^* \hspace{0.5cm} e_3^*=x^*+y^*.$$
Then each polarization $\un q$ on $X$ of total degree $|\un q|=1-p_a(X)=-1$ gives rise to a toric arrangement $\V_{\un q}$ of $5$ lines in $\displaystyle \frac{H_1(\Gamma_X,\R)}{H_1(\Gamma_X,\Z)}\cong  \frac{\R^2}{\mathbb Z^2}$
of the form
$$ \V_{\eta_{\bullet}}=\left\{x=\eta_{e_1^1},\:  x=\eta_{e_1^2},\: y=\eta_{e_2^1},\: y=\eta_{e_2^2},\: x+y=\eta_{e_3} \right\}$$
for some rational numbers $\eta_{\bullet}$ which are determined by the polarization $\un q$, as explained in \eqref{E:arrang-q} above.
Conversely, given such a toric arrangement $\V_{\eta_{\bullet}}$ of $5$ lines in $\mathbb R^2/\mathbb Z^2$, there is a polarization $\un q$ on $X$ of total degree $|\un q|=1-p_a(X)=-1$ such that $\mathcal V_{\un q}=\V_{\eta_{\bullet}}$. Moreover, according to Lemma \ref{L:simple-ar}, the polarization $\un q$ is general in $X$ if and only if the arrangement $\V_{\un q}=\V_{\eta_{\bullet}}$ is simple.
Consider the following two simple toric arrangements of $5$ lines that are drawn on the unit square of $\R^2$ (two of the lines correspond to the edges of the unit square):

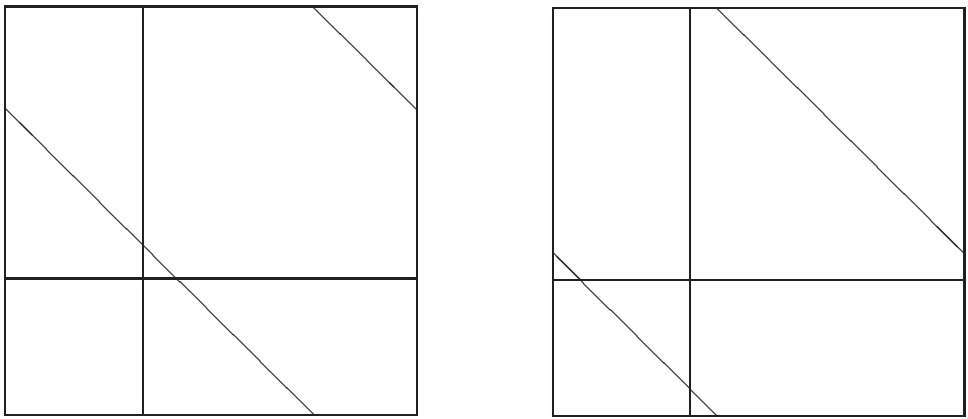
\unitlength 0.45mm 
\linethickness{0.4pt}
\ifx\plotpoint\undefined\newsavebox{\plotpoint}\fi 
\begin{picture}(110,125)(-20,102)
\put(0,100){\framebox(120,120)[cc]{}}
\put(40,220){\line(0,-1){120}}
\put(0,140){\line(1,0){120}}
\put(90,100){\line(-1,1){90}}
\put(90,220){\line(1,-1){30}}
\end{picture}

\begin{picture}(110,0)(-180,93)
\put(0,100){\framebox(120,120)[cc]{}}
\put(40,220){\line(0,-1){120}}
\put(0,140){\line(1,0){120}}
\put(48,100){\line(-1,1){48}}
\put(48,220){\line(1,-1){72}}
\end{picture}

Then it is easy to check that the poset of regions of the two toric arrangements are not isomorphic: it suffices to note that on the one on the left there are $2$ triangular two dimensional regions while on the one on the right there are $4$ triangular two dimensional regions. According to Corollary \ref{C:crit-iso}, this implies that there are at least two generic polarizations on $X$, $\un q$ and $\un q'$, such that $\J_X(\un q)$ and $\J_X(\un q')$ are not isomorphic.

More generally, blow up $X$ further in order to obtain a genus-$2$ curve $X_n$ whose dual graph $\Gamma_{X_n}$ is as follows:

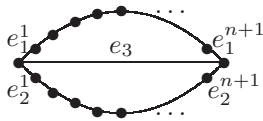
\begin{picture}(130,40)(10,87)
\qbezier(130,110)(160,140)(190,110)
\qbezier(130,110)(160,80)(190,110)
\qbezier(130,110)(160,110)(190,110)
\put(160,125){\circle*{3}}
\put(160,95){\circle*{3}}
\put(135,114){\circle*{3}}
\put(140,118){\circle*{3}}
\put(146.5,122){\circle*{3}}
\put(153,124.25){\circle*{3}}
\put(175,125){\makebox(0,0)[cc]{$\dots$}}
\put(193,117){\makebox(0,0)[cc]{ $e_1^{n+1}$}}
\put(185,114){\circle*{3}}
\put(130,110){\circle*{3}}
\put(190,110){\circle*{3}}
\put(130,117){\makebox(0,0)[cc]{$e_1^1$}}
\put(135,105.5){\circle*{3}}
\put(140,101){\circle*{3}}
\put(146.5,98){\circle*{3}}
\put(153,95.25){\circle*{3}}
\put(185,105.5){\circle*{3}}
\put(130,102){\makebox(0,0)[cc]{$e_2^1$}}
\put(193,102){\makebox(0,0)[cc]{$e_2^{n+1}$}}
\put(160,114){\makebox(0,0)[cc]{$e_3$}}
\put(175,95){\makebox(0,0)[cc]{$\dots$}}
\end{picture}

\noindent In words, $X_n$ is obtained from the dollar sign curve by blowing up two of its nodes $n$ times. Arguing as above, the (simple) toric arrangements associated to the (general) polarizations $\un q$ on $X_n$ of total degree $|\un q|=1-p_a(X_n)=-1$ will be formed by $2n+3$ lines in $\mathbb R^2/\mathbb Z^2$ of the form
$$ \V_{\eta_{\bullet}}=\left\{x=\eta_{e_1^i},\:   y=\eta_{e_2^j},\: x+y=\eta_{e_3} \right\}_{1\leq i, j\leq n+1}$$
for some rational numbers $\eta_{\bullet}$ which depend on $\un q$.  For every $\frac{n}{2}< i\leq n$, consider two simple toric arrangements of hyperplanes of $\R^2/\Z^2$
$$\V_i^+:=\left\{x=\frac{h}{3n}, \: y=\frac{k}{3n}, \: x+y=\frac{2i}{3n}+
\epsilon\right\}_{0\leq h, k\leq n}\hspace{0.2cm}
\V_i^-:=\left\{x=\frac{h}{3n}, \: y=\frac{k}{3n}, \: x+y=\frac{2i}{3n}-
\epsilon\right\}_{0\leq h, k\leq n}
$$
where $\epsilon$ is a sufficiently small rational number (the poset of regions of the above toric hyperplane arrangements do not actually depend on the chosen small value of $\epsilon$).
In the next figure we have represented on the unit square in $\R^2$ the toric arrangement $\V_i^+$ on the left and the toric arrangement $\V_i^-$ on the right.


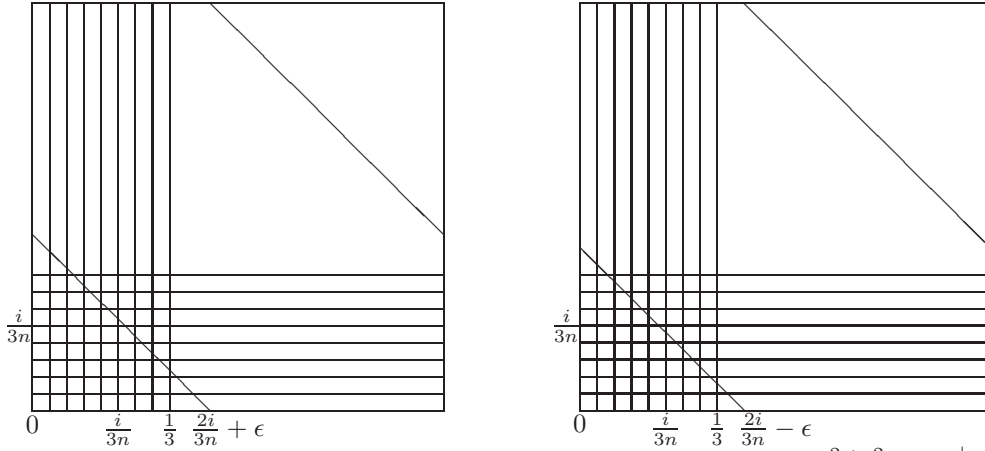
\unitlength 0.45mm 
\linethickness{0.4pt}
\ifx\plotpoint\undefined\newsavebox{\plotpoint}\fi 
\begin{picture}(110,130)(-20,97)
\put(0,100){\framebox(120,120)[cc]{}}
\put(5,220){\line(0,-1){120}}
\put(10,220){\line(0,-1){120}}
\put(15,220){\line(0,-1){120}}
\put(20,220){\line(0,-1){120}}
\put(25,220){\line(0,-1){120}}
\put(30,220){\line(0,-1){120}}
\put(35,220){\line(0,-1){120}}
\put(40,220){\line(0,-1){120}}
\put(0,105){\line(1,0){120}}
\put(0,110){\line(1,0){120}}
\put(0,115){\line(1,0){120}}
\put(0,120){\line(1,0){120}}
\put(0,125){\line(1,0){120}}
\put(0,130){\line(1,0){120}}
\put(0,135){\line(1,0){120}}
\put(0,140){\line(1,0){120}}
\put(52,100){\line(-1,1){52}}
\put(52,220){\line(1,-1){68}}
\put(25,94.5){\makebox(0,0)[cc]{$\frac{i}{3n}$}}
\put(-4,125){\makebox(0,0)[cc]{$\frac{i}{3n}$}}
\put(0,96){\makebox(0,0)[cc]{$0$}}
\put(40,94.5){\makebox(0,0)[cc]{$\frac{1}{3}$}}
\put(57,94.5){\makebox(0,0)[cc]{$\frac{2i}{3n}+\epsilon$}}
\end{picture}

\vspace{-.2cm}

\begin{picture}(110,0)(-180,92)
\put(0,100){\framebox(120,120)[cc]{}}
\put(5,220){\line(0,-1){120}}
\put(10,220){\line(0,-1){120}}
\put(15,220){\line(0,-1){120}}
\put(20,220){\line(0,-1){120}}
\put(25,220){\line(0,-1){120}}
\put(30,220){\line(0,-1){120}}
\put(35,220){\line(0,-1){120}}
\put(40,220){\line(0,-1){120}}
\put(0,105){\line(1,0){120}}
\put(0,110){\line(1,0){120}}
\put(0,115){\line(1,0){120}}
\put(0,120){\line(1,0){120}}
\put(0,125){\line(1,0){120}}
\put(0,130){\line(1,0){120}}
\put(0,135){\line(1,0){120}}
\put(0,140){\line(1,0){120}}
\put(48,100){\line(-1,1){48}}
\put(48,220){\line(1,-1){72}}
\put(25,94.5){\makebox(0,0)[cc]{$\frac{i}{3n}$}}
\put(-4,125){\makebox(0,0)[cc]{$\frac{i}{3n}$}}
\put(0,96){\makebox(0,0)[cc]{$0$}}
\put(40,94.5){\makebox(0,0)[cc]{$\frac{1}{3}$}}
\put(57,94.5){\makebox(0,0)[cc]{$\frac{2i}{3n}-\epsilon$}}
\end{picture}


It is easy to see that the the number of triangular regions cut out on $\R^2/\Z^2$ by $\V_i^+$ (resp. $\V_i^-$) is $4(n-i)+2$ (resp. $4(n-i)+4$).  This implies that the toric arrangements of hyperplanes $\{\V_i^+, \V_i^-\}_{n/2<i\leq n}$ have pairwise non isomorphic posets of regions.
According to Corollary \ref{C:crit-iso}, we conclude that there are at least $n$ if $n$ is even (resp. $n+1$ if $n$ is odd) different generic polarizations on $X_n$ giving rise to pairwise non isomorphic (resp. non homeomorphic if $k=\bbC$) fine compactified Jacobians.
\end{example}


\section{Deformation theory} \label{S:defo}

The aim of this section is to study the deformation theory and the semiuniversal deformation space of a pair $(X,I)$ where $X$ is a (reduced) connected curve and $I$ is rank-1 torsion-free simple sheaf on $X$. For basic facts  on deformation theory, we refer to the book of Sernesi \cite{Ser}.

\subsection{Deformation theory of $X$}\label{S:DefX}

The aim of this subsection is to recall some well-known facts about the deformation theory of a (reduced) curve $X$.

Let $\Def_X$ (resp. $\Def_X'$) be the local moduli functor of $X$ (resp. the locally trivial moduli functor) of $X$ in the sense of \cite[Sec. 2.4.1]{Ser}. Moreover, for any $p\in X_{\rm sing}$, we denote by $\Def_{X,p}$ the
deformation functor of the complete local $k$-algebra $\wh{\O}_{X,p}$ in the sense of \cite[Sec. 1.2.2]{Ser}. The above deformation functors are
related by the following natural morphisms:
\begin{equation}\label{E:mor-func}
\Def_X'\to \Def_X\to \Def_X^{\rm loc}:=\prod_{p\in X_{\rm sing}} \Def_{X,p}.
\end{equation}
Since $X$ is reduced, the tangent spaces to $\Def_X'$, $\Def_X$ and $\Def_{X,p}$ where $p\in X_{\rm sing}$
are isomorphic to (see \cite[Cor. 1.1.11, Thm. 2.4.1]{Ser})
\begin{equation}\label{E:tang-DefX}
\begin{aligned}
& T \Def'_X:=\Def'_X(k[\epsilon])=H^1(X,T_X), \\
& T \Def_X:=\Def_X(k[\epsilon])=\Ext^1(\Omega_X^1,\O_X), \\
& T \Def_{(X,p)}:=\Def_{(X,p)}(k[\epsilon])=(T_X^1)_p,
\end{aligned}
\end{equation}
where $\Omega_X^1$ is the sheaf of K\"ahler differentials on $X$, $T_X:={\mathcal Hom}(\Omega_X^1,\O_X)$
is the tangent sheaf of $X$ and $T_X^1={\mathcal Ext}^1(\Omega_X^1,\O_X)$ is the first cotangent sheaf of $X$, which is
a sheaf supported on $X_{\rm sing}$ by \cite[Prop. 1.1.9(ii)]{Ser}.

The usual local-to-global spectral sequence gives a short exact sequence
\begin{equation}\label{E:mor-tang}
\begin{aligned}
& 0 \to H^1(X, T_X)=T\Def_X'\to \Ext^1(\Omega_X^1,\O_X)=T\Def_X \to \\
& \to H^0(X, {\mathcal Ext}^1(\Omega_X^1,\O_X))=\bigoplus_{p\in X_{\rm sing}} {\mathcal Ext}^1(\Omega_X^1,\O_X)_p
= T\Def_{X}^{\rm loc} \to H^2(X,T_X)=0,
\end{aligned}
\end{equation}
which coincides with the exact sequence on the tangent spaces induced by \eqref{E:mor-func}.

By looking at the obstruction spaces of the above functors, one can give criteria under which the above deformation functors are smooth
(in the sense of \cite[Def. 2.2.4]{Ser}).

\begin{fact}\label{F:for-smooth}
\noindent
\begin{enumerate}[(i)]
\item \label{F:for-smooth1} $\Def_X'$ is  smooth;
\item \label{F:for-smooth2} If $X$ has l.c.i. singularities at $p\in X_{\rm sing}$ then $\Def_{X,p}$ is  smooth;
\item \label{F:for-smooth3}
If $X$ has l.c.i. singularities, then $\Def_X$ is smooth and the morphism $\displaystyle \Def_X\to \Def_{X}^{\rm loc}$ is  smooth.
\end{enumerate}
\end{fact}
\begin{proof}
Part \eqref{F:for-smooth1}: an obstruction space for $\Def_X'$ is $H^2(X,T_X)$ by \cite[Prop. 2.4.6]{Ser} and $H^2(X,T_X)=0$ because $\dim X=1$.
Therefore, $\Def'_X$ is smooth.

Part \eqref{F:for-smooth2} follows from \cite[Cor. 3.1.13(ii)]{Ser}.

Part \eqref{F:for-smooth3}:  by \cite[Prop. 2.4.8]{Ser}\footnote{In loc. cit., it is assumed that the characteristic of the base field is $0$. However, the statement is true in any characteristics, see \cite[Thm. (4.4)]{Vis}.},
an obstruction space for $\Def_X$ is $\Ext^2(\Omega_X^1,\O_X)$, which is zero
by \cite[Example 2.4.9]{Ser}. Therefore we get that $\Def_X$ is smooth.

Since $\displaystyle  \Def_{X}^{\rm loc}$ is smooth by part \eqref{F:for-smooth2} and the map of tangent spaces
$\displaystyle T\Def_X \to T\Def_{X}^{\rm loc}$ is surjective by \eqref{E:mor-tang}, the smoothness of the morphism
$\displaystyle \Def_X\to \Def_{X}^{\rm loc}$ follows from  the criterion \cite[Prop. 2.3.6]{Ser}.


\end{proof}

\subsection{Deformation theory of the pair $(X,I)$}\label{S:Def-pair}

The aim of this subsection is to review some fundamental results due to Fantechi-G\"ottsche-vanStraten \cite{FGvS} on
the deformation theory of a pair $(X,I)$, where $X$ is a (reduced) curve and $I$ is a rank-1 torsion-free  sheaf on $X$ (not necessarily simple).

Let $\Def_{(X,I)}$ be the deformation  functor of the pair $(X,I)$ and, for any $p\in X_{\rm sing}$, we denote by $\Def_{(X, I) ,p}$ the
deformation functor of the pair $(\wh{O}_{X,p},I_p)$.
We have a natural commutative diagram
\begin{equation}\label{E:diag-funct}
\xymatrix{
\Def_{(X,I)} \ar[r] \ar[d] & \Def_{(X,I)}^{\rm loc}:=\prod_{p\in X_{\rm sing}} \Def_{(X,I),p} \ar[d] \\
\Def_X \ar[r] & \Def_X^{\rm loc}:=\prod_{p\in X_{\rm sing}} \Def_{X,p}.
}
\end{equation}

Under suitable hypothesis, the deformation functors appearing in the above diagram \eqref{E:diag-funct} are smooth and the horizontal morphisms are smooth as well.

\begin{fact}[Fantechi-G\"ottsche-vanStraten]\label{F:diag-smooth}
\noindent
\begin{enumerate}[(i)]
\item \label{F:diag-smooth1} The natural morphism
$$\Def_{(X,I)}\to  \Def_{(X,I)}^{\rm loc}\times_{\Def_{X}^{\rm loc}} \Def_X$$
is smooth. In particular, if $X$ has l.c.i. singularities then the morphism $\Def_{(X,I)}\to \Def_{(X,I)}^{\rm loc}$ is smooth.
\item \label{F:diag-smooth2} If $X$ has locally planar singularities at $p\in X_{\rm sing}$ then $\Def_{(X,I),p}$ is smooth. In particular, if $X$ has locally planar singularities
then $\Def_{(X,I)}^{\rm loc}$ and $\Def_{(X,I)}$ are smooth.
\end{enumerate}
\end{fact}
\begin{proof}
Part \eqref{F:diag-smooth1}: the first assertion follows from \cite[Prop. A.1]{FGvS}\footnote{In loc. cit., it is assumed that the base field is the field of complex numbers. However, a direct inspection reveals that  the same argument works over any (algebraically closed) base field.}. The second assertion follows from the first one together with
Fact \ref{F:for-smooth}\eqref{F:for-smooth3} which implies that the morphism $\Def_{(X,I)}^{\rm loc}\times_{\Def_{X}^{\rm loc}} \Def_X \to \Def_{(X,I)}^{\rm loc}$ is smooth.

Part \eqref{F:diag-smooth2}: the first assertion follows from \cite[Prop. A.3]{FGvS}\footnote{As before, the argument of loc. cit. works over any (algebraically closed) base field.}. The second assertion follows from the first together with part \eqref{F:diag-smooth1}.

\end{proof}

\subsection{Semiuniversal deformation space}\label{S:def-space}

The aim of this subsection is to describe and study the semiuniversal deformation spaces for the deformation functors
$\Def_X$ and $\Def_{(X,I)}$.

According to \cite[Cor. 2.4.2]{Ser}, the functor $\Def_X$ admits a semiuniversal \footnote{Some authors use the word miniversal instead of semiuniversal. We prefer to use  the word semiuniversal in order to be coherent with the terminology of the book of Sernesi \cite{Ser}.} formal couple
$(R_X,\ov{\X})$, where $R_X$ is a Noetherian complete local $k$-algebra with maximal ideal $\m_{X}$ and residue field
$k$ and
$$\ov{\X}\in \wh{\Def_X}(R_X):=\varprojlim \Def_X\left(\frac{R_X}{\m_{X}^n}\right)$$
is a formal deformation of $X$ over $R_X$. Recall that this means that the morphism of functors
\begin{equation}\label{E:map-func1}
h_{R_X}:=\Hom(R_X,-)\longrightarrow \Def_X
\end{equation}
determined by $\ov{\X}$ is smooth and induces an isomorphism
of tangent spaces $T R_X:=(\m_X/\m_X^2)^{\vee}\stackrel{\cong}{\to} T \Def_X$ (see \cite[Sec. 2.2]{Ser}).
The formal couple $(R_X,\ov{\X})$ can be also viewed as a flat morphism of formal schemes
\begin{equation}\label{E:form-fam}
\ov{\pi}:{\ov \X}\to \Spf\: R_X,
\end{equation}
where $\Spf$ denotes the formal spectrum, such that the reduced scheme ${\ov \X}_{\rm red}$ underlying ${\ov \X}$ (see \cite[Prop. 10.5.4]{EGAI}) is isomorphic to $X$
(see \cite[p. 77]{Ser}).
Note that the semiuniversal formal couple $(R_X,\ov{\X})$ is unique by  \cite[Prop. 2.2.7]{Ser}.

Since $X$ is projective and $H^2(X,\O_X)=0$, Grothendieck's existence theorem (see \cite[Thm. 2.5.13]{Ser}) gives that
the formal deformation \eqref{E:form-fam} is \emph{effective}, i.e. there exists a deformation
$\pi:\X\to \Spec R_X$ of $X$ over $\Spec R_X$  whose completion along $X=\pi^{-1}([\m_X])$
is isomorphic to \eqref{E:form-fam}.  In other words, we have a Cartesian diagram
\begin{equation}\label{E:eff-fam}
\xymatrix{
X \ar[d]\ar@{^{(}->}[r]\ar@{}[dr]|{\square}& \ov{\X}\ar[r]\ar[d]^{\ov{\pi}}\ar@{}[dr]|{\square} & \X\ar[d]^{\pi}\\
\Spec k\cong [\m_X]\ar@{^{(}->}[r] & \Spf R_X\ar[r] &  \Spec R_X.
}
\end{equation}
Note also that the deformation $\pi$ is unique by \cite[Thm. 2.5.11]{Ser}.

Later on, we will need the following result on the effective semiuniversal deformation of a curve $X$ with locally planar singularities.

\begin{lemma}\label{L:codim1}
Assume that $X$ has locally planar singularities.
Let $U$ be the open subset of $\Spec R_X$ consisting of all the (schematic) points $s\in \Spec R_X$ such that
the geometric fiber $\X_{\ov s}$ of the universal family $\pi:\X\to \Spec R_X$ is smooth or has a unique singular point which is a
node. Then the codimension of the complement of $U$ inside $\Spec R_X$ is at least two.
\end{lemma}
\begin{proof}
Since the natural morphism (see \eqref{E:mor-func})
$$\Def_X \to \Def_X^{\rm loc}:=\prod_{p\in X_{\rm sing}} \Def_{X,p}$$
is smooth by Fact \ref{F:for-smooth}\eqref{F:for-smooth3}, it is enough to show  that if
$\Def_{X,p}$ has dimension at most one then $p\in X_{\rm sing}$ is either a smooth point or a node of $X$.
This is stated in \cite[Prop. 3.1.5.]{Ser} under the assumption that ${\rm char}(k)=0$. However, a slight modification of the argument of loc. cit. works in arbitrary characteristic, as we are now going to show.

First, since $X$ has locally planar singularities at $p$, we can write $\wh{\O}_{X,p}=\frac{k[[x,y]]}{f}$, for some power series $f=f(x,y)\in k[[x,y]]$. By \cite[p. 124]{Ser}, the tangent space to $\Def_{(X,p)}$ is equal to
$$T^1:=T^1_{\wh{\O}_{X,p}}=\frac{k[[x,y]]}{(f,\partial_x f, \partial_y f)}.$$
Since $\Def_{X,p}$ is smooth by Fact \ref{F:for-smooth}\eqref{F:for-smooth2}, then the dimension of $\Def_{X,p}$ is equal to $\dim_k T^1$. 

From the above description, it is clear that $\dim_k T^1=0$ if and only if $f$ contains some linear term, which happens if and only if $p$ is a smooth point of $X$.

Therefore, we are left with showing that $p$ is a node of $X$ (i.e. $f$ can be taken to be equal to $xy$) if and only if $\dim_ k T^1=1$, which is equivalent to $(x,y)=(f, \partial_x f,\partial_y f)$.
Clearly, if $f=xy$ then $\partial_x f=y$ and $\partial_y f=x$ so that $(x,y)=(f,\partial_x f,\partial_y f)=(xy, y, x)$.
Conversely, assume that $(x,y)=(f,\partial_x f,\partial_y f)$. Then clearly $f$ cannot have a linear term.
Consider the degree two part $f_2=Ax^2+Bxy+Cy^2$ of $f$. By computing the partial derivatives and imposing that
$x,y\in (f, \partial_x f,\partial_y f)$, we get that the discriminant $\Delta=B^2-4AC$ of $f_2$ is different from $0$. Then, acting with a linear change of coordinates, we can assume that $f_2=xy$.
Now, it is easily checked that via a change of coordinates of the form $x\mapsto x+g(x,y)$ and $y\mapsto y+h(x)$ with $g(x,y)\in (x,y)^2$ and $h(x)\in (x)^2$, we can transform $f$ into $xy$, and we are done.

\end{proof}

Consider now the functor
$$\bJbar_{\X}^*:\{\Spec R_X-\text{schemes} \} \longrightarrow \{ \text{Sets} \}$$
which sends a scheme $T\to \Spec R_X$ to the set of isomorphism classes of $T$-flat, coherent sheaves on
$\X_T:=T\times_{\Spec R_X} \X$ whose fibers over $T$ are simple rank-1  torsion-free sheaves. The functor
$\bJbar_{\X}^*$ contains the open subfunctor
$$\bJ_{\X}^*:\{\Spec R_X-\text{schemes} \} \longrightarrow \{ \text{Sets} \}$$
which sends a scheme $T\to \Spec R_X$ to the set of isomorphism classes of line bundles on
$\X_T$.

Analogously to Fact \ref{F:huge-Jac}, we have the following

\begin{fact}[Altman-Kleiman, Esteves]\label{F:univ-Jac}
\noindent
\begin{enumerate}[(i)]
\item \label{F:univ-Jac1}The Zariski (equiv. \'etale, equiv. fppf) sheafification  of $\bJbar_{\X}^*$ is represented by
a scheme $\bJbar_{\X}$ endowed with a morphism
$u:\bJbar_{\X}\to \Spec R_X$, which is locally of finite type and satisfies the existence part of the valuative criterion for properness.
The scheme $\bJbar_{\X}$ contains an open subset $\bJ_{\X}$ which represents the Zariski (equiv. \'etale, equiv. fppf) sheafification  of
$\bJ_{\X}^*$ and the restriction $u:\bJ_{\X}\to \Spec R_X$ is smooth.\\
Moreover, the fiber of $\bJbar_{\X}$ (resp. of $\bJ_{\X}$) over the closed point $[\m_X]\in \Spec R_X$ is isomorphic to $\bJbar_X$ (resp. $\bJ_X$).

\item \label{F:univ-Jac2} There exists a sheaf $\wh{\I}$ on $\X\times_{\Spec R_X} \bJbar_{\X}$  such for every $\F\in \bJbar_{\X}^*(T)$ there exists a unique $\Spec R_X$-map $\alpha_{\F}:T\to \bJbar_{\X}$ with the property that $\F=(\id_{\X}\times \alpha_{\F})^*(\wh{\I})\otimes \pi_2^*(N)$ for some $N\in \Pic(T)$, where
$\pi_2:\X\times_{\Spec R_X} T\to T$ is the projection onto the second factor.
The sheaf $\wh{\I}$ is uniquely determined up  to tensor product with the pullback of an invertible sheaf on $\bJbar_{\X}$ and it is called a \emph{universal sheaf} on $\bJbar_{\X}$.\\
Moreover, the restriction of $\wh{\I}$ to $X\times \bJbar_X$ is equal to a universal sheaf as in Fact \ref{F:huge-Jac}\eqref{F:huge3}.

\end{enumerate}
\end{fact}
\begin{proof}
Part  \eqref{F:univ-Jac1}: the representability of the \'etale sheafification (and hence of the fppf sheafification) of $\bJbar_{\X}^*$ by an algebraic space
$\bJbar_{\X}$ locally of finite type over $\Spec R_X$ follows from \cite[Thm. 7.4]{AK}, where it is proved for the moduli functor of simple sheaves, along with the fact that being torsion free and rank-1 is an open condition. 
From \cite[Cor. 52]{est1}, it follows that $\bJbar_{\X}$ becomes a
scheme after an \'etale cover of $\Spec R_X$. However, since
$R_X$ is strictly henselian (being a complete local ring with algebraically closed residue field), then $\Spec R_X$ does not admit non trivial connected \'etale covers (see \cite[Sec. 2.3]{BLR});
hence $\bJbar_{\X}$ is a scheme.
The scheme $\bJbar_{\X}$ satisfies the existence part of the valuative criterion for properness by \cite[Thm. 32]{est1}.

The fact that $\bJbar_{\X}$ represents also the Zariski sheafification of $\bJbar_{\X}^*$ follows from \cite[Thm. 3.4]{AK2}\footnote{This result is stated in loc. cit. only for flat and proper morphisms with integral geometric fibers; however, the same proof works assuming only reduced geometric fibers.} once we prove that the morphism $\pi:\X\to \Spec R_X$ admits a section through its smooth locus.
Indeed, let $U$ be the smooth locus of the morphism $\pi$ and denote by $\pi':U\to \Spec R_X$ the restriction of $\pi$
to $U$. Since $X$ is assumed to be reduced, all the geometric fibers of $\pi$ are reduced by \cite[Thm. 12.2.4]{EGAIV3}; hence,
we deduce that for every $s\in \Spec R_X$ the open subset $\pi'^{-1}(s)$ is dense in $\pi^{-1}(s)$.
Now, since $R_X$ is a strictly henselian ring,  given any point $p\in \pi'^{-1}([\m_X])$, we can find a section of $\pi':U\to \Spec R_X$ passing through $p$
(see \cite[Sec. 2.3, Prop. 5]{BLR}), as required.

Since $\bJ_{\X}^*$ is an open subfunctor of $\bJbar_{\X}^*$, it follows that $\bJbar_{\X}$ contains an open subscheme
$\bJ_{\X}$ which represents the \'etale sheafification of $\bJ_{\X}^*$.
The smoothness of $\bJ_{\X}$ over $\Spec R_X$ follows from \cite[Sec. 8.4, Prop. 2]{BLR}. The last assertion of part \eqref{F:univ-Jac1} is obvious.


Part \eqref{F:univ-Jac2} is an immediate consequence of the fact that $\bJbar_{\X}$ represents the Zariski sheafification of $\bJbar_{\X}^*$ (see   also \cite[Thm. 3.4]{AK2}).
 The last assertion of part \eqref{F:univ-Jac2} is obvious.

\end{proof}

Let now $I$ be a simple rank-1 torsion-free sheaf $I$ on $X$, i.e. $I\in \bJbar_X\subset \bJbar_{\X}$.
If we denote by $R_{(X,I)}:=\wh{\O}_{\bJbar_{\X},I}$ the completion of the local ring of $\bJbar_{\X}$ at $I$ and by
$\m_{(X,I)}$ its maximal ideal, then there is a natural map $j: \Spec R_{(X,I)} \to \bJbar_{\X}$ which fits into the  following Cartesian diagram
\begin{equation}\label{E:diag-fam}
\xymatrix{
(\id\times j)^*(\wh{\I}) \ar@{-}[d] & \wh{\I}\ar@{-}[d] & \\
\X\times_{\Spec R_X} \Spec R_{(X,I)} \ar[r]^{{\rm id} \times j}\ar[d]_{\pi\times \id}
\ar@{}[dr]|{\square} &
\X\times_{\Spec R_X} \bJbar_{\X} \ar[d]\ar[r] \ar@{}[dr]|{\square}& \X\ar[d]^{\pi}\\
\Spec R_{(X,I)} \ar[r]^{j} & \bJbar_{\X} \ar[r]^{u}& \Spec R_X.
}
\end{equation}
Since $I\in \bJbar_X\subset \bJbar_{\X}$, the map $u\circ j$ sends the closed point $[\m_{X,I}]\in \Spec \wh{\O}_{\bJbar_{\X},I}$ into the closed point $[\m_X]\in \Spec R_X$. In particular, we have that
$(\pi\times \id)^{-1}(\m_{(X,I)})=\pi^{-1}(\m_X)=X$ and the restriction of $(\id\times j)^*(\wh{\I})$ to
$(\pi\times \id)^{-1}(\m_{(X,I)})=X$ is isomorphic to $I$ by the universal property in Fact \ref{F:univ-Jac}\eqref{F:univ-Jac2}. The above diagram gives rise to a deformation of the pair $(X,I)$ above $\Spec R_{(X,I)}$, which induces
a morphism of deformation functors
\begin{equation}\label{E:map-func2}
h_{R_{(X,I)}}:=\Hom(R_{(X,I)}, -)\longrightarrow \Def_{(X,I)}.
\end{equation}
We can now prove the main result of this section.

\begin{thm}\label{T:univ-rings}
Let $X$ be a (reduced) curve and $I$ a rank-1 torsion-free simple sheaf on $X$.
\begin{enumerate}[(i)]
\item \label{T:univ-rings1} There exists a Cartesian diagram of deformation functors
\begin{equation}\label{E:cart-diag}
\xymatrix{
h_{R_{(X,I)}}\ar[r] \ar[d]\ar@{}[dr]|{\square} & \Def_{(X,I)} \ar[d]\\
h_{R_X} \ar[r] & \Def_X,
}
\end{equation}
where the horizontal arrows realize $R_{(X,I)}$ and $R_X$ as the semiuniversal deformation rings for
$\Def_{(X,I)}$ and $\Def_X$, respectively.
\item \label{T:univ-rings2} If $X$ has l.c.i. singularities then $R_X$ is regular (i.e. it is a power series ring over $k$).
\item \label{T:univ-rings3} If $X$ has locally planar singularities then $R_{(X,I)}$ is regular. In particular,
the scheme $\bJbar_{\X}$ is regular.
\end{enumerate}
\end{thm}
\begin{proof}
Part \eqref{T:univ-rings1}: the fact that the diagram \eqref{E:cart-diag} is commutative follows from the definition
of the map  \eqref{E:map-func2} and the commutativity of the diagram \eqref{E:diag-fam}.

Let us check that
the above diagram \eqref{E:cart-diag} is Cartesian. Let $A$ be an Artinian local $k$-algebra with maximal ideal $\m_A$. Suppose that there exists
a deformation $(\wt{X},\wt{I})\in \Def_{(X,I)}(A)$ of $(X,I)$ over $A$ and a homomorphism $\phi\in \Hom(R_X,A)=h_{R_X}(A)$ that have the same image in $\Def_X(A)$. We have to find a homomorphism
$\eta\in \Hom(R_{(X,I)}, A)=h_{R_{(X,I)}}(A)$ that maps into $\phi\in h_{R_X}(A)$ and $(\wt{X},\wt{I})\in \Def_{(X,I)}(A)$ via the maps of diagram \eqref{E:cart-diag}. The assumption that the elements $(\wt{X},\wt{I})\in \Def_{(X,I)}(A)$ and $\phi\in h_{R_X}(A)$ have the same image in $\Def_X(A)$ is equivalent to the fact
that $\wt{X}$ is isomorphic to $\X_A:= \X\times_{\Spec R_X}\Spec A$ with respect to the natural morphism $\Spec A\to
\Spec R_X$ induced by $\phi$. Therefore the sheaf
$\wt{I}$ can be seen as an element of $\bJbar_{\X}^*(\Spec A)$. Fact \ref{F:univ-Jac}\eqref{F:univ-Jac2} gives a map
$\alpha_{\wt{I}}:\Spec A\to \bJbar_{\X}$ such that $\wt{I}=(\id_{\X}\times \alpha_{\wt{I}})^*(\wh{\I})$, because
$\Pic(\Spec A)=0$.
Clearly the map $\alpha_{\wt{I}}$ sends $[\m_A]$ into $I\in \bJbar_X\subset \bJbar_{\X}$ and therefore it factors
through a map $\beta:\Spec A\to \Spec R_{(X,I)}$ followed by the map $j$ of \eqref{E:diag-fam}. The morphism $\beta$
determines the element $\eta \in \Hom(R_{(X,I)}, A)=h_{R_{(X,I)}}(A)$ we were looking for.

Finally, the bottom horizontal morphism realizes the ring $R_X$ as the semiuniversal deformation ring for $\Def_X$ by
the very definition of $R_X$. Since the diagram \eqref{E:cart-diag} is Cartesian, the same is true for the top horizontal arrow.

Part \eqref{T:univ-rings2}: $R_X$ is regular since the morphism $h_{R_X}\to \Def_X$ is smooth and $\Def_X$ is smooth
by Fact \ref{F:for-smooth}\eqref{F:for-smooth3}.

Part \eqref{T:univ-rings3}: $R_{(X,I)}$ is regular since the morphism $h_{R_{(X,I)}}\to \Def_{(X,I)}$ is smooth
and $\Def_{(X,I)}$ is smooth by Fact \ref{F:diag-smooth}\eqref{F:diag-smooth2}.
We deduce that the open subset $U$ of regular points of $\bJbar_{\X}$ contains the central fiber
$u^{-1}([\m_X])=\bJbar_X$, which implies that $U=\bJbar_{\X}$ because $u^{-1}([\m_X])$ contains all the closed points of $\bJbar_{\X}$; hence $\bJbar_{\X}$ is regular.

\end{proof}

\section{Universal  fine compactified Jacobians}\label{S:univ-Jac}

The aim of this section is to introduce and study the universal fine compactified Jacobians relative to the
semiuniversal deformation $\pi:\X\to \Spec R_X$ introduced in  \S\ref{S:def-space}.

\vspace{0,2cm}

The universal fine compactified Jacobian will depend on  a general polarization $\un q$ on $X$ as in Definition \ref{def-int}. Indeed, we are going to show that the polarization $\un q$ induces
a polarization on each fiber of the effective semiuniversal deformation family
$\pi:\X\to \Spec R_X$.

With this aim, we will first show that the irreducible components of the fibers of the morphism  $\pi:\X\to \Spec R_X$ are geometrically irreducible. For any (schematic) point $s\in \Spec R_X$, we denote by $\X_s:=\pi^{-1}(s)$ the fiber of $\pi$ over $s$, by $\X_{\ov s}:=\X_s\times_{k(s)} \ov{k(s)}$ the geometric fiber over $s$ and by $\psi_s:\X_{\ov s}\to \X_s$ the natural morphism.

\begin{lemma}\label{L:geo-irr}
The irreducible components of $\X_s$ are geometrically irreducible.  Therefore we get a bijection
$$\begin{aligned}
(\psi_s)_* :\left\{ \text{Subcurves of } \X_{\ov s} \right\} & \stackrel{\cong}{\longrightarrow} \left\{ \text{Subcurves of } \X_s \right\} \\
Z\subseteq \X_{\ov s} & \mapsto \psi_s(Z)\subseteq \X_s.
\end{aligned}$$
\end{lemma}
\begin{proof}
Let $V\subseteq \X$ be the biggest open subset where the restriction of the morphism $\pi:\X\to \Spec R_X$ is smooth.
Since $\pi$ is flat, the  fiber $V_{s}$ of $V$ over a point $s\in \Spec R_X$ is the smooth locus of the curve $\X_{s}=\pi^{-1}(s)$, which is geometrically reduced because the central curve $X=\pi^{-1}([\m_X])$ is reduced.  In particular,  $V_s\subseteq \X_s$ and $V_{\ov s}:=V_s\times_{k(s)} \ov{k(s)}\subseteq \X_{\ov s} $ are dense open subsets.
Therefore, the irreducible components of $\X_s$ (resp. of $\X_{\ov s}$) are equal to the irreducible components of $V_s$ (resp. of $V_{\ov s}$). However, since $V_s$ is smooth over $k(s)$ by construction, the irreducible components of $V_s$ coincide with the connected components of $V_s$ and similarly for $V_{\ov s}$.
In conclusion, we have to show that the connected components of $V_s$ are geometrically connected for any point $s\in \Spec R_X$.

We will need the following preliminary result. 

\un{Claim}: For any point $s\in \Spec R_X$, the irreducible components of $V_{\ov{\{s\}}}:=V\cap \pi^{-1}(\ov{\{s\}})$ do not meet on the central fiber $V_o:=\pi^{-1}([\m_X])\cap V$ and each of them is the closure of a unique irreducible component of $V_s$.  

Indeed, observe that $\ov{\{s\}}$ is a closed integral subscheme of $\Spec R_X$, so that $\ov{\{s\}}=\Spec T$ where $T$ is a  Noetherian complete local domain quotient of $R_X$ with residue field $k=\ov k$; hence, 
$T$ is a strictly Henselian local domain. 
This implies that $\Spec T$ is geometrically unibranch at its unique closed point $o=[\m_x]$ (see \cite[Tag 06DM]{Stack}).  Since the morphism $V_{\ov{\{s\}}}\to \ov{\{s\}}=\Spec T$ is smooth, we infer 
that $V_{\ov{\{s\}}}$ is geometrically unibranch along the central fiber $V_o$  (see \cite[Prop. 6.15.10]{EGAIV2}). This implies that two distinct irreducible components of $V_{\ov{\{s\}}}$ do not meet along the central fiber $V_o$, 
and the first assertion of the Claim follows. The second assertion follows from the fact that, since $V_{\ov{\{s\}}}\to \ov{\{s\}}$ is flat, each generic point of $V_{\ov{\{s\}}}$ maps to the generic point $s$ of $\ov{\{s\}}$, q.e.d.

\vspace{0.1cm}

Let now $C$ be a connected component of $V_s$, for some point $s\in \Spec R_X$. The closure $\ov{C}$ of $C$ inside $\X$ will contain some irreducible component of the central fiber $\X_o=\X_{[\m_X]}$ by the upper semicontinuity of the dimension of the fibers 
(see \cite[Lemma 13.1.1]{EGAIV3}) applied to the projective surjective morphism  $\wt{C}\to \ov{\{s\}}$. Hence, $\ov C\cap V$  will contain some (not necessarily unique) connected component $C_o$ of the central fiber $V_o=V_{[\m_X]}$. 
Since $R_X$ is a strictly henselian ring and $V\to \Spec R_X$ is smooth,  given any point $p\in C_o\subseteq V_o$, we can find a section $\sigma$ of $V\to \Spec R_X$ passing through $p$ (see \cite[Sec. 2.3, Prop. 5]{BLR}). 
By the Claim, $\ov C\cap V$ is the unique irreducible component  of $V_{\ov{\{s\}}}$ containing the point $p$. Therefore, the restriction of $\sigma$ at $\ov{\{s\}}$ must take values in  $\ov C\cap V$. 
In particular, $\sigma(s)$ is a $k(s)$-rational point of $C$. Now  we conclude that $C$ is geometrically connected by  \cite[Cor. 4.5.14]{EGAIV2}.
\end{proof}

Consider now the set-theoretic map
\begin{equation}\label{E:map-subcurve}
\begin{aligned}
\Sigma_s:\{\text{Subcurves of } \X_{\ov s}\} & \longrightarrow \{\text{Subcurves of } X\}\\
\X_{\ov s}\supseteq Z & \mapsto \ov{\psi_s(Z)}\cap X\subseteq X,
\end{aligned}
\end{equation}
where $\ov{\psi_s(Z)}$ is  the Zariski closure inside $\X$ of the subcurve $\psi_s(Z)\subseteq \X_s$ and the intersection
$\ov{\psi_s(Z)}\cap X$ is endowed with the reduced scheme structure. Note that $\ov{\psi_s(Z)}\cap X$ has pure dimension one (in other words, it does not contain isolated points), hence it is a subcurve of $X$, 
by the upper semicontinuity of the local dimension of the fibers (see \cite[Thm. 13.1.3]{EGAIV3}) applied to the morphism $\ov{\psi_s(Z)}\to \overline{\{s\}}$ and using the fact that $\psi_s(Z)$ has pure dimension one in $\X_s$.

The map $\Sigma_s$ satisfies two important properties that we collect in the following

\begin{lemma}\label{L:mapSigma}
\noindent
\begin{enumerate}[(i)]
\item  \label{L:mapSigma1} If $Z_1,Z_2\subseteq\X_{\ov s}$ do not have common irreducible components then
$\Sigma_s(Z_1), \Sigma_s(Z_2)\subseteq X$ do not have common irreducible components. In particular, $\Sigma_s(Z^c)=\Sigma_s(Z)^c$.
\item \label{L:mapSigma2} If $Z\subseteq \X_{\ov s}$ is connected then $\Sigma_s(Z)\subseteq X$ is connected.
\end{enumerate}
\end{lemma}
\begin{proof}
Let us first prove \eqref{L:mapSigma1}. Since $Z_1,Z_2$ are two subcurves of $\X_{\ov s}$ without common irreducible components then the subcurves $\psi_s(Z_1)$ and $\psi_s(Z_2)$ of $\X_s$ do not have common irreducible components by Lemma \ref{L:geo-irr}. As in the proof of Lemma \ref{L:geo-irr},  denote by $V$  the biggest open subset of $\X$ on which the restriction of the morphism $\pi$ is smooth.
Then, since $V_s:=V\cap \X_s$ is the smooth locus of $\X_s$,  we deduce that $\psi_s(Z_1)\cap V$ and $\psi_s(Z_2)\cap V$ are disjoint subsets of $\X_s\cap V$ each of which is a union of connected components of $\X_s\cap V$. 
By the Claim in the proof of Lemma \ref{L:geo-irr},  the closures $\ov{\psi_s(Z_1)}\cap V$ and $\ov{\psi_s(Z_2)}\cap V$ do not intersect in the central fiber $V_o$, or in other words 
$\Sigma_s(Z_1)\cap V=\ov{\psi_s(Z_1)}\cap V\cap X$ and $\Sigma_s(Z_2)\cap V=\ov{\psi_s(Z_2)}\cap V\cap X$ are disjoint. This implies that 
$\Sigma_s(Z_1)$ and $\Sigma_s(Z_2)$ intersect only in the singular locus of $X$, and in particular they do not share any irreducible component of $X$.



Let us now prove \eqref{L:mapSigma2}.  Consider the closed subscheme (with reduced scheme structure) $\ov{\psi_s(Z)}\subseteq \X$ and the
projective and surjective morphism $\sigma:=\pi_{|\ov{\psi_s(Z)}}:\ov{\psi_s(Z)}\to \ov{\{s\}}$, where $ \ov{\{s\}} \subseteq \Spec R_X$ is the  closure (with reduced structure) of the schematic point $s$ inside the scheme $\Spec R_X$.  Note that $\Sigma_s(Z)$ is, by definition, the reduced scheme associated to the central fiber $\ov{\psi_s(Z)}_o:=\sigma^{-1}([\m_X])$ of $\sigma$.
By Lemma \ref{L:geo-irr}, the geometric generic fiber of $\sigma$ is equal to $\psi_s(Z)\times_{k(s)}\ov{k(s)}=Z$, hence it is connected by assumption.  
Therefore, there is an open subset $W\subseteq  \ov{\{s\}}$ such that $\sigma^{-1}(W)\to W$ has geometrically connected fibers (see \cite[Tag 055G]{Stack}).

 Choose now a complete discrete valuation ring $R$, with residue field $k$,
endowed with a morphism $f: \Spec R\to  \ov{\{s\}}$ that maps the generic point $\eta$ of $\Spec R$ to a certain point $t\in W$ and the special point $0$ of $\Spec R$ to $[\m_X]$ in such a
way that the induced morphism $\Spec k(0)\to \Spec k([\m_x])$ is an isomorphism.
Consider the pull-back $\tau: \Y\to \Spec R$ of the family $\sigma:\ov{\psi_s(Z)}\to \ov{\{s\}}$ via the morphism $f$. By construction,  the special fiber $\Y_0=:\tau^{-1}(0)$ of $\tau$ is equal to
$\ov{\psi_s(Z)}_o$ and the generic fiber $\Y_{\eta}:=\tau^{-1}(\eta)$ of $\tau$ is equal to the fiber product  $\sigma^{-1}(t) \times_{\Spec k(t)}\Spec k(\eta)$. In particular, the generic fiber $\Y_{\eta}$ is geometrically connected.

Next, consider the closure $\cZ:=\ov{\Y_{\eta}}$ of the generic fiber $\Y_{\eta}$  inside $\Y$, i.e. the unique closed subscheme $\cZ$ of $\Y$ which is flat over $\Spec R$ and such that its generic fiber $\cZ_{\eta}$ is equal to $\Y_{\eta}$ (see \cite[Prop. 2.8.5]{EGAIV2}). The special fiber $\cZ_0$ of $\cZ$ is a closed subscheme of $\Y_0=\ov{\psi_s(Z)}_o$ which must  contain the dense open subset  $X_{\rm sm}\cap \Sigma_s(Z)\subseteq \Sigma_s(Z)$, where $X_{\rm sm}$ is the smooth locus of $X$. Indeed,  arguing as in the proof of Lemma \ref{L:geo-irr}, through any point $p$ of   $X_{\rm sm}\cap \Sigma_s(Z)$ there is a section of $\X\times_{\Spec R_X} \Spec R\to \Spec R$  entirely contained in $\Y$, which shows that $p$ must lie in the closure of $\Y_{\eta}$ inside $\Y$, i.e. in $\cZ$. Therefore, $\Sigma_s(Z)$ is also the reduced scheme associated to the central fiber $\cZ_0$. Finally, since the morphism $\cZ\to \Spec R$ is flat and projective by construction and the generic fiber $\cZ_{\eta}=\Y_{\eta}$ is geometrically connected, we deduce that $\cZ_0$, and hence $\Sigma_s(Z)$, is (geometrically) connected by applying \cite[Prop. (15.5.9)]{EGAIV3} (which says that the number of geometrically connected components of the fibers of a flat and proper is lower semicontinuous).


\end{proof}

We are now ready to show that a (general) polarization on $X$ induces, in a canonical way, a (general) polarization on each geometric fiber of its semiuniversal deformation $\pi:\X\to \Spec R_X$.

\begin{lemdef}\label{D:def-pola}
Let $s\in \Spec R_X$ and let $\un q$ be a polarization on $X$. The polarization $\un q^s$ induced by $\un q$ on the geometric fiber
$\X_{\ov s}$ is defined by
$$\un q^s_Z:=\un q_{\Sigma_s(Z)}\in \Q$$
for every subcurve $Z\subseteq \X_{\ov s}$. If $\un q$ is general then $\un q^s$ is general.
\end{lemdef}
\begin{proof}
Let us first check that $\un q^s$ is well-defined. i.e. that $|\un q^s|\in \Z$ and that $(Z\subseteq \X_{\ov s}) \mapsto \un q^s_Z$ is additive (see the discussion after Definition \ref{pola-def}).
Since $\Sigma_s(\X_{\ov s})=X$, we have that  $|\un q^s|=\un q^s_{\X_{\ov s}}=\un q_X=|\un q|\in \Z.$ Moreover, the additivity of $\un q^s$ follows from the additivity of $\un q$ using Lemma
\ref{L:mapSigma}\eqref{L:mapSigma1}.

The last assertion follows immediately from Remark \ref{R:conn-pola} and Lemma \ref{L:mapSigma}.

\end{proof}

Given a general polarization $\un q$ on $X$, we are going to construct an open subset of $\bJbar_{\X}$, proper over $\Spec R_X$, whose geometric fibers are fine compactified Jacobians with respect to the  general polarizations constructed in the above Lemma-Definition \ref{D:def-pola}.

\begin{thm}\label{T:univ-fine}
Let $\un q$ be a general polarization on $X$. Then there exists an open subscheme $\J_{\X}(\un q)\subseteq \bJbar_{\X}$
which is projective over $\Spec R_X$ and such that the geometric fiber of  $u:\J_{\X}(\un q)\to \Spec R_X$ over a point
$s\in \Spec R_X$ is isomorphic to $\J_{\X_{\ov s}}(\un q^s)$. In particular, the fiber of $\J_{\X}(\un q)\to \Spec R_X$
over the closed point $[\m_X]\in \Spec R_X$ is isomorphic to $\J_X(\un q)$.
\end{thm}
We call the scheme $\J_{\X}(\un q)$ the \emph{universal fine compactified Jacobian} of $X$ with respect to the polarization $\un q$. We denote by $J_{\X}(\un q)$ the open subset of $\J_{\X}(\un q)$ parametrizing line bundles,
i.e. $J_{\X}(\un q)=\J_{\X}(\un q)\cap \bJ_{\X}\subseteq \bJbar_{\X}$.

\begin{proof}
This statement follows by applying to the effective semiuniversal family $\X\to \Spec R_X$ a general result of Esteves (\cite[Thm. A]{est1}). In order to connect our notations with the notations of loc. cit.,
choose a vector bundle $E$ on $X$ such that $\un q^E=\un q$ (see Remark \ref{R:compEst}), so that  our fine compactified Jacobian $\J_X(\un q)$
coincides with the variety $J_{E}^s=J_{E}^{ss}$ in \cite[Sec. 4]{est1}.

Since an obstruction space for the functor of deformations of $E$ is $H^2(X,E\otimes E^{\vee})$
(see e.g. \cite[Thm. 8.5.3(b)]{FGA}) and since this latter group is zero because $X$ is a curve, we get that $E$ can be
extended to a vector bundle $\ov{\E}$ on the formal semiuniversal deformation $\ov{\X}\to \Spf R_X$ of $X$.
However, by Grothendieck's algebraization theorem for coherent sheaves (see \cite[Thm. 8.4.2]{FGA}), the vector bundle $\ov{\E}$ is the completion of a vector bundle $\E$ on the
effective semiuniversal deformation family $\pi: \X\to \Spec R_X$ of $X$.
Note that the restriction of $\E$ to the central fiber of $\pi$ is isomorphic to the vector bundle $E$ on $X$.
 Denote by $\E_s$ (resp. $\E_{\ov s}$) the restriction of $\E$ to the fiber $X_s$ (resp. the geometric fiber $\X_{\ov s}$).

 \un{Claim:}  For any $s\in \Spec R_X$ and any subcurve $Z\subseteq \X_{\ov s}$, we have that
$$\deg_Z(\E_{\ov s})=\deg_{\psi_s(Z)}(\E_s)=\deg_{\Sigma_s(Z)}(E).$$

Indeed, the first equality  follows  from the fact that $Z$ is the pull-back of $\psi_s(Z)$ via the map $\Spec \ov{k(s)}\to \Spec k(s)$ because of Lemma \ref{L:geo-irr}.
In order to prove the second equality, consider the closed subscheme (with reduced scheme structure) $\ov{\psi_s(Z)}\subseteq \X$ and
 the projective and surjective morphism\footnote{We do not know if $\sigma$ is flat, a property that would considerably simplify the proof of Claim.}
$\sigma:=\pi_{|\ov{\psi_s(Z)}}:\ov{\psi_s(Z)}\to \ov{\{s\}}$, where $ \ov{\{s\}} \subseteq \Spec R_X$ is the closure of the schematic point $s$ inside the scheme $\Spec R_X$. Note that the central fiber $\sigma^{-1}([\m_x]):=\ov{\psi_s(Z)}_o$ of $\sigma$ is a
one-dimensional subscheme of $X$, which is generically reduced  (because $X$ is reduced) and whose underlying reduced curve is $\Sigma_s(Z)$ by definition.
In particular, the $1$-cycle associated to $\ov{\psi_s(Z)}_o$ coincides with the $1$-cycle associated to $\Sigma_s(Z)$. Therefore, since the degree of a vector bundle on a subscheme depends only on the associated cycle, we have that
\begin{equation}\label{E:eq1}
\deg_{\Sigma_s(Z)}(E)=\deg_{\ov{\psi_s(Z)}_o}(E).
\end{equation}

Observe that there exists an open subset $U\subseteq \ov{\{s\}}$ such that $\sigma_{|\sigma^{-1}(U)}:\sigma^{-1}(U)\to U$  is flat (by the Theorem of generic flatness, see \cite[Lecture 8]{Mum}).
Since the degree of a vector bundle is preserved along the fibers of a flat morphism and clearly $s\in U$, we get that
\begin{equation}\label{E:eq1b}
\deg_{\psi_s(Z)}(\E_s)=\deg_{\psi_s(Z)}(\E)=\deg_{\ov{\psi_s(Z)}_t}(\E) \hspace{0.2cm} \text{ for any } t\in U,
\end{equation}
where we set $\ov{\psi_s(Z)}_t:=\sigma^{-1}(t)$.

Choose now a complete discrete valuation ring $R$, with residue field $k$,
endowed with a morphism $f: \Spec R\to  \ov{\{s\}}$ that maps the generic point $\eta$ of $\Spec R$ to a certain point $t\in U$ and the special point $0$ of $\Spec R$ to $[\m_X]$ in such a
way that the induced morphism $\Spec k(0)\to \Spec k([\m_x])$ is an isomorphism.
Consider the pull-back $\tau: \Y\to \Spec R$ of the family $\sigma:\ov{\psi_s(Z)}\to \ov{\{s\}}$ via the morphism $f$ and denote by $\F$ the pull-back to $\Y$ of the restriction of the vector bundle $\E$ to
$\ov{\psi_s(Z)}$. By construction,  the special fiber $\Y_0=:\tau^{-1}(0)$ of $\tau$ is equal to $\ov{\psi_s(Z)}_o$ and the generic fiber $\Y_{\eta}:=\tau^{-1}(\eta)$ of $\tau$ is equal to the fiber product
$\ov{\psi_s(Z)}_t \times_{\Spec k(t)}\Spec k(\eta)$.
Therefore, we have that
\begin{equation}\label{E:eq2}
\deg_{\Y_{0}}(\F)=\deg_{\ov{\psi_s(Z)}_o}(\E)=\deg_{\ov{\psi_s(Z)}_o}(E) \hspace{0.3cm} \text{ and } \hspace{0.3cm}  \deg_{\Y_{\eta}}(\F)=\deg_{\ov{\psi_s(Z)}_t}(\E).
\end{equation}

Next, consider the closure $\cZ:=\ov{\Y_{\eta}}$ of the generic fiber $\Y_{\eta}$  inside $\Y$, i.e. the unique closed subscheme $\cZ$ of $\Y$ which is flat over $\Spec R$ and such that its generic fiber $\cZ_{\eta}$ is equal to $\Y_{\eta}$ (see \cite[Prop. 2.8.5]{EGAIV2}). The special fiber $\cZ_0$ of $\cZ$ is a closed subscheme of $\Y_0=\ov{\psi_s(Z)}_o$ which must  contain the dense open subset
$X_{\rm sm}\cap \Sigma_s(Z)\subseteq \Sigma_s(Z)$, where $X_{\rm sm}$ is the smooth locus of $X$. Indeed,  arguing as in the proof of Lemma \ref{L:geo-irr}, through any point $p$ of
$X_{\rm sm}\cap \Sigma_s(Z)$ there is a section of $\X\times_{\Spec R_X} \Spec R\to \Spec R$  entirely contained in $\Y$, which shows that $p$ must lie in the closure of $\Y_{\eta}$ inside $\Y$, i.e. in $\cZ$.
Therefore,  the $1$-cycle associated to $\cZ_0$ coincides with the $1$-cycle associated to $\Sigma_s(Z)$, from which we deduce that
\begin{equation}\label{E:eq3}
\deg_{\Sigma_s(Z)}(E)=\deg_{\cZ_0}(\F).
\end{equation}
Finally, since the morphism $\cZ\to \Spec R$ is flat, we have that
\begin{equation}\label{E:eq4}
\deg_{\cZ_{0}}(\F)= \deg_{\cZ_{\eta}}(\F)=\deg_{\Y_{\eta}}(\F).
\end{equation}
By combining \eqref{E:eq1}, \eqref{E:eq1b}, \eqref{E:eq2}, \eqref{E:eq3} and \eqref{E:eq4}, the Claim follows.

\vspace{0.1cm}

The above Claim, together with Remark \ref{R:compEst}, implies that $\un q^{\E_{\ov s}}=\un q^s$.
Therefore, exactly as before, we get that a torsion-free rank-1 sheaf $\I$ on $\X$, flat on $\Spec R_X$,
is (semi)stable with respect to $\E$ in the sense of \cite[Sec. 1.4]{est1} if and only if for every $s\in \Spec R_X$
the restriction $\I_s$ of $\I$ to $\X_{\ov s}$ is (semi)stable with respect to $\un q^s$ in the sense of Definition \ref{sheaf-ss-qs}. Since all the polarizations $\un q^s$ are general by Lemma-Definition
\ref{D:def-pola}, we get that the open subscheme $\J_{\X}(\un q):=J_{\E}^{\rm s}=J_{\E}^{\rm ss}\subset \bJbar_{\X}$ parametrizing sheaves $\I\in \bJbar_{\X}$ whose restriction to $\X_{\ov s}$ is $\un q^s$-semistable (or equivalently
$\un q^s$-stable) is a proper scheme over $\Spec R_X$ by \cite[Thm. A]{est1}. Moreover, $J_{\E}^{\rm s}$ is quasi-projective
over $\Spec R_X$ by \cite[Thm. C]{est1}; hence it is projective over $\Spec R_X$.
The description of the fibers
of $\J_{\X}(\un q)\to \Spec R_X$ is now clear from the definition of $\J_{\X}(\un q)$.

\end{proof}

If the curve $X$ has locally planar singularities, then the universal fine compactified Jacobians of $X$ have several nice properties that we collect in the following statement.

\begin{thm}\label{T:univ-Jac}
Assume that $X$ has locally planar singularities and let $\un q$ be a general polarization on $X$.
Then we have:
\begin{enumerate}[(i)]
\item \label{T:univ-Jac1} The scheme $\J_{\X}(\un q)$ is regular and irreducible.
\item \label{T:univ-Jac2} The surjective map $u:\J_{\X}(\un q)\to \Spec R_X$ is projective and flat of relative dimension $p_a(X)$.
\item \label{T:univ-Jac3} The smooth locus of $u$ is $J_{\X}(\un q)$.
\end{enumerate}
\end{thm}
\begin{proof}
The regularity of $\J_{\X}(\un q)$ follows from Theorem \ref{T:univ-rings}\eqref{T:univ-rings3}. Therefore, in order to show that $\ov{J}_{\X}(\un q)$ is irreducible, it is enough to show that it is connected.
Since the open subset $J_{\X}(\un q)$ is dense by Corollary \ref{C:prop-fineJac}, it is enough to prove that $J_{\X}(\un q)$ is connected. However, this follows easily from
the fact that $J_{\X}(\un q)$ is smooth over $\Spec R_X$ and its generic fiber is the Jacobian of degree $|\un q|$ of a smooth curve, hence it is connected.

Since also $\Spec R_X$ is regular by Theorem \ref{T:univ-rings}\eqref{T:univ-rings2}, the flatness of the map
$u:\J_{\X}(\un q)\to \Spec R_X$ will follow if we show that all the geometric fibers are equi-dimensional of the same dimension
(see \cite[Cor. of Thm 23.1, p. 179]{Mat}). By Theorem \ref{T:univ-fine}, the geometric fiber of $u$
over $s\in \Spec R_X$ is isomorphic to $\J_{\X_{\ov s}}(\un q^s)$ which has pure dimension equal to
$h^1(\X_{\ov s},\O_{\X_{\ov s}})=h^1(X,\O_X)=p_a(X)$ by Corollary \ref{C:prop-fineJac}.

The map $u$ is projective by Theorem \ref{T:univ-fine} and the fact that its smooth locus is equal to $J_{\X}(\un q)$ follows from Corollary \ref{C:prop-fineJac}.

\end{proof}

The above result on the universal fine compactified Jacobians of $X$ has also some very important consequences for the fine
compactified Jacobians of $X$, that we collect in the following two corollaries.

\begin{cor}\label{C:connect}
Assume that $X$ has locally planar singularities and let $\un q$ be a general polarization on $X$. Then $\ov{J}_X(\un q)$ is connected.
\end{cor}
\begin{proof}
Consider the universal fine compactified Jacobian $\ov{J}_{\X}(\un q)$ and the natural surjective morphism $u:\ov{J}_{\X}(\un q)\to \Spec R_X$.
According to Theorem \ref{T:univ-Jac}\eqref{T:univ-Jac2}, $u$ is flat and projective. Therefore, we can apply \cite[Prop. (15.5.9)]{EGAIV3} which says that the number of
connected components of the geometric fibers of $u$ is lower semicontinuous. Since the generic geometric fiber of $u$ is the Jacobian of a smooth curve (by Theorem \ref{T:univ-fine}), hence connected, we deduce that also the fiber over the closed point $[\m_X]\in \Spec R_X$, which is $\ov J_X(\un q)$ by Theorem \ref{T:univ-fine}, is connected, q.e.d.
\end{proof}

\begin{cor}\label{C:triv-can}
Assume that $X$ has locally planar singularities and let $\un q$ be a general polarization on $X$.
Then the universal fine compactified Jacobian $u:\J_{\X}(\un q)\to \Spec R_X$ (with respect to the polarization $\un q$) has trivial relative dualizing sheaf. In particular,
$\J_{X}(\un q)$ has trivial dualizing sheaf.
\end{cor}
\begin{proof}
Observe that the relative dualizing sheaf, call it $\omega_u$, of the universal fine compactified Jacobian $u:\J_{\X}(\un q)\to \Spec R_X$
is a line bundle because the fibers of $u$ have l.c.i. singularities by Theorem \ref{T:univ-fine} and Corollary \ref{C:prop-fineJac}.

Consider now the open subset $U\subseteq \Spec R_X$ consisting of those points $s\in \Spec R_X$ such that the geometric fiber $\X_{\ov s}$ over $s$ has at most a unique singular point which is a
node (as in Lemma \ref{L:codim1}).

\un{CLAIM:} $(\omega_u)_{|u^{-1}(U)}=\O_{u^{-1}(U)}$.

Indeed, Theorem \ref{T:univ-fine} implies that the geometric fiber of $\ov J_{\X}(\un q)\to \Spec R_X$ over a point $s$ is isomorphic to $\ov J_{\X_{\ov s}}(\un q^s)$.
If $\X_{\ov s}$ is smooth or if it has a separating node, then $\ov J_{\X_{\ov s}}(\un q^s)$ is an abelian variety, hence it has trivial dualizing sheaf. If $\X_{\ov s}$ is irreducible with a node then
$\ov J_{\X_{\ov s}}(\un q^s)$ has trivial dualizing sheaf by \cite[Cor. 9]{arin1}.
Therefore, the fibers of the proper map $u^{-1}(U) \to U $ have trivial canonical sheaf. It follows that $u_*(\omega_u)_{|U}$ is a line bundle on $U$ and that the natural evaluation morphism
$u^*u_*(\omega_u)_{|u^{-1}(U)}\to (\omega_u)_{|u^{-1}(U)}$ is an isomorphism. Since $\Pic(U)=0$, the line bundle $u_*(\omega_u)_{|U}$ is trivial, hence also $(\omega_u)_{|u^{-1}(U)}$ is trivial, q.e.d.


The above Claim implies that $\omega_u$ and $\O_{\J_{\X}(\un q)}$ agree on an open subset $u^{-1}(U)\subset \ov J_{\X}(\un q)$ whose complement has
codimension at least two by Lemma \ref{L:codim1}. Since $\J_{\X}(\un q)$ is regular (hence $S_2$) by Theorem \ref{T:univ-Jac},
this implies that $\omega_u=\O_{\J_{\X}(\un q)}$.

The second assertion follows now by restricting the equality  $\omega_u=\O_{\J_{\X}(\un q)}$ to the fiber
$\J_X(\un q)$ of $u$ over the closed point $[\m_X]\in \Spec R_X$.
\end{proof}

Note that a statement similar to Corollary \ref{C:triv-can} was proved  by Arinkin in \cite[Cor. 9]{arin1} for the universal compactified Jacobian over the moduli stack of \emph{integral} curves
with locally planar singularities.

\vspace{0,2cm}

Finally, note that the universal fine compactified Jacobians are acted upon by the universal generalized Jacobian, whose properties are collected into the following

\begin{fact}[Bosch-L\"utkebohmert-Raynaud]\label{F:ungenJac}
There is an open subset of $\bJ_{\X}$, called the \emph{universal generalized Jacobian} of $\pi:\X\to \Spec R_X$
and denoted by $v: J(\X) \to \Spec R_X$, whose geometric fiber over any point $s\in \Spec R_X$ is the generalized Jacobian
$J(\X_{\ov s})$ of the geometric fiber $\X_{\ov s}$ of $\pi$ over $s$.

The morphism $v$ makes $J(\X)$ into a smooth and separated group scheme of finite type over $\Spec R_X$.
\end{fact}
\begin{proof}
The existence of a group scheme $v:J(\X)\to \Spec R_X$ whose fibers are the generalized Jacobians of the fibers of $\pi:\X\to \Spec R_X$
follows by \cite[Sec. 9.3, Thm. 7]{BLR}, which can be applied since $\Spec R_X$
is a strictly henselian local scheme (because $R_X$ is a complete local ring) and the geometric fibers of $\pi:\X\to \Spec R_X$ are reduced and connected since $X$ is assumed
to be so. The result of loc. cit. gives also that the map $v$ is smooth, separated and of finite type.
\end{proof}

\subsection{1-parameter regular smoothings of $X$}\label{S:1par-sm}

The aim of this subsection is to study relative fine compactified Jacobians associated to a 1-parameter smoothing of a curve $X$ and their relationship with the N\'eron models of the Jacobians of the generic fiber.
As a corollary, we will get a combinatorial formula for the number of irreducible components of a fine compactified
Jacobian of a curve with locally planar singularities.

Let us start with the definition of  1-parameter regular smoothings of a curve $X$.

\begin{defi}\label{D:1par-sm}
A \emph{1-parameter  regular smoothing} of $X$ is a proper and flat morphism $f:\SS\to B=\Spec R$ where $R$ is a complete discrete valuation domain (DVR for short)
with residue field $k$ and quotient field $K$ and $\SS$ is a regular scheme of dimension two, i.e. a regular surface, and
such that the special fiber  $\SS_k$ is isomorphic to $X$ and the generic fiber $\SS_K$ is a $K$-smooth curve.
\end{defi}

The natural question one may ask is the following: which (reduced) curves $X$ admit a 1-parameter regular smoothing?
Of course, if $X$ admits a 1-parameter regular smoothing $f:\SS\to \Spec R$, then $X$ is a divisor inside a regular surface $\SS$, which implies
that $X$ has locally planar singularities.
Indeed, it is well known to the experts that this necessary condition turns out to be  also sufficient. We include a proof here since we couldn't find a suitable reference.

\begin{prop}\label{P:1parsmooth}
A (reduced) curve $X$ admits a 1-parameter regular smoothing  if and only if $X$ has locally planar singularities.
More precisely, if $X$ has locally planar singularities then there exists a complete discrete valuation domain $R$ (and indeed we can take $R=k[[t]]$) and a morphism $\alpha:\Spec R\to \Spec R_X$ such that
the pull-back
\begin{equation}\label{E:1par-pull}
\xymatrix{
\SS \ar[r]\ar[d]_{f} \ar@{}[dr]|{\square}& \X \ar[d]^{\pi} \\
\Spec R \ar[r]^{\alpha} & \Spec R_X
}
\end{equation}
is a 1-parameter regular smoothing of $X$. 
\end{prop}

\begin{proof}
We have already observed that the only if condition is trivially satisfied. Conversely, assume that $X$ has locally planar singularities, and let us prove that $X$ admits a 1-parameter regular smoothing.

Consider the natural morphisms of deformation functors
\begin{equation*}
F:h_{R_X}\to \Def_X \to \Def_X^{\rm loc}=\prod_{p\in X_{\rm sing}} \Def_{X,p}=\prod_{p\in X} \Def_{X,p},
\end{equation*}
obtained by composing the morphism \eqref{E:mor-func} with the morphism \eqref{E:map-func1} and using the fact if $p$ is a smooth point of $X$ then $\Def_{X,p}$ is the trivial deformation functor 
(see \cite[Thm. 1.2.4]{Ser}). 
Observe that $F$ is smooth because the first morphism is smooth since $R_X$ is a semiuniversal deformation
ring for $\Def_X$ and the second morphism is smooth by Fact \ref{F:for-smooth}\eqref{F:for-smooth3}.

Given an element $\alpha\in h_{R_X}(R)=\Hom(\Spec R, \Spec R_X)$ associated to a Cartesian diagram like in
\eqref{E:1par-pull}, the image of $\alpha$ into $\Def_{X,p}(R)$ corresponds to the formal deformation of
$\wh{\O}_{X,p}$ given by the right square of the following diagram 
\begin{equation*}
\xymatrix{
\Spec \wh{\O}_{X,p} \ar[r] \ar[d]  \ar@{}[dr]|{\square}& \Spf \wh{\O}_{\SS,p}\ar[d]\ar[r] \ar[d]  \ar@{}[dr]|{\square}& \Spec \wh{\O}_{\SS,p}\ar[d] \\
\Spec k \ar[r] & \Spf R  \ar[r] & \Spec R.
}
\end{equation*}

\un{Claim 1:} The morphism $f:\SS\to \Spec R$ is a 1-parameter regular smoothing of $X$ if and only if, for any $p\in X$, we have that 
\begin{enumerate}[(i)]
\item \label{cond1} $\wh{\O}_{\SS,p}$ is regular;
\item \label{cond2} $\wh{\O}_{\SS,p}\otimes_R K$ is geometrically regular over $K$ (i.e. $\wh{\O}_{\SS,p}\otimes_R K'$ is regular for any field extension $K\subseteq K'$).
\end{enumerate}

Indeed, by definition, the surface $\SS$ is regular if and only if the local ring $\O_{\SS,q}$ is regular for any schematic point $q\in \SS$ or, equivalently (see \cite[Thm. 19.3]{Mat}), for any closed point $q\in \SS$. 
Clearly, the closed points of $\SS$ are exactly the closed points of its special fiber $\SS_k=X$. Moreover the local ring $\O_{S, q}$ is regular if and only if its completion 
$\wh{\O}_{\SS, q}$ is regular 
(see \cite[Thm. 21.1(i)]{Mat}).  
Putting everything together, we deduce that $\SS$ is regular if and only if \eqref{cond1} is satisfied.  

Consider now the $\Spec R$-morphism   $\mu:\coprod_{p\in X} \Spec \wh{\O}_{\SS,p}\stackrel{\mu''}{\longrightarrow} \coprod_{p\in X} \Spec \O_{\SS, p} \stackrel{\mu'}{\longrightarrow} \SS$. 
The morphism $\mu$ is flat since 
any localization  is flat (see \cite[Thm. 4.5]{Mat}) and any completion  is flat (see \cite[Thm. 8.8]{Mat}). Moreover, the image of $\mu$ is open because $\mu$ is flat and it contains the special fiber 
$\SS_k=X\subset \SS$ which contains all the closed points of $\SS$; therefore, $\mu$ must be surjective, which implies that $\mu$ is faithfully flat. Finally, $\mu$ has geometrically regular fibers (hence it is regular, i.e.
flat with geometrically regular fibers, see \cite[(33.A)]{Mat0}): this is obvious for $\mu'$ (because it is the disjoint union of localization morphisms), it is true for $\mu''$ because each local ring $\O_{\SS, p}$ is a G-ring (in the sense of \cite[\S 33]{Mat0}) being the localization of a scheme of finite type over a complete local ring (as it follows from \cite[Thm. 68, Thm. 77]{Mat0}) and the composition of regular morphisms is regular (see \cite[(33.B), Lemma 1(i)]{Mat0}). By base changing the morphism $\mu$ to the generic point $\Spec K$ of $\Spec R$, we get a morphism $\mu_K:\coprod_{p\in X} \Spec (\wh{\O}_{\SS,p}\otimes_R K)  \to \SS_K$ which is also faithfully flat and regular, because both properties are stable under base change. 
Therefore, by applying \cite[(33.B), Lemma 1]{Mat0}, we deduce that $\SS_K$ is geometrically regular over $K$ if and only if $\wh{\O}_{\SS,p}\otimes_R K$ is geometrically regular over $K$ for any $p\in X$. 
Hence, $\SS_K$ is $K$-smooth (which is equivalent to the fact that $\SS_K$ is geometrically regular over $K$, because $\SS_K$ is of finite type over $K$ by assumption) if and only if \eqref{cond2} is satisfied, q.e.d.

\vspace{0.1cm}

Suppose now that for any $p\in X$ we can find  an element of $\Def_{X,p}(R)$ corresponding to a formal deformation
\begin{equation}\label{E:fordiag}
\xymatrix{
\Spec \wh{\O}_{X,p} \ar[r] \ar[d]  \ar@{}[dr]|{\square}& \Spf \A\ar[d]\ar[r]  \ar@{}[dr]|{\square}& \Spec \A \ar[d]  \\
\Spec k \ar[r] & \Spf R \ar[r] & \Spec R
}
\end{equation}
such that $\A$ is a regular complete local ring, $R\to \A$ is a local flat morphism and $\A\otimes_R K$ is $K$-formally smooth. Then, using the smoothness of $F$, we can lift this element
to an element $\alpha\in h_{R_X}(R)$ whose associated Cartesian diagram \eqref{E:1par-pull} gives rise to a $1$-parameter regular smoothing of $X$ by the above Claim.

Let us now check this local statement. 
Since $X$ has locally planar (isolated) singularities, we can write
$$\wh{\O}_{X,p}=\frac{k[[x,y]]}{(f)},$$
for some reduced element $0\neq f=f(x,y)\in (x,y)\subset k[[x,y]]$. 

\un{Claim 2:} Up to replacing $f$ with $fH$ for some  invertible element $H\in k[[x,y]]$, we can assume that 
\begin{equation}\label{E:cond-deriv}
\partial_x f \: \text{ and } \: \partial_y f  \text{ do not have common irreducible factors},
\end{equation}
where $\partial_x$ is the formal partial derivative with respect to $x$ and similarly for $\partial_y$.

More precisely, we will show that there exists $a,b\in k$ with the property that $\wt{f}:=(1+ax+by)f$ satisfies the conclusion of the Claim, i.e.  $\partial_x \wt{f}$   and  $\partial_y \wt{f}$   do not have common irreducible factors;
 it will then follow that the same is true for a generic point $(a,b)\in {\mathbb A}^2(k)$.
By contradiction, assume that 
\begin{equation*}\tag{*}
(1+ax+by)\partial_x f +a f \: \text{ and } \: (1+ax+by)\partial_y f +b f \: \: \text{ have a common irreducible factor  for every } a,b\in k. 
\end{equation*}
Observe that $(\partial_x f, \partial_y f)\neq (0,0)$, for otherwise 
$f$ would be a $p$-power in $k[[x,y]]$ with $p={\rm char}(k)>0$, which contradicts the fact that $f$ is reduced. So we can assume, without loss of generality that $\partial_y f\neq 0$.  We now specialize condition (*) by putting $b=0$ and using 
that $1+ax$ is invertible in $k[[x,y]]$, in order to get that  
\begin{equation*}\tag{**}
(1+ax)\partial_x f +a f \: \text{ and } \: \partial_y f \: \: \text{ have a common irreducible factor  for every } a\in k. 
\end{equation*}
Since $k$ is an infinite field (being algebraically closed) and $0\neq \partial_y f$ has, of course, a finite number of irreducible factors, we infer from (**) that there exists an irreducible factor  $q\in k[[x,y]]$ of $\partial_y f$ such that 
$q$ is also an irreducible factor of $(1+ax)\partial_x f+a f=\partial_x f+a(x\partial_x f +f)$ for infinitely many $a\in k$. This however can happen (if and) only if $q$ divides $f$ and $\partial_x f$. 
This implies that the hypersurface $\{f=0\}$ is singular along the entire irreducible component $\{q=0\}$, which contradicts the hypothesis that $\{f=0\}$ has an isolated singularities in $(0,0)$, q.e.d.

\vspace{0.1cm}

From now on, we will assume that $f$ satisfies the conditions of \eqref{E:cond-deriv}.   Let  $R:=k[[t]]$ and consider the local complete $k[[t]]$-algebra 
$$\displaystyle \A:=\frac{k[[x,y,t]]}{(f-t)}.$$
The $k[[x,y]]$-algebra homomorphism (well-defined since $f\in (x,y)$)
\begin{equation}\label{E:isoA}
\begin{aligned}
\A=\frac{k[[x,y,t]]}{(f-t)} & \longrightarrow k[[x,y]] \\
t & \mapsto f 
\end{aligned}
\end{equation}
 is clearly an isomorphism. Therefore, $\A$ is a regular local ring. 
Moreover, since $f $ is not a zero-divisor in $k[[x,y]]$, the algebra $\A$ is flat over $k[[t]]$. From now on, we will use the isomorphism \eqref{E:isoA} to freely identify $\A$ with $k[[x,y]]$ seen as a 
$k[[t]]$-algebra via the map sending $t$ into $f$.

It remains to show that $\A\otimes_{k[[t]]} k((t))$ is geometrically regular over $k((t))$.
Since $\A\otimes_{k[[t]]} k((t))$ is the localization of $\A$ at the multiplicative system generated by $(t)$, we have to check  (by \cite[Def. in (33.A) and Prop. in (28.N)]{Mat0})) that, 
for any ideal $\m$ in the fiber of $\A$ over the generic point of $k[[t]]$ (i.e. such that $\m\cap k[[t]]=(0)$), the local ring $\A_{\m}$ is  formally smooth over $k((t))$ 
for the $\m$-adic topology on $\A_m$ and the discrete topology on $k((t))$ (see \cite[(28.C)]{Mat0} for the definition of formal smoothness). 
Since formal smoothness is preserved under localization (as it follows easily from \cite[(28.E) and (28.F)]{Mat0}), it is enough to prove that $\A_{\m}$ is  formally smooth over $k((t))$ 
for any closed point $\m$ of $\A\otimes_{k[[t]]} k((t))$.
The closed points of  $\A\otimes_{k[[t]]} k((t))$ correspond exactly to those prime ideals of $\A\cong k[[x,y]]$ of the form $\m=(g)$ for some irreducible element $g\in k[[x,y]]$ that is not a factor of $f$. 
Indeed, any such ideal $\m$ of $k[[x,y]]$ must be of height one and hence it must be principal (since $k[[x,y]]$ is regular, see \cite[Thm. 20.1, Thm. 20.3]{Mat}), i.e. $\m=(g)$ for some $g$ irreducible element of $k[[x,y]]$.
Furthermore, the condition $(g)\cap k[[t]]=(0)$ is satisfied if and only if $g$ is not an irreducible factor of $f$. Therefore, we are left with proving the following 

 \un{Claim 3:}  $k[[x,y]]_{(g)}$ is formally smooth over $k((t))$ for any irreducible $g\in k[[x,y]]$ that is not an irreducible factor of $f$. 

Observe first of all that $k[[x,y]]_{(g)}$ is formally smooth over $k$ because $k[[x,y]]$ is formally smooth over $k$ (see \cite[(28.D), Example 3]{Mat0}) and formal smoothness is preserved by localization as observed before.
Therefore, $k[[x,y]]_{(g)}$ is regular (see \cite[Thm. 61]{Mat0}).

Consider now the residue field $L=k[[x,y]]_{(g)}/(g)$ of the local ring $k[[x,y]]_{(g)}$, which is a field extension of $k((t))$.  If $L$ is a separable extension of $k((t))$ (which is always the case if ${\rm char}(k)=0$) then $k[[x,y]]_{(g)}$  is
formally smooth over $k((t))$ by \cite[(28.M)]{Mat0}. In the general case, using \cite[Thm. 66]{Mat},  the Claim is equivalent to the \emph{injectivity}  of the natural $L$-linear map
\begin{equation}\label{E:difmap}
\alpha: \Omega^1_{k((t))/k}\otimes_{k((t))} L \to  \Omega^1_{k[[x,y]]_{(g)}/k}\otimes_{k[[x,y]]_{(g)}} L= \Omega^1_{k[[x,y]]/k}\otimes_{k[[x,y]]} L,
\end{equation}
where in the second equality we have used that the K\"ahler differentials commute with the localization together with the base change for the tensor product. 
Therefore, to conclude the proof of the Claim, it is enough to prove the injectivity of the map \eqref{E:difmap} if ${\rm char}(k)=p>0$. Under this assumption (and recalling that $k$ is assumed to be algebraically closed), we have clearly 
that $k((t))^p=k((t^p))\subset k((t))$, from which it follows that $t$ is a p-basis of $k((t))/k$ in the sense of \cite[\S 26]{Mat}. Therefore, using \cite[Thm. 26.5]{Mat}, we deduce that $\Omega^1_{k((t))/k}$ is the $k((t))$-vector space of 
dimension one generated by  $dt$. Hence, the injectivity of the above $L$-linear map $\alpha$ translates into $\alpha(dt\otimes 1)\neq 0$. Since the natural map $k[[t]]\to k[[x,y]]$ sends $t$ into $f$, we can compute 
\begin{equation}\label{E:Imdt}
\alpha(dt\otimes 1)=d(f)\otimes 1 =(\partial_x f dx+\partial_y f dy)\otimes 1\in  \Omega^1_{k[[x,y]]/k}\otimes_{k[[x,y]]} L.
\end{equation}
Now observe that the map $k[[x,y]]\to k[[x,y]]_{(g)}\to L=k[[x,y]]_{(g)}/(g)$ is also equal to the composition $k[[x,y]]\to k[[x,y]]/(g)\to {\rm Quot}(k[[x,y]]/(g))\cong L$, where ${\rm Quot}$ denotes the quotient field. Moreover, since $dx$ and $dy$ 
generate a free  rank-$2$ submodule of the $k[[x,y]]$-module $\Omega^1_{k[[x,y]]}$, the right hand side of \eqref{E:Imdt} is zero if and only if $g$ divides both $\partial_x f$ and  $\partial_y f$.
Since this does not happen for our choice of $f$ (see Claim 1), the proof is complete. 
 



\end{proof}

From now till the end of this subsection, we fix a 1-parameter regular smoothing $f:\SS\to B=\Spec R$ of $X$ as in Proposition \ref{P:1parsmooth}.
Let ${\mathcal Pic} _f $ denote the \emph{relative Picard functor} of $f$ (often denoted
$ {\mathcal Pic} _{\SS/B}$ in the literature, see \cite[Chap. 8]{BLR} for the general theory) and let
$  {\mathcal Pic} _f^d$ be the subfunctor parametrizing line bundles of relative degree $d\in \Z$.
The functor ${\mathcal Pic} _f$ (resp. ${\mathcal Pic}_f^d$) is represented by a scheme $\Pic_f$ (resp. $\Pic^d_f$) locally of finite presentation over $B$ (see
\cite[Sec. 8.2, Thm. 2]{BLR}) and formally smooth over $B$ (by \cite[Sec. 8.4, Prop. 2]{BLR}).
The generic fiber of $\Pic_f$ (resp. $\Pic^d_f$) is isomorphic to $\Pic(\SS_K)$ (resp. $\Pic^d(\SS_K)$).

Note that $\Pic_f^d$ is not separated over $B$ whenever $X$ is reducible. The biggest separated quotient of $\Pic_f^d$ coincides with the \emph{N\'eron model}
$N(\Pic^d {\SS}_K)$ of $\Pic^d(\SS_K)$, as proved by Raynaud in \cite[Sec. 8]{Ray}
(see also \cite[Sec. 9.5, Thm. 4]{BLR}). Recall that $N(\Pic^d {\mathcal S}_K)$ is smooth and separated over $B$, the generic fiber
$N(\Pic^d {\SS}_K)_K$ is isomorphic to $\Pic^d \SS_K$ and $N(\Pic^d {\SS}_K)$ is uniquely characterized by the
following universal property (the N\'eron mapping property, cf.  \cite[Sec. 1.2, Def. 1]{BLR}):
every $K$-morphism $q_K:Z_K\rightarrow N(\Pic^d {\SS}_K)_K=\Pic^d \SS_K$ defined on  the generic fiber of some scheme
$Z$ smooth over $B$ admits a unique extension  to a $B$-morphism $q:Z\rightarrow N(\Pic^d {\SS}_K)$.
Moreover, $N(\Pic^0 {\SS}_K)$ is a $B$-group scheme while, for every $d\in \Z$, $N(\Pic^d S_K)$ is a torsor under $N(\Pic^0 {\SS}_K)$.

Fix now a general polarization $\un q$ on $X$ and consider the Cartesian diagram
\begin{equation}\label{E:Jacf}
\xymatrix{
\J_f(\un q) \ar[r]\ar[d] \ar@{}[dr]|{\square}& \J_{\X}(\un q) \ar[d]^{u} \\
B=\Spec R \ar[r]^{\alpha} & \Spec R_X
}
\end{equation}
We call the scheme $\J_f(\un q)$ the \emph{$f$-relative fine compactified Jacobian}  with respect to the polarization $\un q$.
Similarly, by replacing $\J_{\X}(\un q)$ with $J_{\X}(\un q)$ in the above diagram, we define the open subset $J_f(\un q)\subset \J_f(\un q)$.
Note  that the generic fibers of $\J_f(\un q)$ and $J_f(\un q)$ coincide and are  equal to $\J_f(\un q)_K=J_f(\un q)_K=\Pic^{d}(\SS_K)$ with $d:=|\un q|+p_a(X)-1$, while their special
fibers are equal to  $\J_f(\un q)_k=\J_X(\un q)$ and $J_f(\un q)_k=J_X(\un q)$, respectively.
From Theorem \ref{T:univ-Jac}, we get that the morphism $\J_f(\un q)\to B$ is flat and its smooth locus is $J_f(\un q)$.
Therefore, the universal property of the N\'eron model $N(\Pic^{d}\SS_K)$ gives a unique $B$-morphism $q_f:J_f(\un q)\to N(\Pic^{d}\SS_K)$ which is the identity on
the generic fiber. Indeed, J. L. Kass proved in \cite[Thm. A]{Kas} that the above map is an isomorphism.

\begin{fact}[Kass]\label{F:Jac-Ner}
For a 1-paramater regular smoothing $f:\X\to B=\Spec R$ as above and any general polarization $\un q$ on $X$, the natural $B$-morphism
$$q_f:J_f(\un q)\to N(\Pic^{|\un q|+p_a(X)-1}\SS_K)$$
is an isomorphism.
\end{fact}

From the above isomorphism, we can deduce a formula for the number of irreducible components of $\J_X(\un q)$.
We first need to recall the description due to Raynaud of the group of connected components of the N\'eron model
$N(\Pic^0(\SS_K))$ (see \cite[Sec. 9.6]{BLR} for a detailed exposition).

Denote by $C_1,\ldots, C_{\gamma}$ the irreducible components of  $X$.
A \emph{multidegree} on $X$ is an ordered $\gamma$-tuple of integers
$$\un d=(\un d_{C_1}, \ldots, \un d_{C_{\gamma}})\in \Z^{\gamma}.$$
We denote by $|\un d|=\sum_{i=1}^{\gamma} \un d_{C_i}$ the total degree of $\un d$.
We now define an equivalence relation on the set of multidegrees on $X$.
For every irreducible component $C_i$ of $X$, consider the multidegree $\un{C_i}=((\un{C_i})_1,\ldots, (\un{C_i})_{\gamma} )$ of total degree $0$ defined by
$$(\un{C_i})_j=\begin{sis}
&|C_i\cap C_j| &  \text{ if } i\neq j,\\
&-\sum_{k\neq i} |C_i\cap C_k| & \text{ if } i=j,
\end{sis}
$$
where $|C_i\cap C_j|$ denotes the length of the scheme-theoretic intersection of $C_i$ and $C_j$.
Clearly, if we take a 1-parameter regular smoothing $f:\SS\to B$ of $X$ as in Proposition \ref{P:1parsmooth}, then
$|C_i\cap C_j|$ is also equal to the intersection product of the two divisors $C_i$ and $C_j$ inside the regular
surface $\SS$.

Denote by $\Lambda_X\subseteq \Z^{\gamma}$ the subgroup of $\Z^{\gamma}$ generated by the multidegrees $\un{C_i}$ for $i=1,\ldots,\gamma$.
It is easy to see that $\sum_i \un{C_i}=0$ and this is the only relation among the multidegrees $\un{C_i}$. Therefore,
$\Lambda_X$ is a free abelian group of rank $\gamma-1$.

\begin{defi}\label{D:equiv-mult}
Two multidegrees $\un d$ and $\un d'$ are said to be \emph{equivalent}, and we write $\un d\equiv \un d'$,  if $\un d-\un d'\in \Lambda_X$.
In particular, if $\un d\equiv \un d'$ then $|\un d|=|\un d'|$.

For every $d\in \Z$, we denote by $\Delta_X^d$ the set of equivalence classes of multidegrees of total degree $d=|\un d|$.
Clearly $\Delta_X^0$ is a finite group under component-wise addition of multidegrees (called the \emph{degree class group } of $X$) and each $\Delta_X^d$ is a torsor under $\Delta_X^0$. The cardinality of $\Delta_X^0$ is called
the \emph{degree class number} or the \emph{complexity} of $X$, and it is denoted by $c(X)$.
\end{defi}
The name degree class group was first introduced by L. Caporaso in \cite[Sec. 4.1]{Cap}. The name complexity
comes from the fact that if $X$ is a nodal curve then $c(X)$ is the complexity of the dual graph $\Gamma_X$
of $X$, i.e. the number of spanning trees of $\Gamma_X$ (see e.g. \cite[Sec. 2.2]{MV}).

\begin{fact}[Raynaud]\label{F:complNer}
The group of connected component of the $B$-group scheme $N(\Pic^0 \SS_K)$ is isomorphic to $\Delta^0_X$.
In particular, the special fiber of $N(\Pic^d \SS_K)$ for any $d\in \Z$ is isomorphic to the disjoint union of $c(X)$
copies of the generalized Jacobian $J(X)$ of $X$.
\end{fact}

For a proof, see the original paper of Raynaud \cite[Prop. 8.1.2]{Ray} or \cite[Sec. 9.6]{BLR}.

Finally, by combining Corollary \ref{C:prop-fineJac}, Fact \ref{F:Jac-Ner} and Fact \ref{F:complNer},
we obtain a formula for the number of irreducible components of a fine compactified Jacobian.

\begin{cor}\label{C:irre-comp}
Assume that $X$ has locally planar singularities and let $\un q$ be any general polarization on $X$.
Then $\J_X(\un q)$ has $c(X)$ irreducible components.
\end{cor}

The above Corollary was obtained for nodal curves $X$ by S. Busonero in his PhD thesis (unpublished) in a combinatorial way; a slight variation of his proof is given  in  \cite[Sec. 3]{MV}.

\vspace{0.1cm}

Using the above Corollary, we can now prove the converse of Lemma \ref{L:nondeg}\eqref{L:nondeg1} for curves with locally planar singularities, generalizing the result of \cite[Prop. 7.3]{MV}  for nodal curves.

\begin{lemma}\label{L:ndg-conv}
Assume that $X$ has locally planar singularities. For a polarization $\un q$ on $X$, the following conditions are equivalent
\begin{enumerate}[(i)]
\item \label{L:ndg-conv1} $\un q$ is general.
\item \label{L:ndg-conv2} Every rank-1 torsion-free sheaf which is $\un q$-semistable is also $\un q$-stable.
\item \label{L:ndg-conv3} Every simple rank-1 torsion-free sheaf which is $\un q$-semistable is also $\un q$-stable.
\item \label{L:ndg-conv4} Every line bundle which is $\un q$-semistable is also $\un q$-stable.
\end{enumerate}
\end{lemma}
\begin{proof}
\eqref{L:ndg-conv1} $\Rightarrow$ \eqref{L:ndg-conv2} follows from Lemma \ref{L:nondeg}\eqref{L:nondeg1}.   

\eqref{L:ndg-conv2} $\Rightarrow$ \eqref{L:ndg-conv3} $\Rightarrow$ \eqref{L:ndg-conv4} are trivial.

\eqref{L:ndg-conv4} $\Rightarrow$ \eqref{L:ndg-conv1}: by contradiction, assume that $\un q$ is not general. Then, by Remark \ref{R:conn-pola}, we can find a proper subcurve $Y\subseteq X$  with $Y$ and $Y^c$ connected and such that
$\un q_Y, \un q_{Y^c}\in \Z$. This implies that we can define a polarization $\un q_{|Y}$ on the connected curve $Y$ by setting $(\un q_{|Y})_Z:=\un q_Z$ for any subcurve $Z\subseteq Y$.
And similarly, we can define a polarization $\un q_{|Y^c}$ on $Y^c$.

\un{Claim 1:} There exists a line bundle $L$ such that $L_{|Y}$ is $\un q_{|Y}$-semistable and $L_{|Y^c}$ is  $\un q_{|Y^c}$-semistable.

Clearly, it is enough to show the existence of a line bundle $L_1$ (resp. $L_2$) on $Y$ (resp. on $Y^c$) that is $\un q_{|Y}$-semistable (resp. $\un q_{|Y^c}$-semistable); the line bundle $L$ that satisfies the conclusion of the Claim is then any line bundle such that $L_{|Y}=L_1$ and $L_{|Y^c}=L_2$ (clearly such a line bundle exists!).
Let us prove the existence of $L_1$ on $Y$; the argument for $L_2$ being the same.   We can deform slightly the polarization $\un q_{|Y}$ on $Y$ in order to obtain a general polarization $\wt{\un q}$ on $Y$.  Corollary \ref{C:irre-comp} implies that  $\J_Y(\wt{\un q})$ is non empty (since $c(Y)\geq 1$); hence also $J_Y(\wt{\un q})$ is non empty by Corollary \ref{C:prop-fineJac}\eqref{C:prop-fineJac2}.  In particular, we can find a line bundle $L_1$ on $Y$ that is $\wt{\un q}$-semistable. Remark \ref{R:stable-general} implies that $L_1$ is also $\un q_{|Y}$-semistable, and the Claim is proved.

Observe that the line bundle $L$ is not $\un q$-stable since $\chi(L_Y)=\chi(L_{|Y})=|\un q_{|Y}|=\un q_Y$ and similarly for $Y^c$.
Therefore, we find the desired contradiction by proving the following

\un{Claim 2:} The line bundle $L$ is $\un q$-semistable.

Let $Z$ be a subcurve of $X$ and set $Z_1:=Z\cap Y$,  $Z_2:=Z\cap Y^c$, $W_1:=\ov{Y\setminus Z_1}$ and $W_2:= \ov{Y^c\setminus Z_2}$. Tensoring with $L$ the exact sequence
$$0\to \O_Z\to \O_{Z_1}\oplus \O_{Z_2} \to \O_{Z_1\cap Z_2}\to 0,$$
and  taking the Euler-Poincar\'e characteristic we get
\begin{equation*}
\chi(L_{|Z})=\chi(L_{|Z_1})+\chi(L_{|Z_2})-|Z_1\cap Z_2|.
\end{equation*}
Combining the above equality with the fact that $L_{|Y}$ (resp. $L_{|Y^c}$) is $\un q_{|Y}$-semistable (resp. $\un q_{|Y^c}$-semistable) by Claim 1 and using Remark \ref{R:ss-lb}, we compute
\begin{equation}\label{E:chiZ}
\chi(L_{|Z})=\chi(L_{|Z_1})+\chi(L_{|Z_2})-|Z_1\cap Z_2|\leq \un q_{Z_1}+|Z_1\cap W_1| +\un q_{Z_2}+|Z_2\cap W_2| -|Z_1\cap Z_2|\leq
\end{equation}
$$\leq \un q_Z+|Z_1\cap W_1|  +|Z_2\cap W_2|.$$
Since $X$ has locally planar singularity, we can embed $X$ inside a smooth projective surface $S$ (see \ref{N:lps}). In this way, the number $|Z_i\cap W_i|$ is equal to the intersection number $Z_i\cdot W_i$ of the divisors $Z_i$ and $W_i$ inside the smooth projective surface $S$. Using that the intersection product of divisors in $S$ is bilinear, we get that
\begin{equation}\label{E:intZ}
|Z\cap Z^c|=Z\cdot Z^c=(Z_1+Z_2)\cdot (W_1+W_2)\geq Z_1\cdot W_1+Z_2\cdot W_2=|Z_1\cap W_1| +|Z_2\cap W_2|,
\end{equation}
where we used that $Z_1\cdot W_2\geq 0$ because $Z_1$ and $W_2$ do not have common components and similarly  $Z_2\cdot W_1\geq 0$.
Substituting \eqref{E:intZ} into \eqref{E:chiZ}, we find that
$$\chi(L_{|Z})\leq \un q_Z+|Z\cap Z^c|,
$$
for every subcurve $Z\subseteq X$, which implies that $L$ is $\un q$-semistable  by Remark \ref{R:ss-lb}.  

\end{proof}

It would be interesting to know if the above Lemma \ref{L:ndg-conv} holds true for every (reduced) curve $X$.

\section{Abel maps}\label{S:Abel}

The aim of this section is to define Abel maps for  singular (reduced) curves. Although in the following sections we will only use Abel maps for curves with locally planar
singularities,  the results of this section are valid for a broader class of singular curves, namely those for which every separating point is a node (see Condition \eqref{E:dagger}),
which includes for example all Gorenstein curves.

\subsection{Abel maps without separating points}\label{S:Abel-nonsep}

The aim of this subsection is to define the Abel maps for a reduced curve $X$ without separating points (in the sense
of \ref{N:sep-node}).

For every (geometric) point $p$ on the curve $X$, its sheaf of ideals $\m_p$ is a torsion-free rank-1 sheaf of degree
$-1$ on $X$. Also, if $p$ is not a separating point of $X$, then $\m_p$ is simple (see \cite[Example 38]{est1}).
Therefore, if $X$ does not have separating points (which is clearly equivalent to the fact that $\delta_Y\geq 2$ for every proper subcurve $Y$ of $X$) then $\m_p$ is torsion-free rank-1 and simple for any $p\in X$.

Let $\I_\Delta$ be the ideal of the diagonal $\Delta$ of $X\times X$. For any line bundle $L\in \Pic(X)$,
consider the sheaf $\I_\Delta\otimes p_1^*L,$
where $p_1:X\times X \to X$ denotes the projection onto the first factor. The sheaf $\I_{\Delta}\otimes p_1^*L$ is flat over $X$
(with respect to the second projection $p_2:X\times X \to X$) and for any point $p$ of $X$ it holds
$$\I_{\Delta}\otimes p_1^*L_{|X\times \{p\}}=\m_p\otimes L.$$
Therefore, if $X$ does not have separating points, then $\I_{\Delta}\otimes p_1^*L\in \bJbar_X^*(X)$ where $\bJbar_X^*$ is the functor
defined by \eqref{E:func-Jbar}. Using the universal property of $\bJbar_X$ (see Fact \ref{F:huge-Jac}\eqref{F:huge3}),
the sheaf $\I_{\Delta}\otimes p_1^*L$ induces a morphism
\begin{equation}\label{D:abel-map}
\begin{aligned}
A_L:X & \to \bJbar_X\\
p & \mapsto \m_p\otimes L.
\end{aligned}
\end{equation}
We call the map $A_L$ the \emph{(L-twisted) Abel map} of $X$.

From the definition \eqref{E:Abel-map}, it follows that a priori the Abel map $A_L$ takes values in the big compactified Jacobian $\bJbar_X^{\chi(L)-1}$. Under the assumption that $X$ is Gorenstein, we can prove that the Abel map $A_L$ takes always values in a fine compactified Jacobian contained in $\bJbar_X^{\chi(L)-1}$  .

\begin{lemma}\label{L:Im-Abel}
Assume that $X$ is Gorenstein. Then for every $L\in \Pic(X)$ there exists a general polarization $\un q$ on $X$ of total degree $|\un q|=\chi(L)-1$ such that
$\Im(A_L)\subseteq \ov{J}_X(\un q)$.
\end{lemma}
\begin{proof}
Consider the polarization $\un q'$ on $X$ defined by setting, for every irreducible component $C_i$ of $X$,
$$\un q'_{C_i}=\deg_{C_i}L-\frac{\deg_{C_i}(\omega_X)}{2} -\frac{1}{\gamma(X)},$$
where $\gamma(X)$ denotes, as usual, the number of irreducible components of $X$.
Note that for any subcurve $Y\subseteq X$ we have that $\un q'_Y=\deg_Y(L)-\frac{\deg_Y(\omega_X)}{2}-\frac{\gamma(Y)}{\gamma(X)}$ and in particular
$|\un q'|=\deg L-\frac{\deg \omega_X}{2} -1=\chi(L)-1$.

We claim that every sheaf in the image of $A_L$ is $\un q'$-stable. Indeed, for any proper subcurve
$\emptyset \neq Y\subsetneq X$ and any point $p\in X$, we have that
\begin{equation}\label{E:stab-Abel}
\deg_Y(\m_p\otimes L)-\un q'_Y-\frac{\deg_Y(\omega_X)}{2}=\deg_Y(L)+\deg_Y(\m_p)-\deg_Y(L)+\frac{\gamma(Y)}{\gamma(X)}
\geq-1+\frac{\gamma(Y)}{\gamma(X)}>-1\geq -\frac{\delta_Y}{2},
\end{equation}
where we have used that $\gamma(Y)>0$ since $Y$ is not the empty subcurve and that $\delta_Y\geq 2$ since $X$ does not
have separating points by assumption. Therefore, $A_L(p)$ is $\un q'$-stable for every $p\in X$ by Remark \ref{R:stabGor}. Using Remark \ref{R:stable-general},
we can  deform slightly $\un q'$ in order to obtain a general polarization $\un q$ with $|\un q|=|\un q'|$ and for which $A_L(p)$ is $\un q$-stable for every $p\in X$, which implies that
 $\Im A_L\subseteq \ov{J}_X(\un q)$, q.e.d.
\end{proof}






Those fine compactified Jacobians  for which there exists an Abel map  as in the above Lemma \ref{L:Im-Abel}
are quite special, hence they deserve a special name.

\begin{defi}\label{D:exi-Abel}
Let $X$ be a curve without separating points.
A fine compactified Jacobian $\ov{J}_X(\un q)$ is said to \emph{admit an Abel map} if there exists a line bundle
$L\in \Pic(X)$ (necessarily of degree $|\un q|+p_a(X)$) such that $\Im A_L\subseteq \ov{J}_X(\un q)$.
\end{defi}



Observe that clearly the property of admitting an Abel map is invariant under equivalence by translation
(in the sense of Definition \ref{D:transla}).

\begin{remark}\label{R:exi-Abelmap}
It is possible to prove that a curve $X$ without separating points and having at most two irreducible components is such that any fine compactified
Jacobian of $X$  admits an Abel map.

However, in Section \ref{S:genus1} we are going to show examples
of curves with more than two components and having a fine compactified Jacobian which does not admit an Abel map
(see Proposition \ref{P:JacIn} and Proposition \ref{P:JacIV}). In particular, Proposition \ref{P:JacIn} shows that, as the number of irreducible components of $X$ increases,
fine compactified Jacobians of $X$ that admit an Abel map became more and more rare.
\end{remark}

\subsection{Abel maps with separating points} \label{S:Abel-sep}

The aim of this subsection is to define Abel maps for (reduced) curves $X$ having separating points (in the sense of \eqref{N:sep-node}) but satisfying
the following
\begin{equation}\label{E:dagger}
\un{\text{Condition } (\dagger)}: \text{Every separating point is a node.}
\end{equation}
Indeed, there are plenty of curves that satisfy condition $(\dagger)$ due to the following result of Catanese
(see \cite[Prop. 1.10]{Cat}).

\begin{fact}[Catanese]\label{F:Gor-dagger}
A (reduced) Gorenstein curve $X$ satisfies condition $(\dagger)$.
\end{fact}

Let us give an example of a curve that does not satisfy condition $(\dagger)$.

\begin{example}\label{Ex:non-dagger}
Consider a curve $X$ made of three irreducible
smooth components $Y_1$, $Y_2$ and $Y_3$ glued at one point $p\in Y_1\cap Y_2\cap Y_3$ with linearly independent tangent directions, i.e. in a such a way that, locally at $p$, the three components $Y_1$, $Y_2$ and $Y_3$ look like the three
coordinate axes in ${\mathbb A}^3$.
A straightforward local computation gives that $\delta_{Y_1}=|Y_1\cap (Y_2\cup Y_3)|=1$, so that $p$ is a separating point of $X$ (in the sense of \ref{N:sep-node}). However $p$ is clearly
not a node of $X$. Combined with Fact \ref{F:Gor-dagger}, this shows that $X$ is not Gorenstein at
$p\in X$.
\end{example}


Throughout this section, we fix a connected (reduced) curve $X$ satisfying condition $(\dagger)$ and let $\{n_1,\ldots$ $ ,n_{r-1}\}$ be its separating points. Since $X$ satisfies condition $(\dagger)$, each $n_i$ is a node. Denote by $\wt{X}$ the partial normalization of $X$ at the set $\{n_1,\ldots,n_{r-1}\}$.
Since each $n_i$ is a node, the curve $\wt{X}$ is a disjoint union of
$r$ connected reduced curves $\{Y_1,\ldots,Y_r\}$ such that each $Y_i$ does not have separating points.
Note also that $X$ has locally planar singularities if and only if each $Y_i$ has locally planar singularities.
We have a natural morphism
\begin{equation}\label{E:part-norm}
\tau:\wt{X}=\coprod_i Y_i\to X.
\end{equation}
We can naturally identify each $Y_i$ with a subcurve of $X$ in such a way that their union  is $X$ and that they do not have common
 irreducible components. In particular, the irreducible components of $X$ are the union of the irreducible
components of the curves $Y_i$. We call the subcurves $Y_i$ (or their image in $X$) the \emph{separating blocks }Êof $X$.

Let us first show how the study of  fine compactified Jacobians of $X$
can be reduced to the study of  fine compactified Jacobians of $Y_i$.
Observe that, given a polarization $\un q^i$ on each curve $Y_i$, we get a polarization $\un q$ on $X$ such that for every irreducible component $C$ of $X$ we have
\begin{equation}
\un q_C=
\begin{cases} 
\un q^i_C & \text{ if }  C\subseteq Y_i   \text{  and } C\cap Y_j=\emptyset  \text{ for all } j\neq i, \\
\un q^i_C-\frac{1}{2} & \text{ if }  C\subseteq Y_i   \text{  and } C\cap Y_j\neq \emptyset  \text{ for some } j\neq i.\\
\end{cases}
\end{equation}
Note that $|\un q|=\sum_i |\un q^i|+1-r$ so that indeed $\un q$ is a polarization on $X$.
We say that $\un q$ is the polarization \emph{induced} by the polarizations $(\un q^1, \ldots, \un q^r)$ and we write  $\un q:=(\un q^1,\ldots,\un q^r)$.


\begin{prop}\label{P:norm-sheaves}
Let $X$ be a connected curve satisfying condition $(\dagger)$.
\begin{enumerate}[(i)]
\item \label{P:norm-sheaves1} The pull-back map
$$\begin{aligned}
\tau^*: \bJbar_X & \longrightarrow \prod_{i=1}^r \bJbar_{Y_i}\\
I & \mapsto (I_{|Y_1},\ldots, I_{|Y_r}),
\end{aligned}
$$
is an isomorphism. Moreover $\tau^*(\bJ_X)=\prod_i \bJ_{Y_i}$.
\item \label{P:norm-sheaves2} Given a polarization $\un q^i$ on each curve $Y_i$, consider the induced
polarization $\un q:=(\un q^1,\ldots, \un q^r)$ on $X$ as above. Then $\un q$ is general if and only if each
$\un q^i$ is general and in this case the morphism $\tau^*$ induces an isomorphism
\begin{equation}\label{E:iso-Jac}
\tau^*: \ov{J}_X(\un q)\xrightarrow{\cong} \prod_i \ov{J}_{Y_i}(\un q^i).
\end{equation}
\item \label{P:norm-sheaves3}
If $\un q$ is a general polarization on $X$ then there exists a general polarization $\un q'$ with
$|\un q'|=|\un q|$ on $X$
which is induced by some polarizations $\un q^i$ on $Y_i$ and such that
$$\ov{J}_X(\un q)=\ov{J}_X(\un q').$$
\end{enumerate}
\end{prop}
\begin{proof}
It is enough, by re-iterating the argument, to consider the case where there is only one separating point $n_1=n$, i.e. $r=2$.
Therefore the normalization
$\wt{X}$ of $X$ at $n$ is the disjoint union of two connected curves $Y_1$ and $Y_2$, which we also identify with two subcurves of $X$ meeting at the node $n$.
Denote by $C_1$ (resp. $C_2$) the irreducible component of $Y_1$ (resp. $Y_2$) that contains the separating point $n$.
A warning about the notation: given a subcurve $Z\subset X$, we will denote
by $Z^c$  the complementary subcurve of $Z$ inside $X$, i.e. $\ov{X\setminus Z}$.
In the case where $Z\subset Y_i\subset X$ for some $i=1,2$ we will write $\ov{Y_i\setminus Z}$ for the complementary
subcurve of $Z$ inside $Y_i$.

Part \eqref{P:norm-sheaves1} is well-known, see \cite[Example 37]{est1} and \cite[Prop. 3.2]{est2}. The crucial fact is that if $I$ is simple
then $I$ must be locally free at the separating point $n$; hence $\tau^*(I)$ is still torsion-free, rank-1
and its restrictions $\tau^*(I)_{|Y_i}=I_{|Y_i}$ are torsion-free, rank-1 and simple.
Moreover, since $n$ is a separating point, the sheaf $I$ is completely determined by its pull-back $\tau^*(I)$, i.e.
there are no gluing conditions.
Finally, $I$ is a line bundle if and only if its pull-back $\tau^*(I)$ is a line bundle.

Part \eqref{P:norm-sheaves2}.
Assume first that each $\un q^i$ is a  general polarization on $Y_i$ for $i=1, 2$. Consider a proper subcurve $Z\subset X$ such that $Z$ and $Z^c$ are connected.
There are three possibilities:
\begin{equation}\label{E:case1}
\begin{sis}
& \un{ \text{Case I:}} \: C_1, C_2\subset Z^c \Longrightarrow Z\subset Y_i \text{ and } \ov{Y_i\setminus Z} \text{ is connected (for some } i=1,2) , \\
& \un{ \text{Case II:}} \: C_1, C_2\subset Z\Longrightarrow Z^c\subset Y_i \text{ and } \ov{Y_i\setminus Z^c} \text{ is connected (for some } i=1,2) , \\
& \un{ \text{Case III:}} \: C_i \subset Z \text{ and } C_{3-i}\subset Z^c \Longrightarrow Z=Y_i \text{ and }
Z^c=Y_{3-i} \text{ (for some } i=1,2). \\
\end{sis}
\end{equation}
Therefore, from the definition of $\un q=(\un q^1, \un q^2)$, it follows that
\begin{equation}\label{E:qtot}
\un q_Z=
\begin{sis}
& \un q^i_Z& \text{ in case I,}\\
& |\un q|-\un q_{Z^c} = |\un q|-\un q^i_{Z^c} & \text{ in case II,}\\
& |\un q^i|-\frac{1}{2} & \text{ in case III.}
\end{sis}
\end{equation}
In each of the cases I, II, III we conclude that $\un q_Z\not\in \Z$ using that $\un q^i$ is general and that $|\un q|, |\un q^i|\in \Z$. Therefore $\un q$ is general by Remark \ref{R:conn-pola}.

Conversely, assume that $\un q$ is general and let us show that $\un q^i$ is general for $i=1,2$.
Consider a proper subcurve $Z\subset Y_i$ such that $Z$ and $\ov{Y_i\setminus Z}$ is connected. There are two possibilities:
\begin{equation}\label{E:case2}
\begin{sis}
& \un{ \text{Case A:}} \: C_i\not\subset Z \Longrightarrow Z^c \text{ is connected,}  \\
& \un{ \text{Case B:}} \: C_i\subset Z \Longrightarrow (\ov{Y_i\setminus Z})^c \text{ is connected.} \\
\end{sis}
\end{equation}
Using the definition of $\un q=(\un q^1,\un q^2)$, we compute
\begin{equation}\label{E:qpart}
\un q^i_Z=
\begin{sis}
& \un q_Z & \text{ in case A,}\\
& |\un q^i|-\un q^i_{\ov{Y^i\setminus Z}}= |\un q^i|-\un q_{\ov{Y^i\setminus Z}}& \text{ in case B.}\\
\end{sis}
\end{equation}
In each of the cases A, B we conclude that $\un q^i_Z\not\in \Z$
using that $\un q$ is general and $|\un q^i|\in \Z$. Therefore $\un q^i$ is general by Remark \ref{R:conn-pola}.

Finally, in order to prove \eqref{E:iso-Jac}, it is enough, using part \eqref{P:norm-sheaves1}, to show that
a simple torsion-free rank-1 sheaf $I$ on $X$  is $\un q$-semistable if and only if $I_{|Y_i}$  is $\un q^i$-semistable for $i=1,2$.
Observe first that, since $I$ is locally free at the node $n$ (see the proof of part \eqref{P:norm-sheaves1}), we have
that for any subcurve $Z\subset X$ it holds
\begin{equation}\label{E:split1}
\chi(I_Z)=
\begin{cases} 
\chi(I_{Z\cap Y_i})=\chi(I_Z) & \text{ if } Z\subseteq Y_i  \text{  for some } i,\\
\chi(I_{Z\cap Y_1})+\chi(I_{Z\cap Y_2})-1 & \text{ otherwise. }
\end{cases}
\end{equation}

Assume first that $I_{|Y_i}$  is $\un q^i$-semistable for $i=1,2$.
Using \eqref{E:split1}, we get
$$\chi(I)=\chi(I_{Y_1})+\chi(I_{Y_2})-1=|\un q^1|+|\un q^2|-1=|\un q|.$$
Consider a proper subcurve $Z\subset X$ such that $Z$ and $Z^c$ are connected. Using  \eqref{E:case1}, \eqref{E:qtot} and \eqref{E:split1}, we compute
\begin{equation*}
\chi(I_Z)-\un q_Z=
\begin{sis}
& \chi(I_Z) -\un q^i_Z=  \chi(I_{Z\cap Y_i}) -\un q^i_{Z\cap Y_i} \hspace{5,13cm} \text{ in case I,}\\
&  \chi(I_{Z\cap Y_{1}}) +\chi(I_{Z\cap Y_2})-1  -|\un q|+ \un q^i_{Z^c}=
\\
&= \chi(I_{\ov{Y_i\setminus Z^c}}) +\chi(I_{Y_{3-i}})-|\un q^1|-|\un q^2| + \un q^i_{Z^c}= \chi(I_{\ov{Y_i\setminus Z^c}}) -\un q^i_{\ov{Y_i\setminus Z^c}}
\hspace{0,5cm} \text{ in case II,}\\
& \chi(I_{Y_i})-|\un q^i|+\frac{1}{2}=\frac{1}{2}  \hspace{6.7cm} \text{ in case III.}
\end{sis}
\end{equation*}
In each of the cases I, II, III we conclude that $\chi(I_Z)-\un q_Z\geq 0$ using that $I_{|Y_i}$ is
$\un q^i$-semistable. Therefore $I$ is $\un q$-semistable by Remark \ref{R:conn-subcurves}.

Conversely, assume that $I$  is $\un q$-semistable. Using \eqref{E:split1}, we get that
\begin{equation}\label{E:tot-sum}
|\un q^1|+|\un q^2|=|\un q|+1=\chi(I)+1=\chi(I_{Y_1})+\chi(I_{Y_2}).
\end{equation}
Since $I$ is $\un q$-semistable, inequalities \eqref{multdeg-sh1} applied to $Y_i$ for $i=1,2$ give that
\begin{equation}\label{E:two-ineq}
\chi(I_{Y_i})\geq \un q_{Y_i}=|\un q^i|-\frac{1}{2}.
\end{equation}
Since $\chi(I_{Y_i})$ and  $|\un q^i|$ are integral numbers, from \eqref{E:two-ineq} we get that
$\chi(I_{Y_i})\geq |\un q^i|$ which combined with \eqref{E:tot-sum} gives that $\chi(I_{Y_i})=|\un q^i|$.
Consider now a subcurve $Z\subset Y_i$ (for some $i=1,2$) such that $Z$ and $\ov{Y_i\setminus Z}$ are connected. Since $I$ is locally free at $n$, we have that $(I_{|Y_i})_{Z}=I_{Z}$.
Using \eqref{E:case2} and \eqref{E:qpart},  we compute
\begin{equation*}
 \chi((I_{|Y_i})_Z)-\un q^i_Z=
\begin{sis}
& \chi(I_Z)-\un q_Z \hspace{5.6cm} \text{ in case A,}\\
&  \chi(I_{\ov{Y_i\setminus Z}^c}) -\chi(I_{Y_{3-i}})+1-|\un q^i|+\un q_{\ov{Y^i\setminus Z}}=  \\
&
 =\chi(I_{\ov{Y_i\setminus Z}^c}) -\un q_{\ov{Y^i\setminus Z}^c}
 \hspace{4cm} \text{ in case B.}\\
\end{sis}
\end{equation*}
In each of the cases A, B we conclude that $ \chi((I_{|Y_i})_Z)-\un q^i_Z\geq 0$ using that $I$ is $\un q$-semistable. Therefore $I_{|Y_i}$ is $\un q^i$-semistable by Remark \ref{R:conn-subcurves}.

Part \eqref{P:norm-sheaves3}: note that a polarization $\un q'$ on $X$ is induced by some polarizations $\un q^i$ on
$Y_i$ if and only if $\un q'_{Y_i}+\frac{1}{2}\in \Z$ for $i=1,2$. For a general polarization $\un q$ on $Y$, we have that
\begin{equation*}\label{E:hypo-pola}
\begin{sis}
& |\un q|=\un q_{Y_1}+\un q_{Y_2} \in \Z,\\
& \un q_{Y_i}\not\in \Z.\\
\end{sis}
\end{equation*}
Therefore, we can find unique integral numbers $m_1,m_2\in \Z$ and a unique rational number
$r\in \Q$ with $-\frac{1}{2}<r <\frac{1}{2}$ such that
\begin{equation}\label{E:approxi}
\begin{sis}
& \un q_{Y_1}=m_1+\frac{1}{2}+r,\\
& \un q_{Y_2}=m_2-\frac{1}{2}-r.\\
\end{sis}
\end{equation}
Define now the polarization $\un q'$ on $X$ in such a way that for an irreducible component $C$ of $X$,
we have that
\begin{equation*}
\un q'_C:=
\begin{sis}
& \un q_C & \text{ if } C \neq C_1, C_2,\\
& \un q_{C_1}-r & \text{ if } C=C_1,\\
& \un q_{C_2}+r & \text{ if } C=C_2.\\
\end{sis}
\end{equation*}
In particular for any subcurve $Z\subset X$, the polarization $\un q'$ is such that
\begin{equation}\label{E:new-pola}
\un q'_Z:=
\begin{sis}
& \un q_Z & \text{ if either } C_1,C_2\subset Z \text{ or } C_1, C_2\subset Z^c,\\
& \un q_{Z}-r & \text{ if } C_1\subset Z \text{ and } C_2\subset Z^c,\\
& \un q_{Z}+r & \text{ if } C_2\subset Z \text{ and } C_1\subset Z^c.\\
\end{sis}
\end{equation}
Specializing to the case $Z=Y_1, Y_2$ and using \eqref{E:approxi}, we get that
\begin{equation}\label{E:value-extreme}
\begin{sis}
& \un q'_{Y_1}=\un q_{Y_1}-r=m_1+\frac{1}{2}, \\
& \un q'_{Y_2}=\un q_{Y_2}+r=m_2-\frac{1}{2}, \\
& |\un q'|=\un q'_{Y_1}+\un q'_{Y_2}=m_1+m_2=\un q_{Y_1}+\un q_{Y_2}=|\un q|.
\end{sis}
\end{equation}
As observed before, this implies that $\un q'$ is induced by two (uniquely determined) polarizations $\un q^1$ and
$\un q^2$ on $Y_1$ and $Y_2$, respectively, and moreover that $|\un q'|=|\un q|$.

Let us check that $\un q'$ is general. Consider a proper subcurve $Z\subset X$ such that $Z$ and $Z^c$ are connected.
Using \eqref{E:case1}, \eqref{E:new-pola} and \eqref{E:value-extreme}, we compute that
\begin{equation}\label{E:q'}
\un q'_Z=
\begin{sis}
& \un q_Z & \text{ in case I and II,}\\
& \un q'_{Y_i}= \un q_{Y_i} +(-1)^i r= m_i+(-1)^{i+1} \frac{1}{2} & \text{ in case III.} \\
\end{sis}
\end{equation}
In each of the above cases I, II, III we get that $\un q'_Z\not\in \Z$ using that $\un q$ is general and that $m_i\in \Z$.
Therefore $\un q'$ is general by Remark \ref{R:conn-pola}.

Finally, in order to check that $\ov{J}_X(\un q)=\ov{J}_X(\un q')$ we must show that a simple rank-1 torsion-free
sheaf $I$ on $X$ with $\chi(I)=|\un q|=|\un q'|$ is $\un q$-semistable if and only if it is $\un q'$-semistable.
Using Remark \ref{R:conn-subcurves}, it is sufficient (and necessary) to check that for, any proper subcurve
$Z\subset X$ such that $Z$ and $Z^c$ are connected, $I$ satisfies \eqref{multdeg-sh1} with respect to $\un q_Z$ if and
only if it satisfies \eqref{multdeg-sh1} with respect to $\un q'_Z$. If $Z$ belongs to case I or II (according to the
classification \eqref{E:case1}), this is clear by \eqref{E:q'}. If $Z$ belongs to case III, i.e. if $Z=Y_i$ for some $i=1,2$, then, using \eqref{E:approxi} together with the fact that $-\frac{1}{2}<r<\frac{1}{2}$ and
$m_i, \chi(I_{Y_i})\in \Z$, we get that
$$\chi(I_{Y_i})\geq \un q_{Y_i}= m_i+(-1)^{i+1}\left( \frac{1}{2}+r \right)
\Longleftrightarrow
\begin{cases} 
\chi(I_{Y_i})\geq m_i+1 & \text{ if } i=1,\\
\chi(I_{Y_i})\geq m_i & \text{ if } i=2.\\
\end{cases} $$
Similarly using \eqref{E:value-extreme}, we get that
$$ \chi(I_{Y_i})\geq \un q'_{Y_i}= m_i+(-1)^{i+1} \frac{1}{2}
\Longleftrightarrow
\begin{cases} 
\chi(I_{Y_i})\geq m_i+1 & \text{ if } i=1,\\
\chi(I_{Y_i})\geq m_i & \text{ if } i=2.\\
\end{cases} 
$$
This shows that $I$ satisfies \eqref{multdeg-sh1} with respect to $\un q_{Y_1}$ if and only if it satisfies
\eqref{multdeg-sh1} with respect to $\un q'_{Y_1}$, which concludes our proof.

\end{proof}

We can now define the Abel maps for $X$.

\begin{prop}\label{P:Abel-sep}
Let $X$ be a connected curve satisfying condition $(\dagger)$ as above.
\begin{enumerate}[(i)]
\item \label{P:Abel-sep1} For any line bundle $L\in \Pic(X)$, there exists a unique morphism $A_L:X\to \bJbar_X$ such that for any $i=1,\ldots,r$ it holds:
\begin{enumerate}[(a)]
\item \label{P:Abel-sep1a}  the following diagram is commutative
\begin{equation*}
\xymatrix{
Y_i\ar[r]^{A_{L_i}}\ar@{_{(}->}[d] & \bJbar_{Y_i} & \\
X  \ar[r]^{A_L} & \bJbar_X \ar[r]^{\tau^*}_{\cong} & \prod_j \bJbar_{Y_j}\ar@{->>}[ul]_{\pi_i}
}
\end{equation*}
where $\pi_i$ denotes the projection onto the $i$-th factor, $\tau^*$ is the isomorphism of Proposition \ref{P:norm-sheaves}\eqref{P:norm-sheaves1} and $A_{L_i}$ is the $L_i$-twisted map of \eqref{E:part-norm}
for $L_i:=L_{|Y_i}$.
\item \label{P:Abel-sep1b} The composition
$$Y_i \hookrightarrow X \xrightarrow{A_L}\bJbar_X \xrightarrow[\cong]{\tau^*}  \prod_j \bJbar_{Y_j}
\xrightarrow{\prod_{j\neq i}\pi_j} \prod_{j\neq i} \bJbar_{Y_j}$$
is a constant map.
\end{enumerate}
Explicitly, the morphism $A_L$ is given for $p\in Y_i$ (with $1\leq i\leq r$) by 
\begin{equation*}
\tau^*(A_L(p))=(L_1(-n_1^i),\ldots, L_{i-1}(-n_{i-1}^i),\m_p \otimes L_{i},L_{i+1}(-n_{i+1}^i),\ldots,L_r(-n_r^i))
\end{equation*}
where for any $h\neq k$ we denote by $n_k^h$ the unique separating node of $X$ that belongs to $Y_k$ and such that $Y_k$ and $Y_h$ belong to the distinct connected components of the partial normalization of $X$ at $n_k^h$
(note that such a point $n_k^h$ exists and it is a smooth point of $Y_k$). 

\item  \label{P:Abel-sep2} Let $\un q^i$ be a general polarization on $Y_i$ for any $1\leq i\leq r$ and denote by $\un q$ the induced (general) polarization on $X$. Then
    $$A_L(X)\subset \J_X(\un q)\Leftrightarrow A_{L_i}(Y_i)\subset \J_{Y_i}(\un q^i) \text{ for any }
    1 \leq i\leq r.$$
\end{enumerate}
\end{prop}
\begin{proof}
Part \eqref{P:Abel-sep1}: assume that such a map $A_L$ exists and let us prove its uniqueness.
From \eqref{P:Abel-sep1a} and \eqref{P:Abel-sep1b} it follows that the composition
$$\wt{A_L}: \wt{X}=\coprod_i Y_i \xrightarrow{\tau} X \xrightarrow{A_L} \bJbar_X\xrightarrow{\tau^*} \prod_i \bJbar_{Y_i}$$
is such that for every $1\leq i\leq r$ and every $p\in Y_i$ it holds
\begin{equation}\label{E:def-AL}
(\wt{A_L})_{|Y_i}(p)=(M_1^i,\ldots,M_{i-1}^i,A_{L_i}(p),M_{i+1}^i,\ldots,M_r^i)
\end{equation}
for some elements $M_j^i\in \bJbar_{Y_j}$ for $j\neq i$. Moreover, if we set $\tau^{-1}(n_k)=\{n_k^1,n_k^2\}$  then we must have that
\begin{equation}\label{E:gluing}
\wt{A_L}(n_k^1)=\wt{A_L}(n_k^2) \text{ for any }Ê1\leq k\leq r-1.
\end{equation}

\un{Claim:} The unique elements   $M_j^i\in \bJbar_{Y_j}$ (for any $i\neq j$) such that the map $\wt{A_L}$ in \eqref{E:def-AL} satisfies the
conditions in \eqref{E:gluing} are given by $M_j^i=L_j(-n_j^i)$, where $n_j^i$ are as above. 

The claim clearly implies the uniqueness of the map $\wt{A_L}$, hence the uniqueness of the map $A_L$. Moreover, the same claim also shows the existence of the map $A_L$
with the desired properties: it is enough to define $\wt{A_L}$ via the formula \eqref{E:def-AL} and notice that, since the conditions \eqref{E:gluing} are satisfied, then the map
$\wt{A_L}$ descends to a map $A_L:X\to \bJbar_X$.

It remains therefore to prove the Claim. Choose a separating node $n_k$ of $X$ with inverse image $\tau^{-1}(n_k)=\{n_k^1,n_k^2\}$ and suppose that $n_k^1\in Y_i$ and $n_k^2\in Y_j$.
Clearly,  we have that $n_k^1=n_i^j$ and  $n_k^2=n_j^i$ by construction and it is easily checked that
\begin{equation}
n_k^i=n_k^j \: \text{ for any } \: k\neq i, j.
\tag{*}\end{equation}
From condition \eqref{E:gluing} applied to $n_k$, we deduce that
\begin{equation}
\begin{sis}
& M_i^j=\m_{n_k^1}\otimes L_i=L_i(-n_i ^j),\\
& M_j^i=\m_{n_k^2}\otimes L_j=L_j(-n_j^i), \\
& M_k^i=M_ k^j \text{ for any } k\neq i,j.
\end{sis}\tag{**}
\end{equation}
By combining (*) and (**), it is easily checked that the unique elements $M_j^i$ that satisfy  condition \eqref{E:gluing} for every separating node are given by $M_j^i=L_j(-n_j^i)$, q.e.d.



Part \eqref{P:Abel-sep2} follows easily from the diagram in \eqref{P:Abel-sep1a} and the isomorphism \eqref{E:iso-Jac}.
\end{proof}

We call the map $A_L$ of Proposition \ref{P:Abel-sep}\eqref{P:Abel-sep1} the \emph{($L$-twisted) Abel map} of $X$.
We can extend Definition \ref{D:exi-Abel} to the case of curves satisfying condition $(\dagger)$ from \eqref{F:Gor-dagger}.

\begin{defi}\label{D:ex-Abel-sep}
Let $X$ be a curve satisfying condition $(\dagger)$. We say that a fine compactified Jacobian $\ov{J}_X(\un q)$
of $X$ \emph{admits an Abel map} if there exists $L\in \Pic(X)$ (necessarily of degree $|\un q|+p_a(X)$) such that $\Im A_L\subseteq \ov{J}_X(\un q)$.
\end{defi}

By combining Propositions \ref{P:norm-sheaves} and \ref{P:Abel-sep}, we can easily reduce the problem of the existence of an
Abel map for a fine compactified Jacobian of $X$ to the analogous question on the separating blocks of $X$.

\begin{cor}\label{C:exist-Abel}
Let $X$ be a curve satisfying condition $(\dagger)$ with separating blocks $Y_1, \ldots, Y_r$.
\begin{enumerate}[(i)]
\item \label{C:exist-Abel1} Let $\un q$ be a general polarization of $X$ and assume (without loss of generality by
    Proposition \ref{P:norm-sheaves}\eqref{P:norm-sheaves3}) that $\un q$ is induced by some general polarizations $\un q^i$ on
    $Y_i$. Then $\ov{J}_X(\un q)$ admits an Abel map if and only if each $\ov{J}_{Y_i}(\un q^i)$ admits an Abel map.
\item \label{C:exist-Abel2} If $X$ is Gorenstein, then for any $L\in \Pic(X)$ there exists a general polarization $\un q$ on $X$ of total degree
$|\un q|=\chi(L)-1$ such that $\Im A_L\subseteq \ov{J}_X(\un q)$.
\end{enumerate}
\end{cor}
\begin{proof}
Part \eqref{C:exist-Abel1} follows from Proposition \ref{P:Abel-sep}\eqref{P:Abel-sep2}. Part \eqref{C:exist-Abel2}
follows from Proposition \ref{P:Abel-sep}\eqref{P:Abel-sep2} together with Lemma \ref{L:Im-Abel}.
\end{proof}

When is the Abel map $A_L$ an embedding? The answer is provided by the following result, whose proof is identical to the
 proof of \cite[Thm. 6.3]{CCE}.

\begin{fact}[Caporaso-Coelho-Esteves]\label{F:emb-Abel}
Let $X$ be a curve satisfying condition $(\dagger)$ and $L\in \Pic(X)$. The Abel map $A_L$  is an embedding away from the rational separating blocks (which are isomorphic to $\PP^1$)
 while it contracts each rational separating block $Y_i\cong \PP^1$  into a seminormal point of $A_L(X)$, i.e. an ordinary singularity with linearly independent tangent directions.
\end{fact}

\section{Examples: Locally planar curves of arithmetic genus $1$}\label{S:genus1}

In this section, we are going to study fine compactified Jacobians and Abel maps for singular curves of arithmetic genus $1$ with locally planar singularities. According to Fact \ref{F:Gor-dagger}, such a curve $X$ satisfies the condition $(\dagger)$ and therefore, using Proposition \ref{P:norm-sheaves}
and Proposition \ref{P:Abel-sep}, we can reduce the study of fine compactified Jacobians and Abel maps to the
case where  $X$ does not have separating points (or equivalently separating nodes).
Under this additional assumption, a classification is possible.

\begin{fact}\label{F:clas-g1}
Let $X$ be a (reduced) connected singular curve without separating points, with locally planar singularities and $p_a(X)=1$. Then $X$ is one of the curves depicted
in Figure \ref{Fig:Kod}, which are called Kodaira curves.
\end{fact}
\begin{proof}
Since $X$ has non separating points and $p_a(X)=1$ then $X$ has trivial canonical sheaf
by \cite[Example 39]{est1}. These curves were classified by Catanese in \cite[Prop. 1.18]{Cat}.
An inspection of the classification in loc. cit. reveals that the only such singular curves that have
locally planar singularities are the ones depicted in Figure \ref{Fig:Kod}, i.e. the Kodaira curves.
\end{proof}

Note that the curves in Figure \ref{Fig:Kod} are exactly
the reduced fibers appearing in the well-known Kodaira classification of singular fibers of minimal
elliptic fibrations (see \cite[Chap. V, Sec. 7]{BPV}). This explains why they are called  Kodaira curves.

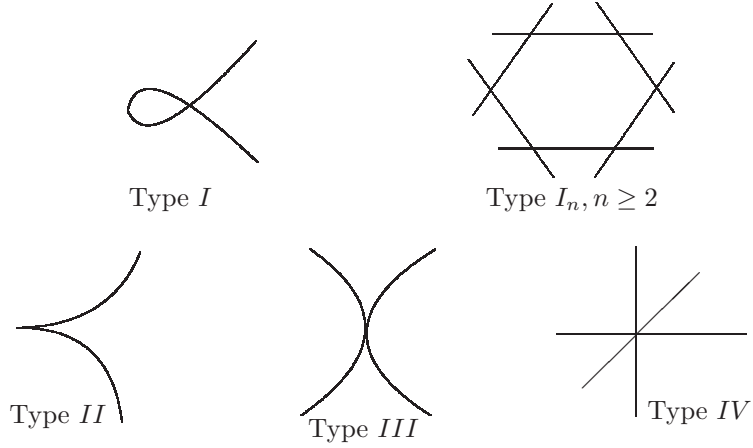
\begin{figure}[h]
\begin{center}
\unitlength .8mm 
\linethickness{0.4pt}
\ifx\plotpoint\undefined\newsavebox{\plotpoint}\fi 
\begin{picture}(125.75,63.75)(0,126)
\qbezier(45.25,168.25)(25.875,187)(24,176.75)
\qbezier(24,176.75)(27.5,168.5)(45,188.25)
\put(31,162){\makebox(0,0)[cc]{Type $I$}}
\qbezier(26,153.25)(21.125,141)(5.75,140.75)
\qbezier(5.75,140.75)(21.125,141)(23,125.25)
\put(12.5,126){\makebox(0,0)[cc]{Type $II$}}
\put(84,189.5){\line(1,0){26.25}}
\qbezier(54,153.75)(72.75,140.5)(52.5,126.25)
\qbezier(74.5,153.75)(52.375,140)(73.75,126.25)
\put(62.75,123.75){\makebox(0,0)[cc]{Type $III$}}
\put(107.5,154.25){\line(0,-1){28.25}}
\put(94.5,139.75){\line(1,0){31.25}}
\put(98.75,130.75){\line(1,1){19.25}}
\put(118,126.75){\makebox(0,0)[cc]{Type $IV$}}
\put(85,170.5){\line(1,0){25.5}}
\multiput(80.75,176)(.0336787565,.0479274611){386}{\line(0,1){.0479274611}}
\multiput(113.75,176.5)(-.0337301587,.0482804233){378}{\line(0,1){.0482804233}}
\multiput(80,185.25)(.0337349398,-.0469879518){415}{\line(0,-1){.0469879518}}
\multiput(113.75,184.75)(-.0336787565,-.0492227979){386}{\line(0,-1){.0492227979}}
\put(97,162){\makebox(0,0)[cc]{Type $I_n, n\geq 2$}}
\end{picture}

\end{center}
\caption{Kodaira curves.}
\label{Fig:Kod}
\end{figure}

Abel maps for Kodaira curves behave particularly well, due to the following result proved in
\cite[Example 39]{est1}.

\begin{fact}[Esteves]\label{F:Abel-g1}
Let $X$ be a connected curve without separating points and such that $p_a(X)=1$.
Then for any $L\in \Pic(X)$ the image $A_L(X)\subseteq \bJbar_X$ of $X$ via the $L$-twisted Abel map is equal
to a fine compactified Jacobian $\ov{J}_X(\un q)$ of $X$ and $A_L$ induces an isomorphism
$$A_L:X\stackrel{\cong}{\longrightarrow} A_L(X)=\ov{J}_X(\un q).$$
\end{fact}

From the above Fact \ref{F:Abel-g1}, we deduce that, up to equivalence by translation (in the sense of Definition
\ref{D:transla}), there is exactly one fine compactified Jacobian that admits an Abel map and this fine compactified Jacobian is isomorphic to the curve itself.
This last property is indeed true for any fine compactified Jacobian of a Kodaira curve, as shown in the following

\begin{prop}\label{P:Jac-Kod}
Let $X$ be a Kodaira curve. Then every fine compactified Jacobian of $X$ is isomorphic to $X$.
\end{prop}
\begin{proof}
Let $\ov{J}_X(\un q)$ be a fine compactified Jacobian of $X$. By Proposition \ref{P:1parsmooth}, we can find a $1$-parameter regular smoothing $f:\SS\to B=\Spec R$ of $X$
(in the sense of Definition \ref{D:1par-sm}), where $R$ is a complete discrete valuation ring with quotient field $K$. Note that the generic fiber $\SS_K$ of $f$ is an elliptic curve.
Following the notation of  \S\ref{S:1par-sm}, we can form the $f$-relative fine compactified Jacobian $\pi:\ov{J}_f(\un q)\to B$ with respect to the polarization $\un q$.
Recall that $\pi$ is a projective and flat morphism whose generic fiber is $\Pic^{|\un q|}(\SS_K)$ and whose special fiber is $\ov{J}_X(\un q)$.
Using Theorem \ref{T:univ-Jac}, it is easy to show that if we choose a generic $1$-parameter smoothing $f:\SS\to B$ of $X$, then the surface $\ov{J}_f(\un q)$ is
regular. Moreover, Fact \ref{F:Jac-Ner} implies that the smooth locus $J_f(\un q)\to B$ of $\pi$ is isomorphic to the N\'eron model of the generic fiber $\Pic^{|\un q|}(\SS_K)$.
Therefore, using the well-known relation between the N\'eron model and  the regular minimal model  of the elliptic curve $\Pic^{|\un q|}(\SS_K)\cong \SS_K$ over $K$
(see \cite[Chap. 1.5, Prop. 1]{BLR}), we deduce that $\pi: \ov{J}_f(\un q)\to B$ is the regular minimal model of $\Pic^{|\un q|}(\SS_K)$. In particular,
$\pi$ is a minimal elliptic fibration with reduced fibers and therefore, according to Kodaira's classification (see see \cite[Chap. V, Sec. 7]{BPV}),  the special fiber $\ov{J}_X(\un q)$
of $\pi$ must be a smooth elliptic curve or a Kodaira curve.

According to Corollary \ref{C:irre-comp},  the number of irreducible components of $\ov{J}_X(\un q)$ is equal to the complexity $c(X)$ of $X$.
However, it is very easy to see that for a Kodaira curve $X$ the complexity number $c(X)$ is equal to the number of irreducible components of $X$.
Therefore if $c(X)\geq 4$, i.e. if $X$ is of Type $I_n$ with $n\geq 4$, then necessarily $\ov{J}_X(\un q)$ is of type $I_n$, hence it is isomorphic to $X$.

In the case $n\leq 3$, the required isomorphism $\ov{J}_X(\un q)\cong X$ follows from the fact that the smooth locus of $\ov{J}_X(\un q)$ is isomorphic to a disjoint union of torsors
under $\Pic^{\un 0}(X)$ (see Corollary \ref{C:prop-fineJac}) and that
$$\Pic^{\un 0}(X)=\begin{cases}
\Gm & \:\: \text{Êif }ÊX \:\: \text{ is of Type }ÊI \:\: \text{ or } I_n \: (n\geq 2),\\
\Ga & \:\: \text{ if } X \:\: \text{ is of Type } II, III \:\: \text{ or } IV.\\
\end{cases}
$$

\end{proof}

Let us now classify the fine compactified Jacobians for a Kodaira curve $X$, up to equivalence by translation, and indicate which of them admits an Abel map.

\vspace{0,1cm}

\fbox{$X$ is of Type $I$ or Type $II$}

\vspace{0,1cm}

Since the curve $X$ is irreducible, we have that the fine compactified Jacobians of $X$ are of the
form $\bJbar_X^d$ for some $d\in \Z$. Hence they are all equivalent by translation and each of them admits an Abel map.

\vspace{0,2cm}

\fbox{$X$ is of Type $I_n$, with $n\geq 2$}

\vspace{0,1cm}

The fine compactified Jacobians of $X$ up to equivalence by translation and their behavior with respect to the Abel map
are described in the following proposition.

\begin{prop}\label{P:JacIn}
Let $X$ be a Kodaira curve of type $I_n$ (with $n\geq 2$) and let $\{C_1,\ldots, C_n\}$ be the irreducible components of $X$,
ordered in such a way that, for any $1\leq i \leq n$, $C_i$ intersects $C_{i-1}$ and $C_{i+1}$,
with the cyclic convention that $C_{n+1}:=C_1$.
\begin{enumerate}[(i)]
\item \label{P:JacIn1} Any fine compactified Jacobian is equivalent by translation to a unique fine compactified Jacobian of the form
$\ov{J}_X(\un q)$ for a polarization $\un q$ that satisfies
\begin{equation*}
\un q=\left(q_1,\ldots,q_{n-1},-\sum_{i=1}^{n-1} q_i\right) \text{ with } 0\leq q_i<1, \tag{*}
\end{equation*}
\begin{equation*}
\sum_{i=r}^{s} q_i \not\in \Z \text{ for any } 1\leq r \leq s \leq n-1, \tag{**}
\end{equation*}
\begin{equation*}
q_i=\frac{k_i}{n} \text{ with } k_i=1,\ldots, n-1,  \: \text{ for any } 1\leq i\leq n-1. \tag{***}
\end{equation*}
In particular, there are exactly $(n-1)!$ fine compactified Jacobians of $X$ up to equivalence by translation.
\item \label{P:JacIn2}
The unique fine compactified Jacobian, up to equivalence by translation, that admits an Abel map is
$$\ov{J}_X\left(\frac{n-1}{n}, \ldots, \frac{n-1}{n}, -\frac{(n-1)^2}{n}\right).$$
\end{enumerate}
\end{prop}
\begin{proof}
Part \eqref{P:JacIn1}:  given any polarization $\un q'$, there exists a unique polarization $\un q$ that satisfies conditions (*) and such that $\un q-\un q'\in \Z^n$.
Since  any connected subcurve $Y\subset X$ is such that $Y$ or $Y^c$ is equal to $C_r\cup\ldots\cup C_s$ (for some $1\leq r\leq s\leq n-1$),
we deduce that  a polarization $\un q$  that satisfies (*) is general if and only if it satisfies (**). Hence any fine compactified Jacobian
is equivalent by translation to a unique $\ov{J}_X(\un q)$,  for a polarization $\un q$ that satisfies (*) and (**).
Consider now the arrangement of hyperplanes in $\R^{n-1}$ given by
$$\left\{\sum_{i=r}^s q_i= n\right\}$$
for all $1\leq r\leq s\leq n-1$ and all $n\in \Z$. This arrangement of hyperplanes cuts the hypercube $[0,1]^{n-1}$ into finitely many chambers.
Notice that a polarization $\un q$ satisfies (*) and (**) if and only if it belongs to the interior of one of these chambers.
Arguing as in the proof of Proposition \ref{P:finite-eq}, it is easy to see that two polarizations $\un q$ and $\un q'$ satisfying (*) and (**) belong to the same
chamber if and only if $\ov{J}_X(\un q)=\ov{J}_X(\un q')$.
Now it is an entertaining exercise (that we leave to the reader) to check that any chamber contains exactly one polarization $\un q$ that satisfies  (***).
This proves the first claim of part \eqref{P:JacIn1}. The second claim in part \eqref{P:JacIn1}  is an easy counting argument that we again leave to the reader.

Part \eqref{P:JacIn2}: if we take a line bundle $L$ of multidegree $\un{\deg} L=(1,\ldots, 1, -(n-2))$ then from the proof of Lemma \ref{L:Im-Abel}
it follows that $\Im A_L \subseteq \ov{J}:=\ov{J}_X(\frac{n-1}{n}, \ldots, \frac{n-1}{n}, -\frac{(n-1)^2}{n})$. Therefore, from Fact \ref{F:Abel-g1},
it follows that $\ov{J}$
is the unique fine compactified Jacobian, up to equivalence by translation, that admits
an Abel map.


\end{proof}

\fbox{$X$ is of Type $III$}

\vspace{0,1cm}

Since $X$ has two irreducible components, then every fine compactified Jacobian of $X$ admits an Abel map
by Remark \ref{R:exi-Abelmap}. By Fact \ref{F:Abel-g1}, all fine compactified Jacobians of $X$ are therefore equivalent by translation.

\vspace{0,1cm}

\fbox{$X$ is of Type $IV$}

\vspace{0,1cm}

The fine compactified Jacobians of $X$ up to equivalence by translation and their behavior with respect to the Abel map
are described in the following proposition.

\begin{prop}\label{P:JacIV}
Let $X$ be a Kodaira curve of type $IV$.
\begin{enumerate}[(i)]
\item \label{P:JacIV1} Any fine compactified Jacobian of $X$ is equivalent by translation to either $\ov{J}_1:=\ov{J}_X\left(\frac{2}{3}, \frac{2}{3}, -\frac{4}{3}\right)$ or
$\ov{J}_2:=\ov{J}_X\left(\frac{1}{3}, \frac{1}{3}, -\frac{2}{3}\right)$.
\item \label{P:JacIV2} $\ov{J}_1$ admits an Abel map while
$\ov{J}_2$ does not admit an Abel map.
\end{enumerate}
\end{prop}

\begin{proof}
Part \eqref{P:JacIV1}  is proved exactly as in the case of the Kodaira curve of type $I_3$ (see Proposition \ref{P:JacIn}\eqref{P:JacIn1}).


Part \eqref{P:JacIV2}: if we take a line bundle $L$ of multidegree $\un{\deg} L=(1,1,-1)$ then $\Im A_L\subseteq \ov{J}_1$ as  follows from the proof of Lemma \ref{L:Im-Abel}.
Therefore $\ov{J}_1$ admits an Abel map.

Let us now show that $\ov{J}_2$ does not admit  an Abel map. Suppose by contradiction that there
exists a line bundle $L$ of multidegree $\deg L=(d_1,d_2,d_3)$ such that $A_L(p)=\m_p\otimes L\in \ov{J}_2$
where $p$ denotes the unique singular point of $X$. The stability of $\m_p\otimes L$ with respect to
the polarization $(q_1,q_2,q_3)=(\frac{1}{3}, \frac{1}{3}, -\frac{2}{3})$ gives for any irreducible component
$C_i$ of $X$:
$$d_i=d_i-1+1=\deg_{C_i}(\m_p\otimes L)+1=\chi((\m_p\otimes L)_{C_i})>q_i.$$
We deduce that $d_1\geq 1$, $d_2\geq 1$ and $d_3\geq 0$. However if $\Im A_L\subset \ov{J}_2$ then the total degree
of $L$ must be one, which contradicts the previous conditions.

\end{proof}

\begin{remark}
Realize a Kodaira curve $X$ of Type $IV$ as the plane cubic with equation $y(x+y)(x-y)=0$.
One can show that  the singular points in $\ov{J}_1$ and in $\ov{J}_2$ correspond to two sheaves that are not locally isomorphic:
the singular point of $\ov{J}_1$ is the sheaf $\I_1:=\m_p\otimes L$ where $\m_p$ is the ideal sheaf of the point $p$ defined by $(x,y)$ and
$L$ is any line bundle on $X$ of multidegree $(1,1,-1)$;
the singular point of $\ov{J}_2$ is the sheaf $\I_2:=\I_Z\otimes M$ where $\I_Z$ is the ideal sheaf of the length $2$ subscheme defined by
$(x, y^2)$ and $M$  is any line bundle on $X$ of multidegree $(1,1,0)$.

Moreover, denote by $\wt{X}$ the seminormalization of $X$ (explicitly $\wt{X}$ can be realized as the union of three lines in projective space meeting in one point
with linearly independent directions) and $\pi:\wt{X}\to X$ is the natural map. Using Table 2 of \cite{Kas3} (where the unique singularity of $\wt{X}$ is called $\wt{D}_4$
and the unique singular point  of $X$ is classically called $D_4$), it can be shown that, up to the tensorization with a suitable line bundle on $X$,  $\I_2$ is the pushforward
of the trivial line bundle on $\wt{X}$ while $\I_1$ is the pushforward  of the canonical sheaf on $\wt{X}$, which is not a line bundle since $\wt{X}$ is not Gorenstein
(see Example \ref{Ex:non-dagger}).
\end{remark}

\begin{remark}
Simpson (possibly coarse) compactified Jacobians of Kodaira curves have been studied by A. C. L{\'o}pez Mart{\'{\i}}n in \cite[Sec. 5]{LM}, see also \cite{LM2}, \cite{LRST}.
\end{remark}

\vspace{0,3cm}

\noindent {\bf Acknowledgments.} We are extremely grateful to E. Esteves for several useful conversations on fine compactified Jacobians and for generously answering several mathematical questions.
We thank E. Sernesi for useful conversations on deformation theory of curves with locally planar singularities and E. Markman for asking about the embedded dimension of compactified Jacobians. 
We are very grateful to the referees for their very careful reading of the paper and for the many suggestions and questions that helped in improving a lot the presentation.

This project started while the first author was visiting the Mathematics Department of the University of Roma ``Tor Vergata'' funded by the ``Michele Cuozzo'' 2009 award. She wishes to express her gratitude to Michele Cuozzo's family and to the Department for this great opportunity.

The three authors were partially supported by the FCT (Portugal) project \textit{Geometria de espa\c cos de moduli de curvas e variedades abelianas} (EXPL/MAT-GEO/1168/2013).
M. Melo and F. Viviani were partially supported by the FCT projects \textit{Espa\c cos de Moduli em Geometria Alg\'ebrica} (PTDC/MAT/ 111332/2009) and \textit{Comunidade Portuguesa de Geometria Alg\'ebrica} (PTDC/MAT-GEO/0675/2012).  A. Rapagnetta and F. Viviani were partially supported by the MIUR project  \textit{Spazi di moduli e applicazioni} (FIRB 2012).
F. Viviani was partially supported by CMUC - Centro de Matem\'atica da Universidade de Coimbra.

\end{document}